\documentclass[11pt]{article}

\usepackage{amsmath, amsfonts, amsthm, amssymb}
\usepackage{authblk, color, bm, graphicx}

\usepackage[margin = 1.5cm, font = small, labelfont = bf, labelsep = period]{caption}
\usepackage{xspace}
\usepackage{setspace}
\usepackage{verbatim}
\usepackage[utf8]{inputenc}

\usepackage{graphicx, psfrag}
\usepackage[square, numbers]{natbib}
\usepackage{color}
\usepackage{enumerate}
\usepackage{subfigure}
\usepackage{multirow}
\usepackage{booktabs}
\usepackage{verbatim}
\usepackage[colorlinks=true]{hyperref}
\usepackage{xcolor}
\hypersetup{colorlinks, linkcolor={blue!80!black}, citecolor={blue!80!black}, urlcolor={green!50!black}}

\newcommand*{\colorboxed}{}
\def\colorboxed#1#{%
	\colorboxedAux{#1}%
}
\newcommand*{\colorboxedAux}[3]{
	\begingroup
	\colorlet{cb@saved}{.}%
	\color#1{#2}%
	\boxed{%
		\color{cb@saved}%
		#3%
	}%
	\endgroup
}

\usepackage{fancyhdr}
\usepackage{mathabx}

\usepackage{tikz}
\usetikzlibrary{calc,3d}

\usepackage{fullpage}[2cm]
\newtheorem{remark}{Remark}

\newtheorem{proposition}{Proposition}
\newtheorem{lemma}{Lemma}
\newtheorem{theorem}{Theorem}

\newtheorem{corollary}{Corollary}

\newcommand{\mb}{\mathbb}
\newcommand{\mc}{\mathcal}

\newcommand{\bs}{\boldsymbol}

\DeclareMathOperator*{\argmin}{arg\,min}

\newcommand{\ol}[1]{\overline{#1}}
\newcommand{\wt}[1]{\widetilde{#1}}
\newcommand{\wh}[1]{\widehat{#1}}

\usepackage[capitalize,nameinlink,compress]{cleveref}
\crefname{appendix}{Appendix}{Appendices} 
\crefname{figure*}{figure*}{Figures} 
\crefname{equation}{}{} 

\crefname{claim}{Claim}{Claims}
\crefname{example}{Example}{Examples}
\crefname{remark}{Remark}{Remarks}
\crefname{assumption}{Assumption}{Assumptions}
\crefname{fact}{Fact}{Facts}
\crefname{proposition}{Proposition}{Propositions}
\crefname{corollary}{Corollary}{Corollaries}
\crefname{lemma}{Lemma}{Lemmas}
\crefname{theorem}{Theorem}{Theorems}
\crefname{definition}{Definition}{Definitions}

\crefformat{problemC}{#2Problem~\cref{prob::#1}#3}
\crefmultiformat{problemC}{Problems~#2\cref{prob::#1}#3}{ and #2\cref{prob::#1}#3}{, #2\cref{prob::#1}#3}{ and #2\cref{prob::#1}#3}

\newcommand{\Ni}{{\scriptscriptstyle \mathcal N_i \scriptstyle}}
\newcommand{\oNi}{{\scriptscriptstyle \overline{\mathcal N}_i \scriptstyle}}

\newcommand{\oNj}{{\scriptscriptstyle \overline{\mathcal N}_j \scriptstyle}}
\newcommand{\Mu}{{\scriptscriptstyle \mathcal M \scriptstyle}}
\newcommand{\AT}{{\scriptscriptstyle \mc{AT} \scriptstyle}}

\newcommand{\FS}{{\scriptscriptstyle \mc{FS} \scriptstyle}}

\newcommand{\PSCs}{Problem~\eqref{Semi-Centralized}\xspace}

\newcommand{\PDs}{Problem~\eqref{Decentralized}\xspace}

\allowdisplaybreaks
\graphicspath{{./drawing/}{./figs/}}

\title{A Robust Optimization Approach to Network Control\\ Using Local Information Exchange}

\author[1]{Georgios Darivianakis}
\author[2]{Angelos Georghiou}
\author[3]{Soroosh Shafiee}
\author[4]{John Lygeros}

\affil[1]{\small \textit{ENEXAN, Zurich, Switzerland}}
\affil[2]{\small \textit{Department of Business and Public Administration, University of Cyprus, Cyprus}}
\affil[3]{\small \textit{Operations Research and Information Engineering, Cornell University, United States}}
\affil[4]{\small \textit{Automatic Control Laboratory, ETH Zurich, Switzerland}}

\date{}
\begin{document}
\maketitle
\begin{abstract}
	Designing policies for a network of agents is typically  done  by  formulating  an  optimization  problem  where  each agent has access to state measurements of all the other agents in  the  network. Such  policy  designs  with  \emph{centralized  information  exchange}  result  in  optimization  problems  that are  typically  hard  to  solve,  require  establishing  substantial communication links, and do not promote privacy since all information  is  shared  among  the  agents.  	Designing  policies based on  arbitrary communication structures can lead to non-convex optimization problems which are typically NP-hard. In  this  work,  we propose an optimization framework for decentralized policy designs.  In contrast to the centralized information exchange, our approach requires only \emph{local communication exchange} among the neighboring  agents  matching the physical coupling of the network. Thus,  each agent only requires information from its direct neighbors, minimizing the need for excessive communication and promoting privacy amongst the agents.
	Using  robust  optimization  techniques, we formulate a convex optimization problem with a loosely coupled structure  that  can  be  solved  efficiently. We numerically demonstrate the  efficacy of the proposed approach in energy management and supply chain applications. We show that the proposed approach  leads to solutions that closely approximate those obtained by the centralized formulation only at a fraction of the computational effort.
\end{abstract}
\textbf{Keywords:} robust optimization; network control;  decentralized policy design; local information; state forecast sets; decision rules. 

\section{Introduction}
Controlling physical networks of interconnected systems remains an active field of research due to its high impact on real-world applications, e.g., regulation of power networks \citep{Venkat2008}, energy management of building districts, \citep{Darivianakis2017b} and supply chains \citep{Tsay1999}. For large-scale systems, designing and deploying  policies that have a centralized communication structure can be challenging. In the design phase, computational limitations can restrict the size of the problem that can be tackled, while in the deployment phase, the centralized nature of the policy requires excessive communication which does not promote privacy between the interconnected systems. In such cases, it is desirable to design policies with local information exchange that  ideally rely on local computational resources to compute their policies.

Synthesizing policies based on an arbitrary communication structure can lead to non-convex infinite-dimensional optimization problems which are typically NP-hard \citep{Tsitsiklis1985}. For that reason, several studies have been devoted to identifying specific communication structures under which the designing policies can be cast as a convex problem \citep{Lin2011,Mahajan2012}. For instance, if the communication network admits a partially nested structure \citep{Ho1972}, then affine policies are known to be optimal for decentralized linear systems with quadratic costs and additive Gaussian noise \citep{Ho1972,Rantzer2006a,Rantzer2006b}. Similar results exist for communication structures that are spatially invariant \citep{Bamieh2002,Bamieh2002,Motee2008}, including delays in information sharing \citep{Lamperski2015,Nayyar2010,Nayyar2013}.

Recent advances shifted research interest in identifying information communication structures that allow for the optimal distributed controllers synthesis problem to be formulated as a convex optimization problem \citep{Bamieh2005,DeCastro2002,Dvijotham2013,Matni2013,Qi2004}. These structures usually possess properties such as quadratic invariance \citep{Rotkowitz2005,Swigart2014} and funnel causality \citep{Bamieh2005} which essentially eliminates the incentive of signaling among the interconnected systems. For general network structures, the usual practice is to resort to linear matrix inequality relaxations \citep{Langbort2004,Zecevic2010} or semidefinite programming relaxations \citep{Lavaei2011,Fazelnia2016} to obtain a suboptimal design with performance guarantees.  

A downside of the aforementioned design approaches is the inability  to cope with state and input constraints in the systems. Optimization based approaches, often referred to as Model Predictive Control, are well-suited for the control of constrained systems \citep{Mayne2000}. In this setting, control designs are usually categorized into cooperative or non-cooperative \citep{Scattolini2009}  referring to the level of information exchange amongst the interconnected systems. As before, cooperative approaches require substantial communication infrastructure and computation resources in the design phase since a system-wide optimization problem is formulated and solved \citep{Venkat2008,Stewart2010,Giselsson2013}. On the other hand, non-cooperative approaches, though computationally simple and effective in practice, can be conservative in presence of strong coupling \citep{Richards2004,Keviczky2006,Trodden2010}.  In addition, non-cooperative schemes typically require a centralized offline design phase, thus suffering from similar complexity and privacy concerns as the centralized designs.

There is a stream of literature that tries to address these issues by developing  schemes that rely on local computational resources and information structures \citep{Camponogara2002,Dunbar2007,Farina2012,Lucia2015,Trodden2017}. This is commonly achieved by each system individually considering the worst-case effect of its neighbors as a bounded exogenous uncertainty to its own system. Nevertheless, this can lead to conservative designs as the sets of bounded exogenous uncertainties are \emph{calculated off-line}, thus disregarding the possibility of adapting their size based on the dynamical evolution of the system.

Recent advances in robust optimization techniques deal with exogenous uncertainties in a computationally efficient way,  allowing to address both static and multistage problems \citep{Bertsimas2011,Delage2015,Gorissen2015}.
In contrast, problems with endogenous uncertainties usually referred to as problems with decision-dependent uncertainty sets,  are typically computationally intractable \citep{Nohadani2018}. However, by exploiting structural characteristics of the problem such as right-hand side uncertainty in the constraints, the  work of \citep{Jaillet2016,Zhang2017,Bitlisliouglu2017} proposed approximations that restrict the space of admissible uncertainty sets to those that exhibit an affine dependence on  fixed sets. By doing so, the resulting approximate problems can be equivalently cast as robust optimization problems with exogenous uncertainties, and hence they can be efficiently solved using existing methods.
In this paper, we leverage similar techniques  to design control policies with local information structure, by \emph{designing on-line} the bounded sets of exogenous uncertainties which model the worst-case effect of  the neighbors for a given system.

In this paper, we address communication and computation issues of policy designs for physical networks of interconnected systems.  The contributions of this paper are:
\begin{itemize}
	\item[$\diamond$] We propose a new modeling paradigm for designing policies with local communication exchanges which exactly match the physical coupling of the network. In contrast to centralized and partially nested models where 
	the states are explicitly communicated amongst the neighboring agents as  functions of the uncertain parameters, the proposed approach communicates compact sets which we refer to as the \emph{state forecast set}, implicitly encompassing the states of neighboring systems. As the agents are only coupled through the forecast sets, the proposed structure addresses privacy concerns in the information exchange and allows for a natural interpretation of the resulting policy. The framework generalizes the proof-of-concept conference paper \citep{Darivianakis2017a} in which a simple version of the idea was applied for the efficient energy management of building districts.

	\item[$\diamond$] We show that the proposed paradigm directly relates to centralized and partially nested information exchange policy designs. In particular, we show that the optimal local communication policy constitutes a conservative approximation to the partially nested information design problem, while the optimal partially nested information policy constitutes a conservative approximation to the centralized information design problem.  In addition, we identify network structures for which the optimal values of these problems coincide. Hence, the proposed approach fits well within the established modeling paradigms.
	
	\item[$\diamond$] We propose a tractable approximation that restricts the functional form of the state forecast sets to an affine transformation of a fixed set, in a similar spirit as \citep{Jaillet2016,Zhang2017,Bitlisliouglu2017}. Under this restriction, we show that the problem can be cast as a multistage robust optimization problem,  which is a well studied class of problems and multiple solution methods exist to either approximate or solve the problem explicitly. By construction, the resulting problem has a  decoupled structure making the problem  highly scalable with respect to the number of interconnected systems in the network. The efficacy of the proposed approach is demonstrated in three numerical experiments where we study the quality of the solution with respect to the state forecast sets approximation, the scalability properties of the approach, as well as how the problem can be solved in an almost decentralized manner using the alternating direction method of multipliers algorithm.
	
\end{itemize} 

The remainder of this paper is organized as follows. Section~\ref{sec::ProbForm} provides the problem formulation and briefly reviews the centralized and partially nested information exchange policies. Section~\ref{sec::DecCont} presents the new modeling paradigm and discusses the relationship between the centralized and partially nested policy designs. In Section~\ref{sec::SolMethod} we derive tractable approximations and show how the problem can be cast  as a multistage robust optimization problem. In Section \ref{sec::Numerics} we present two numerical examples: (i) an illustrative example that showcases the proposed method and discusses the effect of the state forecast sets approximation, and (ii) a cooperative energy management system. All proofs are found in the Appendix~\ref{app::proof}. We also provide an additional numerical examples in Appendix \ref{supplychain} that examines a contract design problem for a supply chain.

\textbf{Notation:} 
The calligraphic letters $\mathcal M, \mathcal P, \mathcal T$ are reserved for finite index sets with cardinalities $M, P, T$, that is, $\mathcal M = \{1, \dots, M \}$ etc. The subscript $+$ in $\mathcal T_+$ indicates that the index set $\mc T$ additionally includes $T+1$, that is, $\mathcal T_+ = \mathcal T \bigcup \{T+1\}$. Concatenated vectors are represented in boldface. Dimensions of matrices and concatenated vectors are assumed clear from the context. For given vectors $ v_{i} \in \mb R^{k_i} $ with $ k_i \in \mb N $, $ i \in \mc M$, we define $ \bm v_\Mu = [v_{i}]_{i\in \mc M} = [v_{1}^\top \ldots v_{M}^\top]^\top \in \mb R^{k} $ with $ k = \sum_{i=1}^{M}k_i $ as their vector concatenation. Given time dependent vectors $ \nu_{i,t} \in \mb R^{\ell_i} $ with $ i \in \mc M $, $ t \in \mc T $ and $ \ell_i \in \mb N $, we define $ \bm \nu_{\Mu,t} = [\nu_{i,t}]_{i \in \mc M} $ as the concatenated vector at time $t$, $ \bm \nu_i^t = [\nu_{i,1}^\top \ldots \nu_{i,t}^\top]^\top $ as the history of the $ i $-th vector up to time $ t $, and $ \bm \nu_\Mu^t = [\bm \nu_i^t]_{i \in \mc M} $ as the history of the concatenated vector up to time $ t $. We denote by $\text{ext}(\Xi)$ the set of extreme points of set $\Xi$. An extended notation section summarizing the  major notation can be found in Appendix~\ref{appedix::notation}.

\section{Problem formulation} \label{sec::ProbForm}

We consider a physical network comprising $ M $ interconnected systems, henceforth referred to as agents. We assume that the agents are coupled  through the dynamics. We describe these interactions through a directed graph in which an arc connecting agent $ j $ to agent $ i $, with $ i, j \in \mc M$, indicates that the states of the $ j $-th agent affect the dynamics of the $ i $-th agent. We refer agent $ j $ as the preceding neighbor to agent $ i $, henceforth neighbor, and we define the set $ \mc N_i \subset \mc M $ to include all the neighbors of the $ i $-th agent. Figure \ref{fig::physNet} illustrates a system of $ M = 5 $ agents where the neighbors of agent 3 are given by $ \mc N_3 = \{2, 5\}$. In the sequel, we refer to the {physical network} depicted in Figure \ref{fig::physNet} as the ``working example'' and  use it to streamline the presentation of key ideas in the paper. 
\begin{figure*}[ht]
	\centering
	\includegraphics[width = 0.4\textwidth]{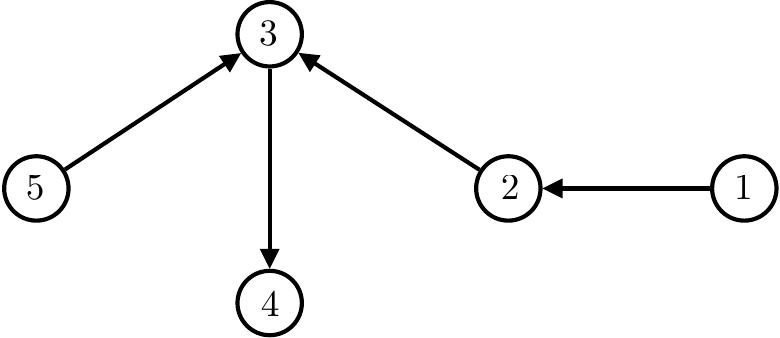}
	\caption{\textbf{Working example}: \emph{Physical coupling graph of 5 agents}. Solid line arrows represented the direction of the interaction.}
	\label{fig::physNet}
\end{figure*}

\subsection{System dynamics, constraints, and objective function} \label{Section::BasicInfo}

In this paper, we study finite horizon problems with $T$ stages. We use linear dynamics to model the state evolution of the agent $i$ at time instant $ t \in \mc T $, as 
\begin{equation} \label{eq::stateDynamics}
	x_{i,t+1} = A_{i,t} x_{i,t} + B_{i,t} \bm x_{\Ni,t} +D_{i,t} u_{i,t} + E_{i,t}\xi_{i,t}. 
\end{equation} 
Here $ x_{i,t} \in \mb R^{n_{x,i}} $ denotes the states, with the initial state $ x_{i,1} $ known. The interaction of agent $i$ and its neighbors is captured through the term $ B_{i,t} \bm x_{\Ni,t}$. Vector $ u_{i,t} \in \mb R^{n_{u,i}} $ models the inputs and $ \xi_{i,t} \in \mb R^{n_{\xi,i}} $ captures the exogenous uncertainties affecting the system dynamics. The time-varying system matrices $ A_{i,t} $, $ B_{i,t} $, $D_{i,t} $ and $ E_{i,t} $ are assumed known with proper dimensions and of full column rank. To economize on notation, we now compactly rewrite \eqref{eq::stateDynamics} as 
\begin{equation*}\label{eq::stateDynamicsCompact}
	\bm x_i = f_i(\bm x_{\Ni}, \bm u_i, \bm \xi_i) := A_i x_{i,1} + B_i \bm x_{\Ni} + D_i \bm u_i + E_i \bm \xi_i,
\end{equation*}
where $ \bm x_i := [x_{i,t}]_{t\in \mc T_+} $, $ \bm u_i := [u_{i,t}]_{t \in \mc T} $, $ \bm \xi_i := [\xi_{i,t}]_{t\in \mc T} $ and $ \bm x_{\Ni} := [\bm x_{\Ni,t}]_{t \in \mc T} $. The system matrices $ A_i $, $ B_i $, $ D_i $ and $ E_i $ used to define the function $ f_{i}(\cdot)  $ are directly constructed by the problem data given in \eqref{eq::stateDynamics} (see, e.g., \citep{Goulart2006} for such a derivation). 
The  $ i $-th agent is subject to linear constraints
\begin{equation}\label{eq::InequalitiesCompact}
	\mc O_i = \big\{(\bm x_{i}, \bm u_{i}) \,:\, H_{x,i}\bs{x}_i + H_{u,i} \bm u_i \leq h_i\big\},
\end{equation}
where the matrices $ H_{x,i} $, $ H_{u,i} $ and $ h_i $ are assumed known and of proper dimensions. Note that this compact constraint formulation allows the consideration of time-varying linear operational constraints with time-stage coupling. In addition,  constraints involving neighboring states or exogenous uncertainties can always be included by appropriately extending the state space of the $ i $-th system. The objective associated with $ i $-th agent is given by
\begin{equation}\label{eq::objFnc}
	J_i(\bm x_i,\bm u_i) =  \sum_{t=1}^T \left(\| Q_i  x_{i,t} \|_q + \| R_i  u_{i,t} \|_q \right), 
\end{equation}
where $ q \in \{\infty, 1\} $ allows for different objective formulation. The penalization matrices $Q_i$, $R_i$ are assumed known and of proper dimensions.

In the following, we assume that the exogenous uncertainties affecting agent $ i $ reside in the nonempty, convex and compact polyhedral uncertainty set $\Xi_i = \{\bm \xi_i :  W\bm \xi_i \geq  w\}$ where matrix $W$ and  vector $w$ are known and of proper dimensions. We will be making the simplifying assumption that the joint uncertainty set of all agents in the system has a decoupled structure, i.e.,
$ \bm \xi_\Mu \in \Xi_\Mu = \bigtimes_{i \in \mc M} \Xi_i $, which  essentially precludes the existence of  coupled uncertainties amongst the agents. This assumption can be  relaxed at the expense of further case distinctions in what follows.

\subsection{Designing policies with centralized information exchange}\label{centralized}

A common assumption in designing  policies is to assume that at time $t$, each agent has access to the states from all the other agents in the network up to and including period $t$ \citep{Goulart2006,Hadjiyiannis2011}. We will refer to this communication as the \emph{centralized information exchange}, depicted in Figure \ref{fig::InfStr_CS} for the working example. In this context, we denote the \emph{causal state feedback} policies for agent $i$ at time $t\in\mc T$ as  \mbox{$\pi_{i,t}:\mb R^{n_x^t} \rightarrow \mb R^{n_{u,i}} $} where $ n_{x}^t = t \left(\sum_{j\in \mc M} n_{x,j} \right)$, such that its input at time $ t $ is given as $ u_{i,t} = \pi_{i,t}(\bm x_\Mu^t) $. We write $ \bm \pi_i(\bm x_\Mu)  = [\pi_{i,t}(\bm x_\Mu^t)]_{t\in\mc T} $ to denote the policy concatenation over the horizon.

\begin{figure*}[ht]
	\centering
	\includegraphics[width = 0.4\textwidth]{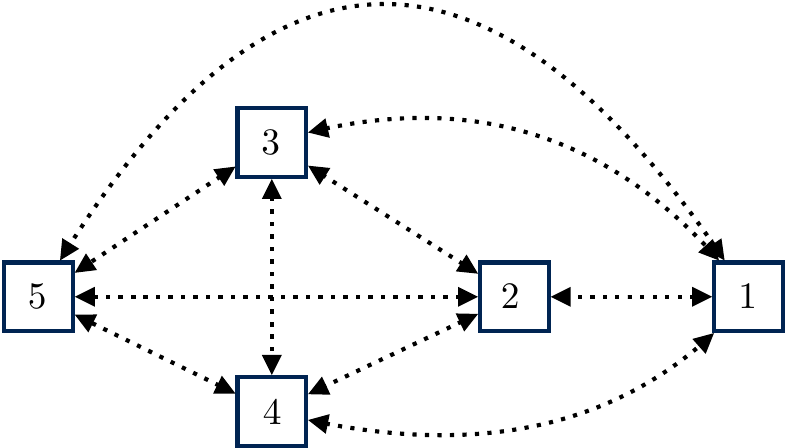}
	\caption{\textbf{Working example}: \emph{Centralized information exchange}. Dotted line arrows represent the  communication links between the agents, with $\mathcal{M} = \{1,2,3,4,5\}$.}
	\label{fig::InfStr_CS}
\end{figure*}

The optimization problem for designing policies with centralized information exchange which minimize the sum of worst-case individual objectives is formulated as
\begin{subequations}\label{BothCentralized}
	\begin{equation}\label{Centralized}
		\begin{array}{l}
			\text{~minimize~~} \displaystyle\sum\limits_{i = 1}^M \max\limits_{\bm \xi_{\mathcal M} \in \Xi_{\mathcal M}} J_i(\bm x_i,\bm u_i) \\
			\left.\begin{array}{@{}r@{~}l@{}}
				\text{subject to }
				& \bm u_i =\bm \pi_i(\bm x_\Mu)  := [\pi_{i,t}(\bm x_\Mu^t)]_{t\in\mc T}\\
				
				& \bm x_{i} = f_{i}\big(\bm x_{\Ni}, \bm u_{i}, \bm \xi_{i}\big)\\
				& (\bm x_{i}, \bm u_{i}) \in \mc O_i
			\end{array}\right \rbrace \forall \bm \xi_\Mu \in \Xi_\Mu,\; \forall i \in \mc M,
		\end{array}
	\end{equation}
	where the  optimization variables  are the policies $ \bm \pi_i(\cdot) $ for all $ i \in \mc M $. 
	As shown in \citep{Goulart2006,Hadjiyiannis2011}, the state feedback structure of the policies induces a non-convex feasible region. To deal with this, they propose the design of \emph{strictly causal uncertainty feedback policies} $ \Pi_{i,t}:\mb R^{n_{\xi}^t} \rightarrow \mb R^{n_{u,i}} $ where $ n_{\xi}^t = (t-1)\sum_{i\in \mc M} n_{\xi,i} $, such that the input at each time step is given by $ u_{i,t} = \Pi_{i,t}(\bm \xi_\Mu^{t-1}) $. We write $ \bm \Pi_i(\bm \xi_\Mu)  = [\Pi_{i,t}([\bm \xi_\Mu^{t-1}]_{t\in\mc T}) $ to denote the policy concatenation over the horizon.
	This formulation leads to the infinite dimensional linear optimization
	\begin{equation}\label{Centralizedb}
		\begin{array}{l}
			\text{~minimize~~} \displaystyle\sum\limits_{i = 1}^M \max\limits_{\bm \xi_{\mathcal M} \in \Xi_{\mathcal M}} J_i(\bm x_i,\bm u_i) \\
			\left.\begin{array}{@{}r@{~}l@{}}
				\text{subject to }
				& \bm u_i =\bm \Pi_i(\bm \xi_\Mu)  := [\Pi_{i,t}(\bm \xi_\Mu^{t-1})]_{t\in\mc T}\\
				& \bm x_{i} = f_{i}\big(\bm x_{\Ni}, \bm u_{i}, \bm \xi_{i}\big)\\
				& (\bm x_{i}, \bm u_{i}) \in \mc O_i
			\end{array}\right \rbrace \forall \bm \xi_\Mu \in \Xi_\Mu,\; \forall i \in \mc M.
		\end{array}
	\end{equation}
\end{subequations}
Using the fact that matrices $A_{i,t}$, $B_{i,t}$ and $E_{i,t}$ are full column rank, there is a one-to-one relationship between state and uncertainty feedback policies in terms of feasibility and optimality. Restricting the admissible policies to have an affine structure and reformulating using robust optimization techniques reduces the problem to a finite-dimensional linear optimization problem which can be solved with off-the-shelf optimization solvers. 
Furthermore, due to the one-to-one relationship between state and uncertainty feedback policies, there is also a unique mapping that translates the resulting affine uncertainty  policies to an equivalent affine state feedback policy which  allows  being implemented locally by each agent using the centralized communication network. 

Although theoretically appealing, policies based on centralized information exchange are hard to design and implemented in practice for large-scale systems. This is partially the case since for large networks $ (i) $ solving the linear optimization problem resulting from the affine policy approximation  can be computationally challenging due to its monolithic structure; while $ (ii) $ the excessive communication and the centralized physical network required to allow each agent to evaluate its policy, does not promote privacy since the exact policy/constraints of individual agents are eventually revealed to the rest of the network. We will demonstrate the former through numerical experiments in Section \ref{sec::Numerics}.  

\subsection{Designing policies with partially nested information exchange}\label{partiallynested}

In an attempt to address the computational and privacy issues, researchers have proposed a number of policy designs that consider only partial communication among the agents. For an arbitrary information exchange network the design phase typically results in a non-convex problem that is computationally intractable. A notable exception is the work of \citep{Lin2016} which assumes a \emph{partially nested information exchange}, leading to  convex formulations. Roughly speaking this communication exchange implies that  agent $ i $ has access to information coming from all of its \emph{precedent agents} in a non-anticipative manner. In this setting, agent $ j $ is named a\emph{ precedent to agent $ i $}, if the input of agent  $ j $ at time $ t' $ can affect the local information available to agent $ i $ at some time $ t > t' $ in the future \cite[Definition 1]{Lin2016}. The partially nested information exchange associated with the working example is depicted in Figure~\ref{fig::InfStr_NS}, where we see that although agent 1 is not a physical neighbor of agents 3 and 4, its actions can affect the future states of both these agents, thus it is a precedent agent.
\begin{figure*}[ht]
	\centering
	\includegraphics[width = 0.4\textwidth]{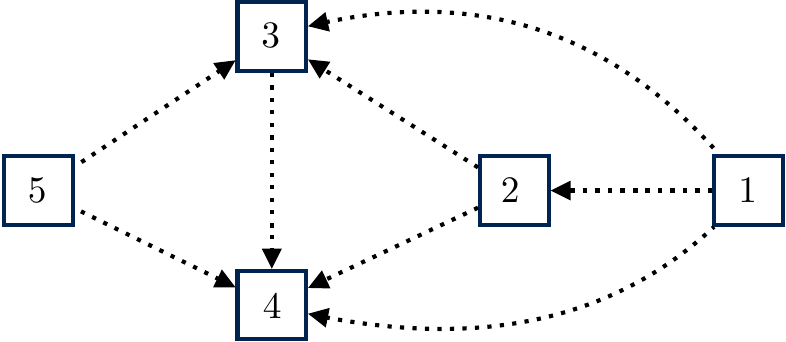}
	\caption{\textbf{Working example}: \emph{Partially nested information exchange}. Dotted line arrows represent the established communication links between local controllers, with $\ol{\mc N}_1 = \{1\}$, $\ol{\mc N}_2 = \{1,2\}$, $\ol{\mc N}_3 = \{1,2,3,5\}$, $\ol{\mc N}_4 = \{1,2,3,4,5\}$ and $\ol{\mc N}_5 = \{5\}$.}
	\label{fig::InfStr_NS}
\end{figure*}

In a partially nested communication, the policy is designed as follows. We denote by $ \ol{\mc N}_i \subseteq \mc M $ the set that includes agent $ i $ and all its precedent agents. At time $t$,  the $ i $-th agent measures  its own states and the states of its precedent agents, denoted by  $ \bm x_{\oNi,t} = [x_{j,t}]_{j \in \ol{\mc N}_i} $. Using all measurements from stage 1 up to time $t$, it designs a \emph{causal partial state feedback} policy $ \phi_{i,t}:\mb R^{\ol{n}_{x,i}^t} \rightarrow \mb R^{n_{u,i}} $, where $ \ol{n}_{x,i}^t = t \left(\sum_{j\in \ol{\mc N}_i} n_x^j\right) $. The input is now given as $ u_{i,t} = \phi_{i,t}(\bm x_\oNi^t) $. We write $ \bm \phi_i(\bm x_\oNi)  = [\phi_{i,t}(\bm x_\oNi^t)]_{t\in\mc T} $ to denote the policy concatenation over the time horizon.

The optimization problem to design   policies with partially nested information exchange that minimize the sum of worst-case individual objectives is formulated as
\begin{subequations}\label{BothPartiallyNested}
	\begin{equation}\label{Semi-Centralized}
		\begin{array}{l}
			\text{~minimize~~} \displaystyle\sum\limits_{i = 1}^M \max\limits_{\bm \xi_\oNi \in \Xi_\oNi} J_i(\bm x_i,\bm u_i) \\
			\left.\begin{array}{@{}r@{~}l@{}}
				\text{subject to } 
				&\bm u_i =\bm \phi_i(\bm x_\oNi)  := [\phi_{i,t}(\bm x_\oNi^t)]_{t\in\mc T}\\
				& \bm x_{i} = f_{i}\big(\bm x_{\Ni}, \bm u_{i}, \bm \xi_{i}\big)\\
				& (\bm x_{i}, \bm u_{i}) \in \mc O_i
			\end{array}\right \rbrace \forall \bm \xi_\oNi \in \Xi_{\oNi},\;
			\forall i \in \mc M,
		\end{array}
	\end{equation}
	where $ \Xi_{\oNi} = \bigtimes_{j \in \ol{\mc N}_i} \Xi_j $ and  the  optimization variables  are the policies $ \bm \phi_{i}(\cdot) $ for all $ i \in \mc M $. \PSCs is typically non-convex due to the state feedback structure of the policies. Similar to the centralized case, \citep{Lin2016} proposes the use of \emph{strictly causal partial nested uncertainty feedback policies} $ \Phi_{i,t}:\mb R^{\ol{n}_{\xi,i}^t} \rightarrow \mb R^{n_{u,i}} $ where $ \ol{n}_{\xi,i}^t = (t-1) \left( \sum_{j\in \ol{\mc N}_i} n_\xi^j\right) $, such that the input at each time step is given by $ u_{i,t} = \Phi_{i,t}(\bm \xi_\oNi^{t-1}) $.
	We write $ \bm \Phi_i(\bm \xi_{\oNi})  = [\Phi_{i,t}(\bm \xi_\oNi^{t-1})]_{t\in\mc T} $ to denote the policy concatenation, 
	leading to the infinite dimensional linear optimization 
	\begin{equation}\label{Semi-Centralizedb}
		\begin{array}{l}
			\text{~minimize~~} \displaystyle\sum\limits_{i = 1}^M \max\limits_{\bm \xi_\oNi \in \Xi_\oNi} J_i(\bm x_i,\bm u_i) \\
			\left.\begin{array}{@{}r@{~}l@{}}
				\text{subject to }
				& \bm u_i =\bm \Phi_i(\bm \xi_{\oNi})  := [\Phi_{i,t}(\bm \xi_\oNi^{t-1})]_{t\in\mc T} \\
				& \bm x_{i} = f_{i}\big(\bm x_{\Ni}, \bm u_{i}, \bm \xi_{i}\big)\\
				& (\bm x_{i}, \bm u_{i}) \in \mc O_i
			\end{array}\right \rbrace \forall \bm \xi_\oNi \in \Xi_{\oNi},\;
			\forall i \in \mc M.
		\end{array}
	\end{equation}
\end{subequations}
If these uncertainty feedback policies are restricted to admit an affine structure and using robust optimization to reformulate the resulting semi-infinite constraints, then Problem~\eqref{Semi-Centralizedb} becomes a finite-dimensional linear optimization problem. As with the centralized case, there is a one-to-one relationship between the state and uncertainty feedback policies, both for the infinite-dimensional and affine restriction. This  allows the agents in the network to evaluate their policies based on the established  partially nested communication.

The following theorem establishes the connection between centralized and partially nested information design problems and will be used in the following section  to demonstrate the relationship between the proposed approach to the centralized and partially nested policy structures.
\begin{theorem}\label{thm::1}
	Problem~\eqref{Semi-Centralizedb} is a conservative approximation of Problem~\eqref{Centralizedb} in the following sense: every feasible solution of Problem~\eqref{Semi-Centralizedb} is feasible in Problem~\eqref{Centralizedb}, and the optimal value of Problem~\eqref{Semi-Centralizedb} is  larger or equal to the optimal value of Problem~\eqref{Centralizedb}.
\end{theorem}

Partially nested information exchange slightly reduces the communication requirements compared to the centralized problem (see Figures~\ref{fig::InfStr_CS} and \ref{fig::InfStr_NS} of the working example). This has a positive impact on the solution time needed to design affine feedback policies, however, the resulting linear program inherits in large part the monolithic structure of the centralized problem due to the absence of a  non-sparse structure. Most importantly, even in the simple model of our working example, the partially nested communication requires three additional links compared  to the physical links. In strongly connected physical networks, i.e.,  in networks where there is a directed path from any node to every other node in the network, the minimum number of communications links needed to ensure partially nested information coincides with the centralized information exchange, see the example presented in Section~\ref{sec::energy} with Figures~\ref{fig::energyhub} and \ref{fig::SMD_Network}.
Therefore, the synthesis of policies with partially nested information  inherits, to a large extent, the drawbacks of the centralized problem, both from a computational and privacy standpoint.

\section{Designing  policies with local information exchange} \label{sec::DecCont}
In this paper, we propose a decentralized policy structure that relies on local information exchange among the agents. The proposed policy aims to address both the computational and privacy concerns discussed so far. While in the previous section the information flow had to be sufficiently complex to capture a partially nested structure, in this section we will assume that the information flow can be as simple as the physical network, as this is illustrated in Figure \ref{fig::InfStr_DS} for the working example. 
\begin{figure*}[ht]
	\centering
	\includegraphics[width = 0.4\textwidth]{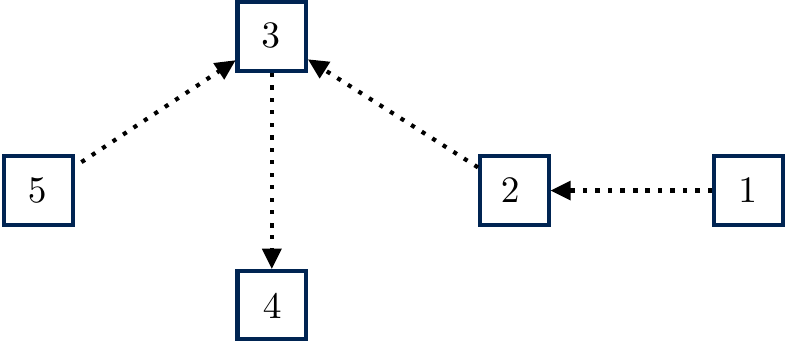}
	\caption{\textbf{Working example}: \emph{Local information exchange}. Dotted line arrows represent the established communication links between agents, $\mc N_1 = \emptyset$, $\mc N_2 = \{1\}$, $\mc N_3 = \{2,5\}$, $\mc N_4 = \{3\}$ and $\mc N_5 = \emptyset$}
	\label{fig::InfStr_DS}
\end{figure*}

In contrast to the previous section where agent  $ j\in\mc{N}_i $ communicates to agent $i$ explicitly 	communicates its states which are functions of the uncertain parameters, in the proposed framework agent  $ j$  communicates a compact set $ \mc X_j \subseteq \mb R^{N_{x,i}}$ where $N_{x,i} = \sum_{t\in \mathcal{T}_+}n_{x,i}$, henceforth referred as the \emph{state forecast set}, that contains possible evolution of its states, i.e., $\bm x_j \in \mc X_j$. In this framework, the shape of $ \mc X_j $ is a decision quantity for  agent $ j $. Upon receiving these state forecast sets from all of its neighbors, agent $ i $ treats these states as uncertain quantities that affect its dynamics in a similar way as exogenous uncertainties. To emphasize this, we denote by $ \zeta_{j,t}\in\mathbb{R}^{n_{x,j}} $ the \emph{belief} of agent $ i $ about the states of agent $j$ at time $t$, such that $\bs\zeta_j=[\zeta_{j,t}]_{t\in\mc{T}_+} \in \mc{X}_j$. In this context, the policy of agent $i$ at time $t$ is based on the information from its own states $x_{i,t}$ and on the belief states $ \bm \zeta_{\Ni,t} $, from stage 1 up to $t$. This leads us to design \emph{causal local state/uncertainty feedback policies} $ \psi_{i,t}:\mb R^{\hat n_{x,i}^t} \rightarrow \mb R^{n_u^i} $ where $ \hat {n}_{x,i}^t = t\left(n_{x,i} + \sum_{j\in \mc N_i} n_{x,j}\right) $, such that the  input of agent $i$ at time $ t $ is given by $ u_{i,t} = \psi_{i,t}(\bm x_i^t, \bm \zeta_\Ni^t) $. We denote by $ \bm {\psi}_{i}(\bm x_i, \bm \zeta_{\Ni}) = [\psi_{i,t}(\bm x_i^t, \bm \zeta_\Ni^t)]_{t \in \mc T} $ the policy concatenation over the time horizon.

In this decentralized setting, the robust optimization problem to design policies with local information exchange which minimize the sum of worst-case individual objectives is formulated as
\begin{subequations}
	\begin{equation}\label{Decentralized}
		\begin{array}{l}
			\text{~minimize~~} \displaystyle\sum\limits_{i = 1}^M \max\limits_{\bm \xi_i \in \Xi_i, \bm \zeta_{\Ni} \in \mc X_{\Ni}} J_i(\bm x_i,\bm u_i) \\
			\left.\begin{array}{@{}r@{~}l@{}}
				\text{subject to }
				& \bm u_i = \bm {\psi}_{i}(\bm x_i, \bm \zeta_{\Ni}) := [\psi_{i,t}(\bm x_i^t, \bm \zeta_\Ni^t)]_{t \in \mc T}\\
				& \bm x_{i} = f_{i}\big(\bm \zeta_{\Ni}, \bm u_{i}, \bm \xi_{i}\big)\\
				& (\bm x_{i}, \bm u_{i}) \in \mc O_i \\
				& \bm x_i \in \mc X_i, \; \mc X_i \in \mc{F}(\mc S_i)
			\end{array}\right \rbrace \begin{array}{@{}l}
				\forall \bm \zeta_{\Ni} \in \mc X_{\Ni}\\
				\forall \bm \xi_i \in \Xi_i
			\end{array}\;\forall i \in \mc M,
		\end{array}
	\end{equation}
	where $ \mc X_{\Ni} = \bigtimes_{j \in \mc N_i} \mc X_j $, $\mc F(\mc S_i) $ denotes the field of sets generated by all the subsets of the power set of $ \mc S_i $, where $\mc S_i\subseteq\mb R^{N_{x,i}}$.
	The optimization variables  are the policies $ \bm {\psi}_{i}(\cdot) $ and sets $ \mc X_i $ for all $ i \in \mc M $.  Since agent $i$ treads the beliefs $\bm \zeta_{\Ni}$ as exogenous uncertainties,  its decisions are taken in view of the worst-case both with respect to  $\bm \xi_i \in \Xi_i$  and $\bm \zeta_{\Ni} \in \mc X_{\Ni}$. This is also reflected in the construction of  the objective function. Notice that agent $i$ is not directly affected by the uncertainties $\bm \xi_j \in \Xi_j$ of its neighbors. Rather, the effect of $\bm \xi_j \in \Xi_j$ of agent $j\in\mc N_i$ is been translated into  set $\mc X_j$, which in turn affects agent $i$.

	\PDs can be interpreted as a method for finding a compromise between the agents as this is represented through sets $\mc X_i$. If we focus attention on two agents, agents $i$ and $j$ with $j\in \mc{N}_i$, on the one hand,  agent $j$ will benefit the most if the set $\mc X_j$ is as large as possible. By doing so, set  $\mc X_j$  will not impose any additional constraints on its states, thus agent $j$ will individually achieve the lowest objective value contribution. On the other  hand, agent $i$ will benefit the most if it receives from agent $j$ the smallest possible set $\mc X_j$ (preferable $\mc X_j$ is a  singleton) which will reduce the uncertainty on its dynamics, and will have a positive effect in individually achieving the lowest objective value contribution. Since the objective of \PDs is to minimize the equally weighted sum of individual worst-case costs, the resulting policy/set pair finds the trade-off among the agents, achieving the lowest network-wide objective value, while being robustly  feasible with respect to all  $\bs{\xi}_\Mu \in \Xi_\Mu$.
	
	\PDs diminishes the privacy concerns in the following ways. First,  the local communication network ensures that the information exchange is the minimum among the agents as it is the same as the physical network. Second, focusing again on agents $i$ and $j$ with $j\in \mc{N}_i$,  agent $j$ does  not directly reveal its states which are functions of the uncertain parameters
	to agent $i$, which is the case in both the centralized and partially nested information exchange. 
	Rather,   the future state trajectories,  are ``masks" by  set $\mc X_j$, thus reducing exposure to  agent $i$ who might want to leverage the knowledge gained about the  constraints and dynamics of agent $j$.  If additionally, the set $\mc X_j$ has a simple structure, e.g., $\mc X_j$ is rectangular, agent $j$ will reveal  the bare minimum information to its neighbor.  This will be studied further in the next section. 
	
	The state feedback nature of policies 
	induces a non-convex  optimization problem. As in \citep{Hadjiyiannis2011,Lin2016},  we now  focus on purely uncertain feedback policies $ \Psi_{i,t}:\mb R^{\hat n_{\xi,i}^t} \rightarrow \mb R^{n_{u,i}} $ where $ \hat n_{\xi,i}^t = (t-1)n_{\xi,i} + t\sum_{j\in \mc N_i} n_{x,j} $, such that the input at time $t$ is $ u_{i,t} = \Psi_{i,t}(\bm \xi_i^{t-1},\bm \zeta_\Ni^{t}) $. We write $ \bm \Psi_i(\bm \xi_i, \bm \zeta_{\Ni})  = [\Psi_{i,t}(\bm \xi_i^{t-1},\bm \zeta_\Ni^{t})]_{t\in\mc T} $ to denote the policy concatenation over time. 
	Note that  the policy is ``strictly causal"  to the uncertain vector $\bm \xi_{i}^{t-1}$ which the policy is allowed to depend up to stage $t-1$, while the policy is ``causal" in state beliefs $ \bm \zeta_\Ni^{t} $ and is allowed to depend up to stage $t$.  The following theorem, establishes the equivalence between the proposed state/uncertainty and uncertainty feedback policies and the corresponding optimization problems.
	\begin{equation}\label{Decentralizedb}
		\begin{array}{l}
			\text{~minimize~~} \displaystyle\sum\limits_{i = 1}^M \max\limits_{\bm \xi_i \in \Xi_i, \bm \zeta_{\Ni} \in \mc X_{\Ni}} J_i(\bm x_i,\bm u_i) \\
			\left.\begin{array}{@{}r@{~}l@{}}
				\text{subject to }
				& \bm u_i = \bm \Psi_i(\bm \xi_i, \bm \zeta_{\Ni}) := [\Psi_{i,t}(\bm \xi_i^{t-1},\bm \zeta_\Ni^{t})]_{t\in\mc T} \\
				& \bm x_{i} = f_{i}\big(\bm \zeta_{\Ni}, \bm u_{i}, \bm \xi_{i}\big)\\
				& (\bm x_{i}, \bm u_{i}) \in \mc O_i \\
				& \bm x_i \in \mc X_i, \; \mc X_i \in \mc{F}(\mc S_i)
			\end{array}\right \rbrace \begin{array}{@{}l}
				\forall \bm \zeta_{\Ni} \in \mc X_{\Ni}\\
				\forall \bm \xi_i \in \Xi_i
			\end{array}\;\forall i \in \mc M,
		\end{array}
	\end{equation}
\end{subequations}
\begin{theorem}\label{thm::2}
	Problem~\eqref{Decentralized} is a equivalent to Problem~\eqref{Decentralizedb} in the following sense: Given a feasible state/uncertainty feedback policy $ \bm \psi_{i}(\cdot) $ for Problem~\eqref{Decentralized}, a feasible uncertainty feedback policy $ \bm \Psi_{i}(\cdot) $ for Problem~\eqref{Decentralizedb} can be constructed that achieves the same objective value, and vice versa. 
\end{theorem}

The key difference between the partially nested information \PSCs and the local information \PDs is that the synthesis phase of the latter requires that each agent communicates only with its direct neighbors rather than with all its precedent agents in the network (compare Figure \ref{fig::InfStr_NS} to Figure \ref{fig::InfStr_DS} for the working example). This minimum communication exchange is sufficient to address \PDs since the coupling among agents in the network only appears through sets $ \mc X_i $. This however introduces a level of conservativeness which is formalized in the following theorem. 

\begin{theorem}\label{thm::3}
	Problem~\eqref{Decentralizedb} is a conservative approximation of Problem~\eqref{Semi-Centralizedb} in the following sense: every feasible solution of Problem~\eqref{Decentralizedb} is feasible in Problem~\eqref{Semi-Centralizedb}, and the optimal value of Problem~\eqref{Decentralizedb} is equal or larger than the optimal value of Problem~\eqref{Semi-Centralizedb}.
\end{theorem}

Theorem~\ref{thm::3} indicates that for a general network topology, Problem~\eqref{Decentralizedb} is a conservative approximation of Problem~\eqref{Semi-Centralizedb}. Nonetheless, there are specific topologies where the optimal values of these two problems coincide. This occurs when the local and partially nested information networks are the same. Such a scenario is realized when the underlying physical network is a directed acyclic graph with depth 1. By definition, this implies that it is  a directed acyclic bipartite network, as illustrated in Figure~\ref{fig::bipartite}. The result is formalized in Proposition~\ref{prop::1}.

\begin{figure*}[ht]
	\centering
	\includegraphics[width = 0.2\textwidth]{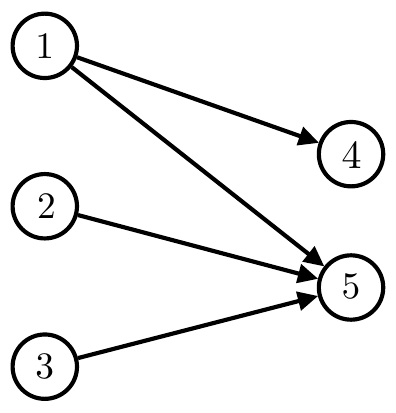}
	\caption{\textbf{Directed acyclic bipartite graph}: An example of a directed acyclic bipartite graph where $\ol{\mc N}_1 = \{1\}$, $\ol{\mc N}_2 = \{2\}$, $\ol{\mc N}_3 = \{3\}$, $\ol{\mc N}_4 = \{1,4\}$ and $\ol{\mc N}_5 = \{1,2,3,5\}$.}
	\label{fig::bipartite}
\end{figure*}

\begin{proposition}\label{prop::1}
	If the physical network is a  directed acyclic bipartite graph, 
	then Problem~\eqref{Decentralizedb}  coincides with problem  Problem~\eqref{Semi-Centralizedb} in the following sense: feasible solutions in Problem~\eqref{Decentralizedb} are mapped to feasible solutions in Problem~\eqref{Semi-Centralizedb}, and their optimal values are equal.
\end{proposition}

In contrast to Problems~\eqref{BothCentralized} and \eqref{BothPartiallyNested} in which  the structure of the problem is a multistage robust optimization \emph{with exogenous uncertainty}, Problem~\eqref{Decentralized} falls in the category of  multistage robust problem \emph{with endogenous uncertainty} also referred to as problems with \emph{decision-dependent uncertainty sets} \cite{Nohadani2017,Nohadani2018,Spacey2012}. Although sets $\mc X_{\Ni}$ do not model exogenous uncertainty, they affect the dynamics of the corresponding agent in an adversarial manner.
Problem~\eqref{Decentralizedb} is computationally intractable because $ (i) $ the optimization of the policies is performed over the infinite space of strictly causal functions; $ (ii) $ the optimization of the state forecast sets $ \mc X_i $ is performed over arbitrary sets, and $ (iii) $ the constraints must be satisfied robustly for every uncertain realization. In the next section, we propose tractable approximations making the problem amenable to off-the-shelf optimization algorithms.

We conclude this section by discussing how the optimal policies will be implemented in practice, highlighting the main differences between centralized, partially nested, and local information exchange problems.
Assume that we have solved Problems~\eqref{Centralized}, \eqref{Semi-Centralized} and \eqref{Decentralized}, and we have obtained the optimal policies $\pi_i^\star(\bm x_\mathcal{M})$, $\bm \phi_{i}^\star(\bm x_\oNi)$ and $\bm {\psi}_{i}^\star(\bm x_i, \bm \zeta_{\Ni})$ for all agents $i\in\mathcal{M}$, respectively. Consider first the centralized information problem.	
At time $t$, all states $\bm x_{\mathcal{M}}^t = (\bm x_{\mathcal{M},1},\ldots, \bm x_{\mathcal{M},t})$ have been realized and communicated to all agents $\mathcal{M}$ in the system. Consequently, the input $u_{i,t} = \pi^\star_{i,t}(\bm x_{\mathcal{M}}^t)$ can be evaluated deterministically by each agent. 
Next, we examine the partially nested information problem. At time $t$, all states $\bm x_{\mathcal{M}}^t = (x_{\mathcal{M},1},\ldots, x_{\mathcal{M},t})$ have been realized. In contrast to the centralized information problem, however, each agent~$i$ receives just the states $\bm x_{\overline{\mathcal{N}}_i}^t$ from its precedent agents $ \ol{\mc N}_i$. As such, each agent $i$ can deterministically calculate the  input $u_{i,t} = \pi_{i,t}^\star(\bm x_{\ol{\mc N}_i}^t)$.
Finally, we study the local information problem. Similar to the previous cases, at time $t$, all states $\bm x_{\mathcal{M}}^t = (\bm x_{\mathcal{M},1},\ldots, \bm x_{\mathcal{M},t})$ have been realized. In this case, however, each agent~$i$ receives just the states $\bm x_{{\mathcal{N}}_i}^t$ from its neighboring agents $ \ol{\mc N}_i$.
Using this information, each agent $i$ can deterministically calculate the input $u_{i,t} = \psi^\star_{i,t}(\bm x_{i}^{t},\bm x_{\mathcal{N}_i}^t)$.
Notice that agent $i$ substitutes the uncertain vector $\bm \zeta_{\mathcal{N}_i}^t$ with the realized vector $\bm x_{\mathcal{N}_i}^t$ when calculating the input $u_{i,t}$.
Recall that $\mathcal{N}_i\subseteq\ol{\mc N}_i\subseteq\mathcal{M}$ for all $i\in\mathcal{M}$. This indicates that each agent in the local information exchange problem requires less communication to implement and calculate its input compared to the centralized and partially nested information exchange problems. An alternative implementation procedure typically referred to as \emph{rolling horizon}, only evaluates the policy at $t=1$ and applies the input. After each agent observes the realization of its uncertain vector $ \xi_{i,1}$, the  states of the period $t=2$ are deterministically determined, i.e., $x_{i,2}$, and communicated to the neighbors. The optimization problem is subsequently re-optimized and the process is repeated. 
The rolling horizon is typically used when the Problems~\eqref{Centralized},  \eqref{Semi-Centralized} and \eqref{Decentralized} are solved approximately as re-optimization has been observed to improve the solution quality.

\section{State forecast set and policy approximation} \label{sec::SolMethod}

In this section, we discuss appropriate restrictions to Problem~\eqref{Decentralizedb}  that allow obtaining a computationally tractable approximation. We start with the state forecast sets. For agent $i\in\mathcal{M}$, a natural choice is to restrict  $ \mc X_i $ to admit  convex conic structure. However, for arbitrary conic structures the problem is NP-hard even if we restrict the policies to  be static (or \emph{open loop} as denoted in the control literature). This is the case as the optimization problem reduces the class of problems with decision-depend uncertainty as discussed in \cite[Theorem 3.2]{Nohadani2018}. 

A simpler approximation is  to restrict  $ \mc X_i$ to  sets that can be represented through an affine transformation of a fixed set, in a similar spirit as \citep{Bitlisliouglu2017,Jaillet2016,Zhang2017}.  Considering agent $i$.  For a given convex conic compact set  $ \mc S_i \subseteq \mathbb{R}^{N_{x,i}}  $,  $ \mc X_i$ is restricted to 
\begin{subequations}\label{SetApproximation}
	\begin{equation}\label{eq::SFSv1_nonconvex}
		\begin{array}{l}
			\mc X_i(Y_i,z_i) = \Big\{\bs{x}_i \in\mathbb{R}^{N_{x,i}} \,:\, \exists \bm s_{i} \in \mc S_i \text{ s.t. } \bm x_{i} = Y_i \bm s_{i} + z_{i}\Big\}
		\end{array}
	\end{equation}
	where the decision variables that define the shape of the state forecast set is the positive semi-definite matrix $Y_i\in\mathbb{S}_+^{N_{x,i}}$ and vector $ z_{i}\in\mathbb{R}^{N_{x,i}}$.  $\mc S_i$ is expressed through  given matrices $ G_{i} \in \mathbb{R}^{\ell_{i} \times N_{x,i}} $, vectors $ g_{i} \in \mathbb{R}^{\ell_k }$ and convex cones $ \mc K_{i} $, as follows
	\begin{equation}
		\begin{array}{l}
			\mc S_i = \Big\{\bs{s}_{i} \in \mb R^{N_{x,i}}\;:\;  G_{i}  \bm s_{i} \preceq_{\mc K_{i}}  g_{i} \Big\}.
		\end{array}
	\end{equation}
\end{subequations}
Matrix $Y_{i}$ allows scaling and rotation of set $\mathcal{S}_i$, while vector $z_i$ is responsible for translating $\mathcal{S}_i$. 

In the following, we denote by $ \mc F_{\AT}(\mc S_i) $ the field of bounded convex conic sets that can be represented through an affine transformation of a fixed set $ \mc S_i $ given in equation~\eqref{SetApproximation}. Restricting 
$\mc{F}(\mc S_i)$ to  $ \mc F_{\AT}(\mc S_i) $ in Problem~\eqref{Decentralizedb} gives rise to the following optimization problem:
\begin{subequations}
	\begin{equation}\label{DecentralizedXab_GeneralY}
		\begin{array}{@{}l}
			\text{~minimize~~} \displaystyle\sum_{i = 1}^M \max\limits_{\bm \xi_i \in \Xi_i, \bm \zeta_{\Ni} \in {\mc X}_{\Ni}(Y_i,z_i)} J_i(\bm x_i,\bm u_i) \\
			\left.\begin{array}{@{}r@{~~}l@{}}
				\text{subject to}
				&  \bm u_i = \bm \Psi_i(\bm \xi_i, \bm \zeta_{\Ni})  := [\Psi_{i,t}(\bm \xi_i^{t-1},\bm \zeta_\Ni^{t})]_{t\in\mc T} \\
				& Y_i\in\mathbb{S}_+^{N_{x,i}},\, z_{i}\in\mathbb{R}^{N_{x,i}}\\
				& \bm x_{i} = f_{i}\big(\bm \zeta_{\Ni}, \bm u_{i}, \bm \xi_{i}\big)\\
				& (\bm x_{i}, \bm u_{i}) \in \mc O_i \\
				& \bm x_i \in {\mc X}_i(Y_i, z_i),\, {\mc X}_i(\cdot) \in \mc F_{\AT}(\mc S_i)
			\end{array}\right \rbrace \begin{array}{@{}l}
				\forall \bm \zeta_{\Ni} \in {\mc X}_{\Ni}(Y_\Ni, z_\Ni)\\
				\forall \bm \xi_i \in \Xi_i
			\end{array}\forall i \in \mc M,
		\end{array}
	\end{equation}
	where for each $i\in \mc M$ we define $ {\mc X}_{\Ni}(Y_{\Ni}, z_\Ni) = \bigtimes_{j \in \mc N_i} {\mc X}_j(Y_{j}, z_j) $. The optimization variables are policies $ \bm {\Psi}_{i}(\cdot) $,  matrices $ Y_i $ and vectors $ z_i $ for all $ i \in \mc M $. 
	Notice that by construction Problem~\eqref{DecentralizedXab_GeneralY} belongs in the class of problems with a decision-dependent uncertainty set, thus is in general hard to solve \citep{Nohadani2018}. 
	The affine restriction $ \mc F_{\AT}(\mc S_i) $ provides a conservative approximation of Problem~\eqref{Decentralized} which is stated without proof in the following corollary.
	\begin{corollary}
		\label{corr::1} Problem~\eqref{DecentralizedXab_GeneralY}  is a conservative approximation of Problem~\eqref{Decentralized} in the following sense: every feasible solution of Problem~\eqref{DecentralizedXab_GeneralY} is feasible in Problem~\eqref{Decentralized}, and the optimal value of Problem~\eqref{DecentralizedXab_GeneralY}  is equal or larger than the optimal value of Problem~\eqref{Decentralized}.
	\end{corollary}

	By taking advantage of the simple structure of \eqref{SetApproximation} we can reformulate the problem as a multistage robust optimization problem with decision-independent uncertainty sets. In the spirit of \citep{Bitlisliouglu2017,Jaillet2016,Zhang2017}, we design strictly  causal feedback policies $ \Gamma_{i,t}:\mb R^{\hat n_{i}^t} \rightarrow \mb R^{n_{u,i}} $ such that the control input at each time step is given by $ u_{i,t} = \Gamma_{i,t}(\bm \xi_i^{t-1}, \bm s^t_\Ni) $.  We write $ \bm \Gamma_i(\bm \xi_i, \bm s_{\Ni})  = [\Gamma_{i,t}(\bm \xi_i^{t-1}, \bm s^t_\Ni)]_{t\in\mc T} $ to denote the policy concatenation over time. In this context, the counterpart of Problem~\eqref{DecentralizedXab_GeneralY} is given~as
	\begin{equation}\label{DecentralizedFinal_GeneralY}
		\begin{array}{@{}l}
			\text{~minimize~~} \displaystyle\sum\limits_{i = 1}^M \max\limits_{\bm \xi_i \in \Xi_i, \bm s_{\Ni} \in \mc S_{\Ni}} J_i(\bm x_i,\bm u_i) \\
			\left.\begin{array}{@{}r@{~}l@{}}
				\text{subject to }
				& \bm u_i = \bm \Gamma_i(\bm \xi_i, \bm s_{\Ni})  := [\Gamma_{i,t}(\bm \xi_i^{t-1}, \bm s^t_\Ni)]_{t\in\mc T} \\
				& Y_i\in\mathbb{S}_+^{N_{x,i}},\, z_{i}\in\mathbb{R}^{N_{x,i}}\\
				& \bm \zeta_{\Ni} = Y_\Ni \bm s_\Ni + z_\Ni\\ 
				& \bm x_{i} = f_{i}\big(\bm \zeta_{\Ni}, \bm u_{i}, \bm \xi_{i}\big)\\
				& (\bm x_{i}, \bm u_{i}) \in \mc O_i \\
				& \bm x_i \in {\mc X}_i(Y_i, z_i),\, {\mc X}_i(\cdot) \in \mc F_{\AT}(\mc S_i)
			\end{array}\right \rbrace \begin{array}{@{}l}
				\forall \bm s_{\Ni} \in \mc S_{\Ni}\\
				\forall \bm \xi_i \in \Xi_i
			\end{array}\forall i \in \mc M,
		\end{array}
	\end{equation}
\end{subequations}
where $ \mc S_{\Ni} = \bigtimes_{j \in \mc N_i} \mc S_j $, and with a slight abuse of notation  $ Y_\Ni \bm s_\Ni + z_\Ni = \left[ Y_i \bm s_{j} + z_j \right]_{j \in \mc N_i} $ for each $ i \in \mc M $. The decision variables are $ \bm \Gamma_i(\cdot) $, $ Y_i $ and $ z_i $ for all $ i \in \mc M $. Problem~\eqref{DecentralizedFinal_GeneralY} can be classified as a multistage robust optimization problem with exogenous uncertainty~\citep{Ben2004}. 

In the following, we investigate the relationship between Problem~\eqref{DecentralizedXab_GeneralY} and  Problem~\eqref{DecentralizedFinal_GeneralY}. To do so, we first define the linear mapping $ R_{i,t}: \mb R^{n_{x,i}} \rightarrow \mb R^{n_{x,i}} $, $ s_{i,t} \mapsto x_{i,t} $, as,
\begin{subequations}\label{app::maps}
	\begin{equation}\label{app::map1}
		R_{i,t}(s_{i,t}) = Y_{i,t} s_{i,t} + z_{i,t},
	\end{equation}
	and the linear mapping, $ L_{i,t}: \mb R^{n_{x,i}} \rightarrow \mb R^{n_{x,i}} $, $ x_{i,t} \mapsto s_{i,t} $, as,
	\begin{equation}\label{app::map2}
		L_{i,t}(x_{i,t}) = Y_{i,t}^{+}(x_{i,t} -z_{i,t})
	\end{equation}
\end{subequations}
where $  Y^+_{i,t} := (Y_{i,t}^\top Y_{i,t})^{-1}Y^\top_{i,t} $ is the pseudo-inverse of the positive semi-definite matrix $ Y_{i,t} $.  Note that the $ L_{i,t} $ may not be unique because of the pseudo-inverse $ Y^+_{i,t} $. Moreover, $ R_{i,t} $ can be viewed as a ``left inverse'' of the operator $ L_{i,t} $, i.e., it satisfies $ R_{i,t}\big(L_{i,t}(x_{i,t})\big) = x_{i,t}  $. Using \eqref{app::maps}, the following theorem establishes equivalence between the two optimization problems.

\begin{theorem} \label{thm::5}
	Problem~\eqref{DecentralizedXab_GeneralY} is equivalent to Problem~\eqref{DecentralizedFinal_GeneralY} in the following sense: feasible solutions in Problem~\eqref{DecentralizedXab_GeneralY} are mapped to feasible solutions in Problem~\eqref{DecentralizedFinal_GeneralY}, and their optimal values are equal.
\end{theorem}
{\color{black}
	Policies $\bm {\Gamma}_{i} (\bm \xi_i, \bm s_{\Ni})$ are designed using \eqref{DecentralizedFinal_GeneralY}, however, in practice  agent $i$ will require policies $ \Psi_i(\bm \xi_i, \bm \zeta_{\Ni}) $ since  $\bm \zeta_{\Ni}$ will be the observed states of its neighbors, while $\bm s_{\Ni}$ is in some sense the mathematical construct that allows us to convexify the problem. Theorem~\eqref{thm::5} allows us to construct a mapping between the two policies. For agent $i$, given $ u_{i,t} = \Gamma_{i,t}(\cdot) $ as the optimal  policy  of \eqref{DecentralizedFinal_GeneralY}, then $ u_{i,t} = \Psi_{i,t}(\bm \xi_i^{t-1}, \bm \zeta_\Ni^t) = \Gamma_{i,t}(\bm \xi_i^{t-1},[ R_j^t(\bm \zeta_j^t)]_{j\in \mc N_i}) $ for all $\bm \zeta_{\Ni} \in \mc X_{\Ni},\, \bm \xi_i \in \Xi_i$ is an optimal policy in problem \eqref{DecentralizedXab_GeneralY}. One can go a step further and construct the optimal state feedback policy $\bm {\psi}_{i}(\bm x_i, \bm \zeta_{\Ni})$ in a similar spirit as in \citep{Hadjiyiannis2011,Goulart2006}, however, we omit this derivation in the interest of space. 
}

Despite the exogenous uncertainty structure of Problem~\eqref{DecentralizedFinal_GeneralY}, the problem is in general hard to solve due to the non-convex constraint~\eqref{eq::SFSv1_nonconvex}   as both $Y_i$ and $\bm s_i$ are optimization variables resulting in  bilinear terms. The problem can be efficiently approximated in practice by performing block-coordinate descent \citep{MR3444832} on $Y_i$ and $\bm s_i$ in a sequential manner until the parameters converge. In the following, we present an additional restriction that results in a linear formulation of the problem.  

Consider the case where $Y_i = y_i$ where $y_i\in\mathbb{R}_+$, thus \eqref{eq::SFSv1_nonconvex}  takes the form
\begin{subequations}\label{ConvexSetApproximation}
	\begin{equation}\label{eq::SFSv1_new}
		\begin{array}{l}
			\mc X_i(y_i,z_i) = \displaystyle\left\{\bs{x}_i \in\mathbb{R}^{N_{x,i}} \,:\, \exists \bm s_{i} \in \mc S_i \text{ s.t. } \bm x_{i} =  y_{i} \bm s_{i} + z_{i} \right\}
		\end{array}
	\end{equation}
	This restriction allows for the scaling of $\mc S_i$ but does not allow for the rotating of the set. By taking advantage of the fact that $ \mc S_i $ is compact,  $ \mc X_i(y_i,z_i) $ can be expressed as the following convex set. 
	\begin{equation}\label{eq::SFSv2_new}
		\mc {\widehat X}_i(y_i,z_i) = \displaystyle\left\{\bs{x}_i \,:\, \exists \bm  \nu_{i}\in\mathbb{R}^{N_{x,i}} \text{ s.t. }  \displaystyle\bm x_{i} = \bm \nu_{i,k} + z_{i},\;\;\\
		G_{i} \bm \nu_{i}  \preceq_{\mc K_{i}}  y_{i}  g_{i}\right\}
	\end{equation}
\end{subequations}
where $ \bm \nu_{i} $ are auxiliary variables. By replacing constraint $\bm x_i\in\mc X(y_i,z_i)$ in Problem~\eqref{DecentralizedFinal_GeneralY} with  $\bm x_i\in\wh{\mc X} (y_i,z_i)$ the optimization problem reduces to a multistage \emph{linear} optimization problem with exogenous uncertainty. The relationship between sets \eqref{eq::SFSv1_new} and \eqref{eq::SFSv2_new} is summarized in the following proposition. 
\begin{proposition}\label{prop::nCc}
	Set $ \mc X_{\FS} = \{(\bm x_i, y_i, z_i)\,:\, \bm x_i \in \mc X_i(y_i, z_i) \} $ is equivalent to $ \wh{\mc X}_{\FS} =\{(\bm x_i, y_i, z_i)\,:\, \bm x_i \in \mc {\widehat X}_i(y_i,z_i) \} $ in the following sense: there exist a unique mapping between feasible points in $ {\mc X}_{\FS} $~and~$\wh{\mc X}_{\FS} $.
\end{proposition}

\begin{remark}\label{remark1}
	Approximation \eqref{ConvexSetApproximation} can be very restrictive as it controls the scaling of $\mathcal X_i$ through a single  parameter, namely $y_i$. Regardless, the idea can be further extended to allow for the scaling of each (or groups) of the coordinate of $\mathcal X_i$ independently from the rest. Consider the case in which $\mathcal X_i$  can be represent as $\mathcal X_i = \mathcal X_{i,1}\times\cdots\times \mathcal X_{i,K_i}$ for some $K_i\leq N_{x,i}$. Then, for each $j\in\{1,\ldots,K_i\}$, we  restrict each $\mathcal X_{i,j}$ to admit a representation as in \eqref{ConvexSetApproximation}. For example, if $K_i = N_{x,i}$,  and setting $ \mathcal S_{i,j} = \{ s \in \mathbb R : | s | \leq 1 \} $ then we have 
	\begin{equation*}
		\begin{array}{{l@{\,}l}}
			\mc X_{i,j}(y_{i,j},z_{i,j}) &= \displaystyle\left\{\bs{x}_{i,j} \in\mathbb{R} \,:\, \exists s \in \mc S_{i,j} \text{ s.t. } \bm x_{i,j} =    y_{i,j} s + z_{i,j} \right\}\\
			&= \displaystyle\left\{\bs{x}_{i,j} \in\mathbb{R} \,:\,  z_{i,j} -y_{i,j} \leq \bm x_{i,j}\leq z_{i,j} +y_{i,j} \right\}\\
		\end{array}
	\end{equation*}
	thus allowing to control the width and position of the communication set for  each of the components of $\bs{x}_{i}$. This approximation is natural in many applications like the contract design example discuss in Section~\ref{supplychain}, in which this modeling choice is prescribed by the problem itself and does not constitute an additional approximation to the problem.
\end{remark}

As previously stated, problem~\eqref{DecentralizedFinal_GeneralY} can be classified as a multistage robust optimization problem.
There is a plethora of schemes that try to solve this class of problems either exactly using dynamic programming techniques \citep{Georghiou2019b}, or approximately  by  restricting the functional form of the admissible policies to admit prescribed structures, see \citep{Delage2015,Georghiou2019a} for a thorough survey on the topic. To simplify  exposition, in the remainder of the  paper we restrict the policies to admit an affine structure. This will also allow for direct comparisons with the work in \citep{Goulart2006} and \citep{Lin2016}, which we  discussed in Sections~\ref{centralized} and \ref{partiallynested}, respectively.

Problem~\eqref{DecentralizedFinal_GeneralY} retains the decoupled structure of \PDs since agent $ i $ only needs to communicate to its direct neighbors  $(Y_i,z_i) $ that control the shape of its state forecast set.  This loosely coupled structure allows employing distributed computation algorithms such as the alternating direction method of multipliers (ADMM) algorithm \citep[\S~7]{Boyd2004}. Such schemes will allow the agents to solve Problem~\eqref{DecentralizedFinal_GeneralY} in an almost decentralized manner, i.e., each agent solves the subproblem of Problem~\eqref{DecentralizedFinal_GeneralY} involving its own objective and constraints while seeking to achieve consensus for the variables $ y_i $ and  $ z_i $ with its neighbors.
In Appendix~\ref{supplychain} we illustrate how the ADMM algorithm can be implemented in a supply chain example.

We next demonstrate that if the physical network forms an arborescence, it is possible to explicitly establish an upper bound on the complexity of the approximation required for the optimal values of Problem \eqref{DecentralizedFinal_GeneralY} and Problem \eqref{Decentralizedb} to coincide. For reference, an arborescence, sometimes referred to as a directed rooted tree, is a directed network in which there is a unique directed path from a root node to any other node. To achieve this, we can adjust the approximation~\eqref{SetApproximation} to include not only scaling and rotation of $\mathcal{S}_i$ but also projections. This adjustment involves selecting a higher dimensional set $\mathcal{S}_i \subseteq \mathbb{R}^{\widetilde{N}_{x,i}}$ with $\widetilde{N}_{x,i} \geq N_{x,i}$ and allowing for rectangular matrix $Y_i \in \mathbb{R}^{N_{x,i} \times \widetilde{N}_{x,i}}$. Thus the affine mapping $\bm x_{i} = Y_i \bm s_{i} + z_{i}$ where $\bm s_{i} \in \mc S_i$  projects the high dimensional set $\mc S_i$ to the lower dimension of the state $\bm x_i$. A practical choice of $\mc S_i$ is  the simplex set since by increasing its dimension, one can systematically increase the complexity of the projected set. We note that other choices are also possible. The result is summarized in the following corollary. 

\begin{corollary}\label{cor::2}
	If the physical network is an  arborescence, then the optimal value of Problems~\eqref{Decentralizedb} and \eqref{DecentralizedFinal_GeneralY} coincide if approximation \eqref{SetApproximation} is constructed such that for each $i\in\mc M$ the set $\mathcal{S}_i$ is chosen to be the simplex of dimension  $\prod_{j\in \oNi} |\text{ext}(\Xi_j)|$.
\end{corollary}

We end this section with Figure~\ref{fig::theorems} which summarizes the relationship between the different problems discussed in the paper so far.

\begin{figure}[ht]
	\centering
	\includegraphics[width = 1\textwidth]{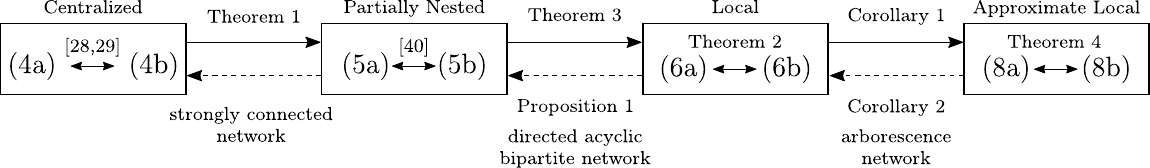}
	\caption{Illustration of the relationships among the various problems discussed in the paper. The optimal values of these problems are non-decreasing in the directions of the arcs. Dashed arcs indicate  special cases of networks, while solid arcs represent inequalities that hold under general network topologies.}
	\label{fig::theorems}
\end{figure}

\section{Numerical experiments}\label{sec::Numerics}
We now investigate the efficacy of the proposed approach in two examples: an illustrative which allows us to understand numerically and graphically the connection between the shape of the primitive sets $\mathcal{S}_i$ and the solution quality and a stylized contract design for supply chains that are typically operated in a distributed fashion. All optimization problems were solved with Gurobi in MATLAB using the RSOME \citep{Chen2020robust} and YALMIP \citep{Lofberg2004} interfaces on a computer equipped with \mbox{$ 8 $ GB RAM} and \mbox{$ 2.4 $ GHz} quad-core Intel processor.

\subsection{Illustrative example}\label{illustrative}

We consider a single stage problem composed by two agents with states $ x_1, x_2 \in \mathbb R^2 $, inputs $ u_1, u_2 \in \mathbb R $ and uncertainties $ \xi_1, \xi_2 \in \Xi = \{\xi \in \mathbb R^2: \| \xi \|_\infty \leq 1 \} $, respectively. The network structure along with the objective functions and constraint sets of each agent are shown in Figure~\ref{fig::toyExample}. The system matrices are given by
\begin{equation}
	c = \begin{bmatrix}
		1 \\ -1
	\end{bmatrix}, \;B = \begin{bmatrix}
		1 & 0 \\ 0 & -2
	\end{bmatrix},\; D = \begin{bmatrix}
		1 \\ 0.8
	\end{bmatrix} \text{ and } E = \begin{bmatrix}
		1 & -1 \\ -1 & 1
	\end{bmatrix}.
\end{equation}

\begin{figure*}
	\centering
	\includegraphics[width=0.6\textwidth]{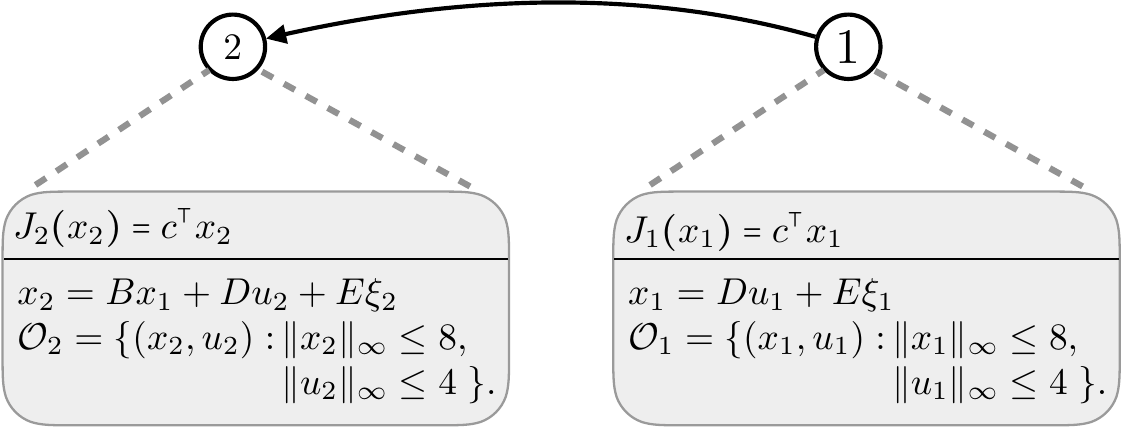}
	\caption{Physical and information structure of the two agents in the system.}
	\label{fig::toyExample}
\end{figure*}

We start by computing the optimal affine policy associated for the centralized communication and the partially nested communication, which are reminiscent of Problem~\eqref{Centralized} and Problem~\eqref{Semi-Centralized}, respectively. In the centralized communication both agents observe the uncertain variables $\xi_1$ and $\xi_2$, while in the partially nested communication agent $1$ only observes $\xi_1$ whereas agent $2$ observes both $\xi_1$ and $\xi_2$ as dictated by the structure of the network. In other words, the second agent's policy is of the form $u_2(\xi_1, \xi_2) = \Gamma_{2,1} \xi_1 + \Gamma_{2,2} \xi_2 + \gamma_{2}$ in both cases, and the first agent's policy is of the form $u_1(\xi_1, \xi_2) = \Gamma_{1,1} \xi_1 + \Gamma_{1,2} \xi_2 + \gamma_{1}$ in the centralized communication and $u_1(\xi_1, \xi_2) = \Gamma_{1,1} \xi_1 + \gamma_{1}$ in the partially nested communication. Figure~\ref{fig::centralizedSynthesis} illustrates information about the states of agent $ 1 $. With red we depict the feasible region of $ x_1 $, while blue depicts region $x_1(\xi_1,\xi_2) = D u_1^\star(\xi_1,\xi_2) + E\xi_1$ for all $\xi_1 \in \Xi_1,\, \xi_2 \in \Xi_2$, where $u_1^\star$ is the optimal affine policy of the agent $1$. Figure~\ref{fig::centralizedSynthesis}(a) and (b) illustrate these regions in the case of centralized and partially nested communications, respectively. In addition, we report the resulting objective values in the title of the respective figure as ``obj. = $ J_1(x_1) + J_2(x_2) $''. We observe that the information restriction imposed on the partially nested policies results in a larger objective value compared to the centralized solution. 
\begin{figure*}[ht]
	\center
	\begin{minipage}{0.425\textwidth}
		\subfigure[]{\includegraphics[width = \textwidth]{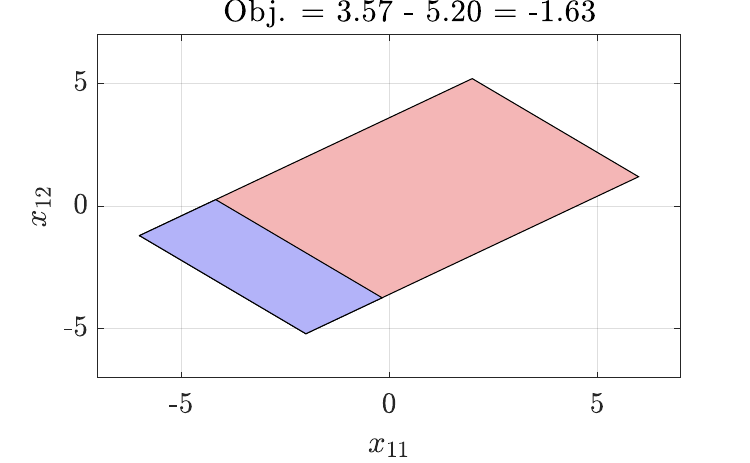}\label{fig::centralizedSynthesis_a}}
	\end{minipage}\hfil
	\begin{minipage}{0.425\textwidth}
		\subfigure[]{\includegraphics[width = \textwidth]{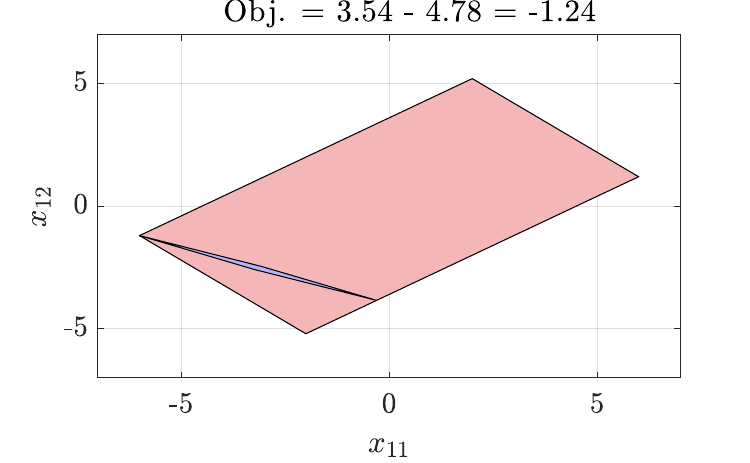}\label{fig::centralizedSynthesis_b}}
	\end{minipage}
	\caption{Illustrating information for the state of agent $1$ for the (a) centralized and (b) partially nested communication exchange. The red region represents the feasible set for $ x_1 $ and the blue region represents all possible values for $x_1$ under the optimal policy.}
	\label{fig::centralizedSynthesis}
\end{figure*}

Next, we compute affine policies with local information exchange, which is reminiscence of Problem \eqref{DecentralizedFinal_GeneralY}. In the local information exchange, the agents' policies are of the form $u_1(\xi_1, \xi_2) = \Gamma_{1,1} \xi_1 + \gamma_{1}$ and $u_2(s_1, \xi_2) = \Gamma_{2,1} s_1 + \Gamma_{2,2} \xi_2 + \gamma_{2}$, respectively, where $s_1 \in \mc S_1$ represents the auxiliary uncertainty introduced in Section~\ref{sec::SolMethod}. 
We start with investigating how the shape of set $ \mathcal S_1 $  together with the restriction of $\mathcal{X}_1$  affects the quality of the solution in terms of objective value for three cases. In the first case (flexible), we choose $ \mathcal S_1 = \{ s \in \mathbb R^2 : \| s \|_{\infty} \leq 1 \} $ and restrict $\mathcal{X}_1$ as in equation~\eqref{eq::SFSv1_nonconvex}. In the second case (rectangle), we choose $ \mathcal S_{1,1} = S_{1,2} = \{ s \in \mathbb R : | s | \leq 1 \} $ and restrict $\mathcal{X}_1 = \mathcal{X}_{1,1}\times\mathcal{X}_{1,2}$ as in Remark~\ref{remark1}, and in the third case (circle), we choose $ \mathcal S_1 = \{ s \in \mathbb R^2 : \| s \|_2 \leq 1 \} $ and  set $\mathcal{X}_1 = \{\bs{x}_1 \in\mathbb{R}^{2}: \exists \bm s \in \mc S_1 \text{ s.t. } \bm x_{1} =  \bm s y_{1} + z_{1}\}$. 

The robust counterpart in the first case results in a bilinear optimization problem which we solve using block-coordinate descent, the second case results in a linear program, while the third case results in a semidefinite program due to the use of the celebrated S-Lemma \citep[\S~3.5]{Ben2001lectures} for reformulating the robust constraints.
The flexible approximation exactly recovers the optimal solution of partially nested communication exchange, i.e., Figure~\ref{fig::centralizedSynthesis} (b). 
The result coincides with  Corollary~\ref{cor::2} and Proposition~\ref{prop::1} as the physical network is both and arborescence and a directeed acyclic bipartite graph, and additionally the approximation has enough degrees of freedom to adequately represent the optimal state forecast set.
This is not the case for the rectangular and circular approximation, whose performance is depicted in Figure~\ref{fig::distributedNorms}. 

\begin{figure*}[!tb]
	\center
	\begin{minipage}{0.425\textwidth}
		\subfigure[]{\includegraphics[width = \textwidth]{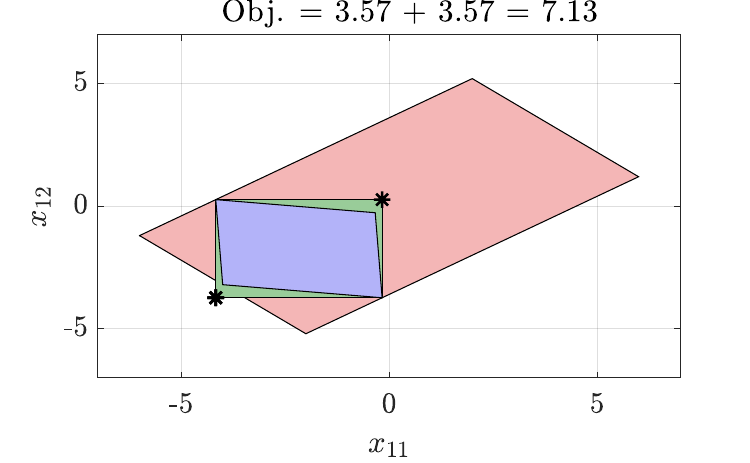}}
	\end{minipage}\hfil
	\begin{minipage}{0.425\textwidth}
		\subfigure[]{\includegraphics[width = \textwidth]{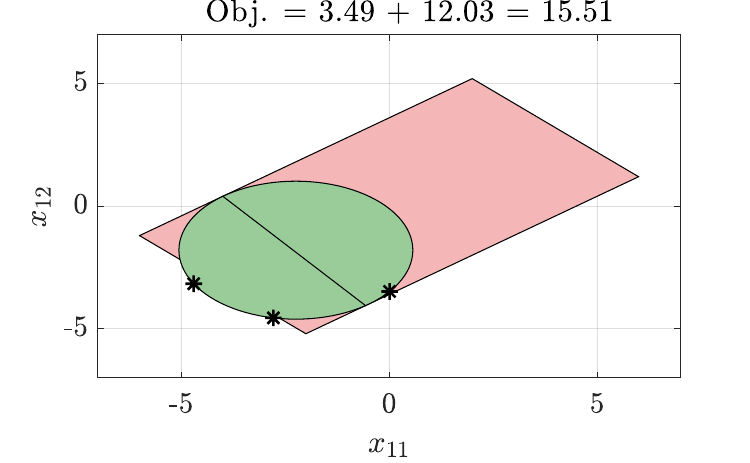}}
	\end{minipage}
	\caption{Illustrating information for the state of agent $1$ using $ \mathcal S_1 $ as {(a) box}, and {(b) circle}.  The green region represents the feasible set for $ x_1 $, the yellow represents $x_1(\xi_1)$ under the optimal policy, while red denotes $\mc X_1$.  Black stars depict the binding scenarios for agent $2$ in the state forecast set $ \mathcal X_1 $.}
	\label{fig::distributedNorms}
\end{figure*}

As in the first experiment, red area represents the feasible region of $ x_1 $, while the blue region depicts $x_1(\xi_1) = D u_1^\star(\xi_1) + E\xi_1$ for all $\xi_1 \in \Xi_1$. Additionally, the green area denotes the optimal forecast set $ \mathcal X_1 $ communicated to agent $2$. Finally, the black stars depict the binding scenarios for agent $2$ in the state forecast set $ \mathcal X_1 $. 
We observe that the optimal region of agent $1$ changes with the different primitive sets in an attempt to cooperate with agent $2$. This cooperative behavior is also identified in the objective values as agent $1$ ``sacrifices'' some of its optimality for the good of agent $2$. Interestingly, part of the state forecast set $ \mathcal X_1 $ may lie outside the feasible region of the problem which adds conservativeness to the system as can be verified by inspecting the position of the binding scenarios. Moreover, since the binding scenarios are not necessarily placed at the corner points defined by the optimal region, agent $1$ retains some of its privacy since the behavior of its optimal policy is not revealed to agent~$2$.

To quantify the importance of the primitive set orientation in space, in the following experiment we use as primitive sets $ (i) $ a rotated rectangular set whose major and minor axis can be scaled independently, i.e., using the same notation as before and with slide abuse of notation we have $\mathcal{X}_1 = A(\mathcal{X}_{1,1}\times\mathcal{X}_{1,2})$ where $A\in\mathbb{R}^{2\times 2}$ is a fixed rotation matrix, and $ (ii) $ a rotated ellipsoid for which the major axis is forced to be $ 1.5 $ times larger than its minor axis, i.e, we set $\mathcal{X}_1 = \{\bs{x}_1 \in\mathbb{R}^{2}: \exists \bm s \in \mc S_1 \text{ s.t. } \bm x_{1} =  A(\bm s y_{1} + z_{1})\}$ where again $A\in\mathbb{R}^{2\times 2}$ is a fixed rotation matrix. Figure~\ref{fig::rotatedSets} illustrates information for the state of agent $1$, where we use the same color code as in Figure~\ref{fig::distributedNorms}.
Comparing the solutions from Figure~\ref{fig::distributedNorms} and Figure~\ref{fig::rotatedSets}, we can conclude that both the shape and orientation heavily affect the optimal policies as well as the optimal value of the problems. 
\begin{figure*}[th]
	\center
	\begin{minipage}{0.425\textwidth}
		\subfigure[]{\includegraphics[width = \textwidth]{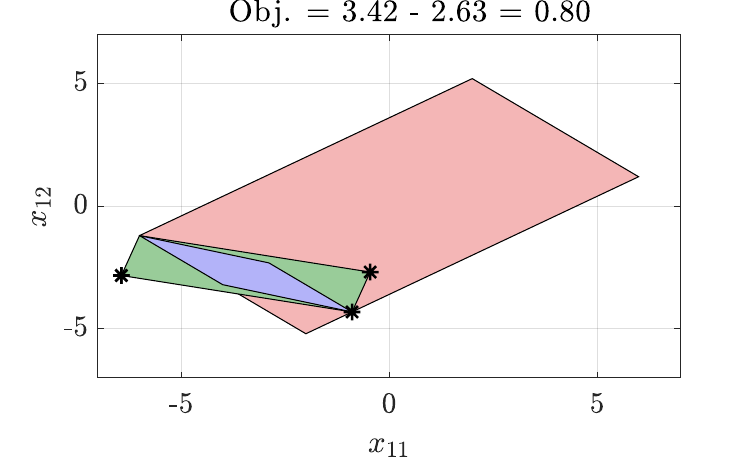}}
	\end{minipage}\hfil
	\begin{minipage}{0.425\textwidth}
		\subfigure[]{\includegraphics[width = \textwidth]{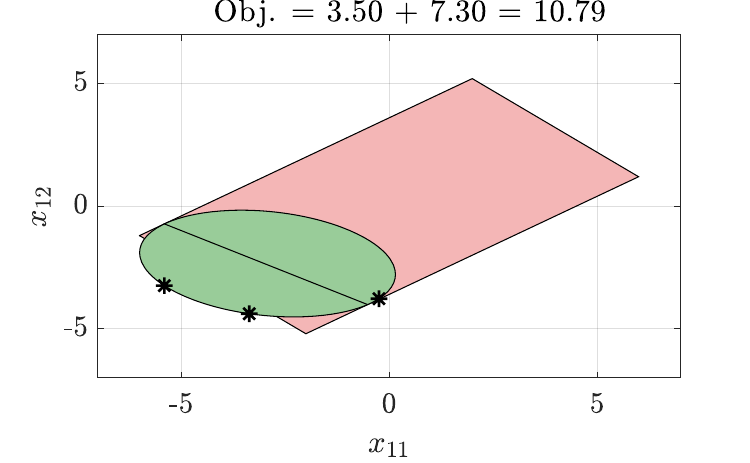}}
	\end{minipage}
	\caption{Illustrating information for the state of agent $1$ using $ \mathcal S_1 $ as (a) rectangle rotated by 15 degrees and (b) ellipsoid rotated by 15 degrees  for which the major axis is forced to be $ 1.5 $ times larger than its minor axis. The graphs have the same interpretation as in Figure~\ref{fig::distributedNorms}.}
	\label{fig::rotatedSets}
\end{figure*}

To clarify this finding, we repeated the experiment for all possible rotations in the range $ [0,\,180] $ degrees of the rectangular and scaled ellipsoid sets by appropriately choosing matrix $A$. The results are reported in Figure~\ref{fig::costPlot}. We observe that if the rotation of the communicated sets is aligned with the set generated by the optimal affine policy, then the solution resulting from the proposed distributed method closely approximates the partially nested and centralized solutions. If, however, this is not the case, then the cost considerably deviates and even leads to infeasibility. 
\begin{figure*}[!tb]
	\centering
	\includegraphics[width=0.7\textwidth]{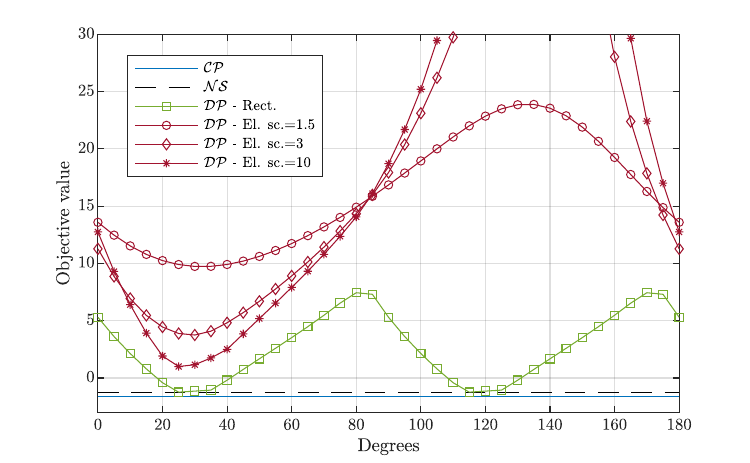}
	\caption{Comparing the objective value achieved by affine policies using the centralized (C), partially nested (PN), and local information exchange (L). For the local information we set $\mc S_1$ to be a rotated rectangle (L-Rect.),  and a rotated ellipsoid (L-El.) for which the major axis is forced to be $ 1.5, 3 $ and $10$ times larger than its minor axis. The graph reports the objective values for all possible rotations in the range $ [0,\,180] $.}\label{fig::costPlot}
\end{figure*}

\subsection{Cooperative energy management system}
\label{sec::energy}

Buildings account for approximately $40\%$ of the world's energy consumption, with a significant portion dedicated to enhancing living conditions through heating, ventilation, and air conditioning to promote healthier and more comfortable environments \citep{laustsen2008energy}. Motivated by these substantial energy demands, researchers have invested considerable efforts into developing complex control schemes that can reduce energy usage while keeping room temperatures within predefined ranges. Among such schemes, cooperative control systems have been instrumental for achieving energy savings by coordinating aggregated demands across multiple users. Most existing cooperative methods assume the existence of a central operator that is capable of controlling both the buildings actuation systems and the energy hub devices \citep{oldewurtel2012use,sturzenegger2015model}. This assumption, however, becomes restrictive for large-scale implementations due to heavy computational burden. Additionally, privacy concerns around revealing sensitive building information, such as occupancy patterns and comfort preferences, arise. Consequently, a new wave of research has focused on developing decentralized or distributed control schemes to address the limitations associated with centralized approaches \citep{keviczky2006decentralized, scattolini2009architectures, stewart2010cooperative}.

\begin{figure} [!tb]
	\centering
	\includegraphics[width=0.8\textwidth]{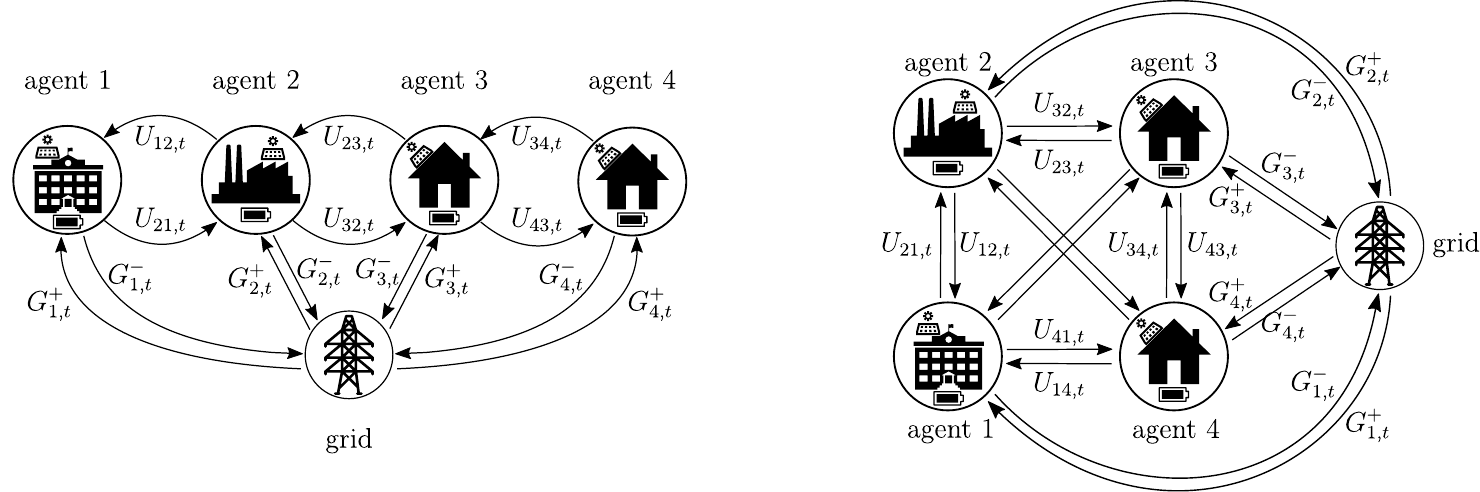}
	\caption{An example of a serial (left) and complete (right) energy hub topologies with different consumers.}
	\label{fig::energyhub}
\end{figure}

In this section we examine the scalability and efficiency of our proposed method using an energy hub system with aggregated consumers. Figure~\ref{fig::energyhub} is an illustration of a serial and a complete (fully-connected) energy hub consisting of four different consumers. In this setup agent $i$ at time $t$ can buy $G_{i,t}^+$ units of electricity from the grid at a cost of $c_t^+$ and can return $G_{i,t}^-$ units at the cost of $c_t^-$. The agents also have the ability to share excess electricity amongst themselves, with  $U_{ij,t}$ denoting the amount of electricity agent $i$ takes from agent $j$ at time $t$. This transfer of electricity is done at a cost of $c_t^u$ per unit since the agents need to use the existing grid infrastructure. The state of charge of the battery of agent $i$ at time $t$ is denoted by $I_{i,t}$. We assume that the output of the photovoltaic system is uncertain and denoted by $R_{i,t}(1+\xi_{i,t}^R)$ where $R_{i,t}$ denotes the average production and $\xi_{i,t}^R$ denotes the uncertain percentage deviation from the average production. Similarly, the uncertain demand for electricity of agent $i$ at time $t$ is denoted by $D_{i,t}(1+\xi_{i,t}^D)$ where $D_{i,t}$ is the average demand and $\xi_{i,t}^D$ denotes the uncertain percentage deviation from the average demand. Finally we define $\xi_{i,t} = (\xi_{i,t}^D, \xi_{i,t}^R)$. We express the battery dynamics of agent $i$ through
\begin{align*}
	I_{i,t+1} = I_{i,t} + G^+_{i,t} - G^-_{i,t} + \sum_{j \in \mc N_i} U_{ij,t} - \sum_{j \in \mc N_i} U_{ji,t} + R_{i,t} ( 1 + \xi_{i,t}^p ) - D_{i,t} ( 1 + \xi_{i,t}^d) \quad \forall \bm\xi_{i} \in \Xi_i, \forall t\in\mathcal{T},
\end{align*}
where for the illustrative example of the serial energy hub in  Figure~\ref{fig::energyhub}, we have $\mc N_1 = \{2\},\, \mc N_2 = \{1 , 3 \},\, \mc N_3 = \{ 2,4 \},\, \mc N_4 = \{ 3 \}$, and for the complete energy hub we have  $\mc N_1 = \{2,3,4\},\, \mc N_2 = \{1 , 3,4 \},\, \mc N_3 = \{ 1,2,4 \},\, \mc N_4 = \{1,2, 3 \}$.
In addition, at any $t \in \mc T_+$, we want to ensure that constraint $0 \leq I_{i,t} \leq B_i$ is robustly satisfied for some $B_i > 0$ that represents the capacity of the battery.

The communication links needed to formulate the serial energy hub in Figure~\ref{fig::energyhub} are depicted in Figure~\ref{fig::SMD_Network}(a). Since the serial network is a strongly connected network, as discussed in Section~\ref{partiallynested}, the number of links required to formulate the partially nested Problem~\eqref{Semi-Centralized} is the same as in the centralized Problems~\eqref{Centralized}, effectively leading to the same problem. On the other hand, the links needed by the  local information exchange   matches the physical coupling of the network thus  establishing communication only between adjacent agents in the system, see Figure~\ref{fig::SMD_Network}(b). This demonstrates the modeling benefits of the proposed approach and highlights the limitations of the partially nested information exchange on certain network topologies.
\begin{figure}
	\center
	\begin{center}	\scalebox{0.45}{ \includegraphics{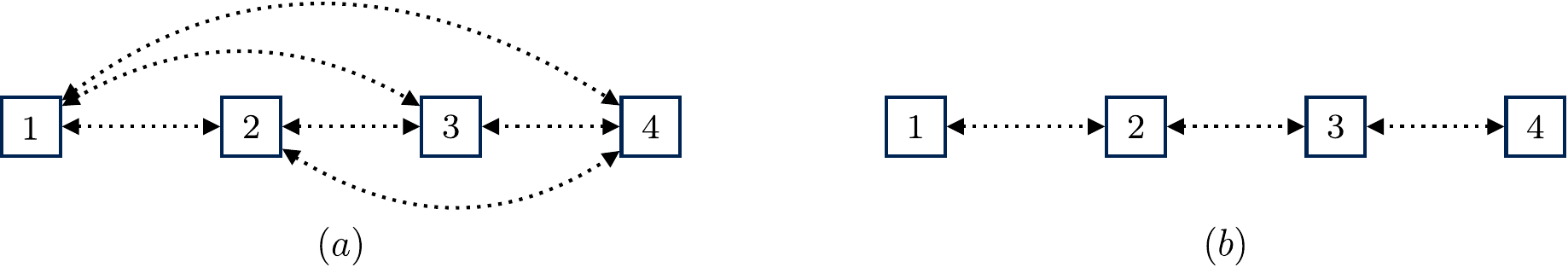}} \end{center}
	\caption{Communication links need to achieve (a)  centralized and partially nested information exchange and (b)  local information exchange for the serial energy hub system.}
	\label{fig::SMD_Network}
\end{figure}

In the centralized information exchange problem, the agents communicate the state of their batteries to all agents in the grid. We formulate the centralized problem as an instance of  Problem~\eqref{Centralizedb} as follows, where the objective minimizes the total cost incurred  by all agents. 
\begin{align}
	\label{eq:central:energy}
	\begin{array}{l}
		\text{~minimize~~} \displaystyle \sum_{i=1}^M \max\limits_{\bm \xi_{\mc M} \in \Xi_{\mc M}} \sum_{t =1}^{T} c_{t}^+ G^+_{i,t} + \sum_{t =1}^{T} c_{t}^- G^-_{i,t} + \sum_{j \in \mc N_i} \sum_{t =1}^{T} c_{t}^u U_{ij,t} \\ 
		\left.
		\begin{array}{@{}r@{~}ll}
			\text{subject to} & \bm G_i^+ = \bm \Pi_i^+(\bm \xi_{\mc M}), ~ \bm G_i^- = \bm \Pi_i^-(\bm \xi_{\mc M}) \\
			& \bm U_{ij} \in \bm \Pi_{ij}(\bm \xi_{\mc M}), ~ \bm U_{ji} \in \bm \Pi_{ji}(\bm \xi_{\mc M}) & \forall j \in \mc N_i \\
			& G_{i,t}^+ \geq 0, ~  G_{i,t}^- \geq 0, ~ U_{ij,t} \geq 0, ~ U_{ji,t} \geq 0 & \forall j \in \mc N_i, \forall t \in \mc T \\
			& \displaystyle I_{i,t+1} = I_{i,t} + G^+_{i,t} - G^-_{i,t} + \sum_{j \in \mc N_i} U_{ij,t} - \sum_{j \in \mc N_i} U_{ji,t}
			\\[-1em]
			& & \forall t \in \mc T \\[-1em]
			& \hspace{5em} + R_{i,t} ( 1 + \xi_{i,t}^p ) + D_{i,t} ( 1 + \xi_{i,t}^d) \\
			& \displaystyle 0 \leq I_{i,t} \leq B_i & \forall t \in \mc T_+
		\end{array}
		\right \rbrace 
		\begin{array}{l}
			\forall i \in \mc M  \\
			\forall \bm \xi_{\mc M} \in \Xi_{\mc M}
		\end{array} \hspace{-2em}
	\end{array}
\end{align}

Consider now the case in which instead of communicating the state of their batteries to all agents in the grid, agents are restricted to communicate an interval that represents the amount of electricity they can supply  their neighbors at each time period. The interval $[\underline{b}_{ji,t}, \overline{b}_{ji,t}]$ denotes the amount of electricity agent $i$  commits to supply agent $j$ at time $t$.	Of course, agent $i$ does not have  to use all the electricity available from agent $j$ hence $U_{ij,t} \in [\underline{b}_{ij,t}, \overline{b}_{ij,t}]$. In this case, the battery dynamics of agent $i$ at time $t$ is expressed as 
\begin{equation*}
	I_{i,t+1} = I_{i,t} + G^+_{i,t} - G^-_{i,t} + \sum_{j \in \mc N_i} U_{ij,t} - \sum_{j \in \mc N_i} \zeta_{ji,t} + R_{i,t} ( 1 + \xi_{i,t}^p ) - D_{i,t} ( 1 + \xi_{i,t}^d) \quad  \forall \zeta_{\mathcal{N}_i,i,t} \in\underset{j\in\mathcal{N_i}}{\times} [\underline{b}_{ji,t}, \overline{b}_{ji,t}]  
\end{equation*}
where $\zeta_{\mathcal{N}_i,i,t}$ denotes all the electricity that can potentially be drawn from agent $i$ by its neighbors $\mathcal{N}_i$ at time $t$. This problem gives rise to a local information exchange problem which is an instance of Problem~\eqref{Decentralized} and can be written as follows.
\begin{equation}\label{energyhub_local}
	\begin{array}{l}
		\text{~minimize~~} \displaystyle \sum_{i=1}^M \max\limits_{\bm \xi_i \in \Xi_i, \bm \zeta_{\Ni,i} \in \mc U_{\Ni,i}} \sum_{t =1}^{T} c_{t}^+ G^+_{i,t} + \sum_{t =1}^{T} c_{t}^- G^-_{i,t} + \sum_{j \in \mc N_i} \sum_{t =1}^{T} c_{t}^u U_{ij,t} \\ 
		\left.
		\begin{array}{@{}r@{~}ll}
			\text{subject to} & \bm G_i^+ = \bm \psi_i^+(\bm I_{i}, \bm \zeta_{\Ni,i}), ~ \bm G_i^- = \bm \psi_i^-(\bm I_{i}, \bm \zeta_{\Ni,i}) \\
			& \bm U_{ij} \in \bm \psi_{ij}(\bm I_{i}, \bm \zeta_{\Ni,i}), ~ \underline{\bm{b}}_{ij},\,\overline{\bm{b}}_{ij} \in \mathbb{R}_+^{T \times P} & \forall j \in \mc N_i\\
			& \bm U_{ij} \in \mc U_{i,j} = [\underline b_{ij,1}, \,\overline b_{ij,1}] \times \dots \times [\underline b_{ij,T}, \,\overline b_{ij,T}] & \forall j \in \mc N_i \\
			& G_{i,t}^+ \geq 0, ~  G_{i,t}^- \geq 0, ~ U_{ij,t} \geq 0 & \forall j \in \mc N_i, \forall t \in \mc T \\
			& \displaystyle I_{i,t+1} = I_{i,t} + G^+_{i,t} - G^-_{i,t} + \sum_{j \in \mc N_i} U_{ij,t} - \sum_{j \in \mc N_i} \zeta_{j,i,t}
			\\[-1em]
			& & \forall t \in \mc T \\[-1em]
			& \hspace{5em} + R_{i,t} ( 1 + \xi_{i,t}^p ) + D_{i,t} ( 1 + \xi_{i,t}^d) \\
			& \displaystyle 0 \leq I_{i,t} \leq B_i & \forall t \in \mc T_+
		\end{array}
		\right \rbrace 
		\begin{array}{l}
			\forall \bm \xi_{i} \in \Xi_{i} \\
			\forall \bm \zeta_{\mc N_i, i} \in \mc U_{\mc N_i, i} \\
			\forall i \in \mc M, 
		\end{array}
	\end{array}
\end{equation}
By construction of Problem~\eqref{energyhub_local}, each set $\mathcal{U}_{i,j} = \mathcal{U}_{i,j,1}\times \cdots \times \mathcal{U}_{i,j,T}$ is a  hyper-rectangles which can be controlled coordinate wise. Hence, it can be exactly represented by the primitive sets $\mathcal S_{i,j,t} = \{ s \in \mathbb R : | s | \leq 1 \} $ and $\mc U_{i,j,t}(y_{i,j,t},z_{i,j,t}) = \displaystyle\left\{U_{ij,t} \in\mathbb{R} \,:\, \exists s_{i,j,t}  \in \mc S_{i,j,t} \text{ s.t. } U_{ij,t} =   s_{i,j,t} y_{i,j,t} + z_{i,j,t} \right\},$
as discussed in Remark~\ref{remark1}. Applying Theorem~\ref{thm::5}, Problem~\eqref{energyhub_local} can thus be reformulated as
\begin{align}
	\label{eq:local:energy}
	\begin{array}{l}
		\text{~minimize~~} \displaystyle \sum_{i=1}^M \max\limits_{\bm \xi_i \in \Xi_i, \bm s_{\Ni, i} \in \mc S_{\Ni, i}} \sum_{t =1}^{T} c_{t}^+ G^+_{i,t} + \sum_{t =1}^{T} c_{t}^- G^-_{i,t} + \sum_{j \in \mc N_i} \sum_{t =1}^{T} c_{t}^u U_{ij,t} \\ 
		\left.
		\begin{array}{@{}r@{~}ll}
			\text{subject to} & \bm G_i^+ = \bm \Gamma_i^+(\bm \xi_{i}, \bm s_{\Ni, i}), ~ \bm G_i^- = \bm \Gamma_i^-(\bm \xi_{i}, \bm s_{\Ni, i}) \\
			& \bm U_{i,j} \in \bm \Gamma_{i,j}(\bm \xi_{i}, \bm s_{\Ni, i}), ~ \underline{\bm{b}}_{i,j},\,\overline{\bm{b}}_{i,j},\, \bm y_{i,j}, \bm z_{i,j} \in \mathbb{R}_+^{T \times P} & \forall j \in \mc N_i \\
			& \underline{\bm{b}}_{i,j} = \bm z_{i,j} - \bm y_{i,j}, \, \overline{\bm{b}}_{i,j} = \bm z_{i,j} + \bm y_{i,j} & \forall j \in \mc N_i \\
			& \bm U_{i,j} \in \mc U_{i,j} = [\underline b_{i,j,1}, \,\overline b_{i,j,1}] \times \dots \times [\underline b_{i,j,T}, \,\overline b_{i,j,T}] & \forall j \in \mc N_i \\
			& G_{i,t}^+ \geq 0, ~  G_{i,t}^- \geq 0, ~ U_{i,j,t} \geq 0 & \forall j \in \mc N_i, \forall t \in \mc T \\
			& \zeta_{j,i,t} = y_{j,i,t} s_{j,i,t} + z_{j,i,t} & \forall j \in \mc N_i, \forall t \in \mc T \\
			& \displaystyle I_{i,t+1} = I_{i,t} + G^+_{i,t} - G^-_{i,t} + \sum_{j \in \mc N_i} U_{i,j,t} - \sum_{j \in \mc N_i} \zeta_{j,i,t}
			\\[-1em]
			& & \forall t \in \mc T \\[-1em]
			& \hspace{5em} + R_{i,t} ( 1 + \xi_{i,t}^p ) + D_{i,t} ( 1 + \xi_{i,t}^d) \\
			& \displaystyle 0 \leq I_{i,t} \leq B_i & \forall t \in \mc T_+
		\end{array}
		\right \rbrace 
		\begin{array}{l}
			\\ \\
			\forall \bm \xi_{i} \in \Xi_{i} \\
			\forall \bm s_{\mc N_i, i} \in \mc S_{\mc N_i, i} \\
			\forall i \in \mc M.
		\end{array} \hspace{-3em}
	\end{array}
\end{align}

In the sequel we use the open-source dataset from~\cite{bergmeir_2023_8219786} comprising residential power and battery data at minute resolution. This anonymized dataset contains real-world prosumer information on energy usage, photovoltaic power generation, and battery characteristics. Figure~\ref{fig::data} depicts the hourly power consumption and generation of the first prosumer in this dataset from July 1, 2021, to October 1, 2021, using lines with low opacity. Additionally, the solid lines represent the average data, while the dashed lines represent observations that are one standard deviation away from the mean. These intervals are used to calibrate the uncertainty set for each prosumer. We assume that the initial charge of all batteries is zero. The capacity of batteries for each prosumer is also available in the dataset. For example, the capacity is $B_1 = 13$ for the first prosumer.  Since the data is anonymized and we lack access to the price data, we set the average hourly electricity price~as
\begin{align*}
	c_t = 18 - \tanh (t) + \tanh (t - 4) - 2 \tanh (t - 6) + 4 \tanh (t - 17) - 4 \tanh(t - 24),
\end{align*}
following the typical electricity price pattern, where prices are higher at the beginning and end of the day~\citep{California}. Problems~\eqref{eq:central:energy} and \eqref{eq:local:energy} are multistage robust linear programs which we approximate with affine decision rules.  

\begin{figure*}[!tb]
	\center
	\includegraphics[width = 0.3\textwidth]{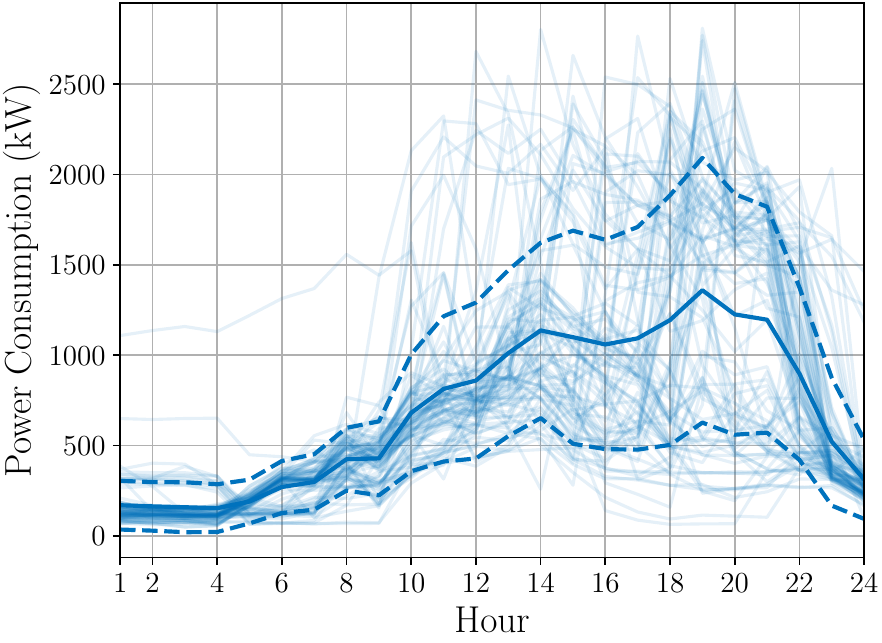}
	\includegraphics[width = 0.3\textwidth]{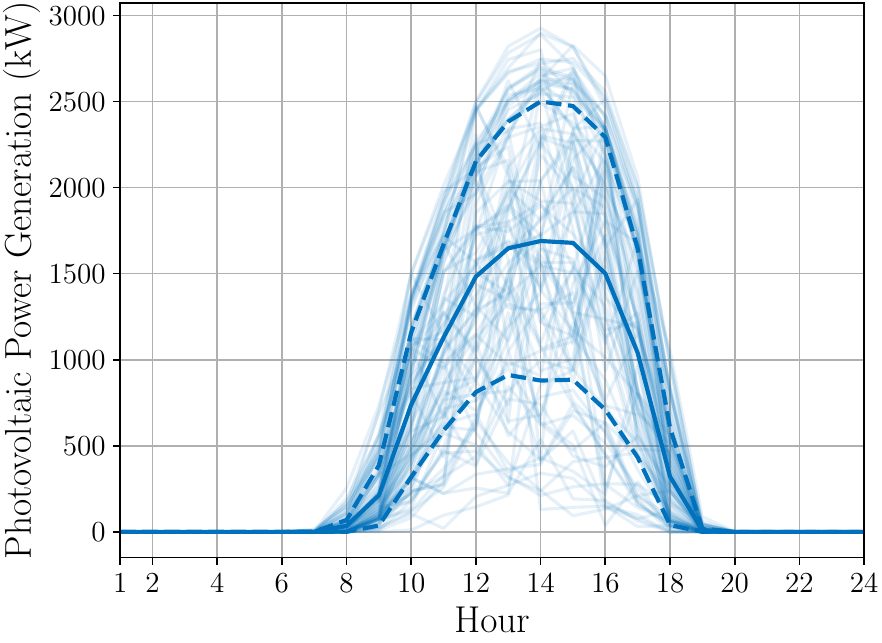} 
	\includegraphics[width = 0.29\textwidth]{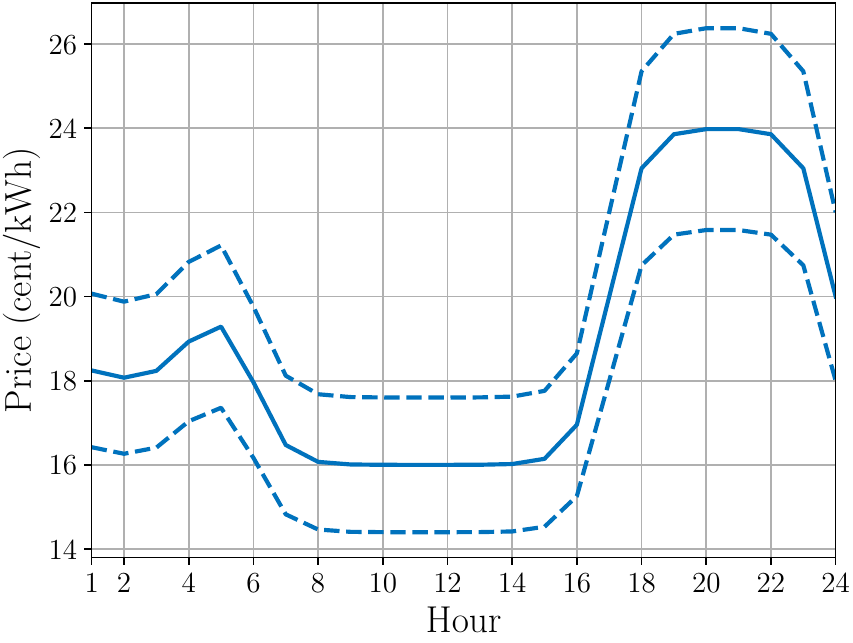}
	\caption{Illustration of power consumption, photovoltaic power generation, and electricity price for the first prosumer.}
	\label{fig::data}
\end{figure*}

We conduct $10$ independent simulation. For each simulation, we pick uniformly at random $\varepsilon_t \in [-0.1, 0.1]$ and we set the market hourly price to $c_t^+ = (1 + \varepsilon_t) c_t$ for all $t \in \mc T$, i.e., the market prices are allowed to differ by up to $10\%$ from the average price, see Figure~\ref{fig::data} (right). 
We also set $c_t^- = 0.5 c_t^+$ and $c_t^u = 0.2 c_t^+$.  
We consider bi-hourly time steps, meaning that every prosumer is allowed to request or provide electricity to others every $2$ hours.
The electricity price for each $2$-hour interval is determined by the average price, while the power consumption and generation are the sum of their respective values.
In other words, we represent a day with a horizon length of $T=12$. As baseline, we also examine a system in which all prosumers try to minimize their worst-case electricity costs individually, and we refer to it as the \emph{decoupled problem}. This can be obtained by including the constraint $U_{ij,t} = 0$ for all $t \in \mc T, j \in \mc N_i, i \in \mc M$ in~\eqref{eq:central:energy}. We conduct simulations on serial and complete networks with an increasing number of $M$ prosumers, using the data of the first prosumers in \cite{bergmeir_2023_8219786} which have both demand and photovoltaic data available.\footnote{In particular, we use the prosumers with identification numbers $1, 2, 3, 9, 10, 11$.}

\begin{figure*}[!tb]
	\center
	\includegraphics[width = \textwidth]{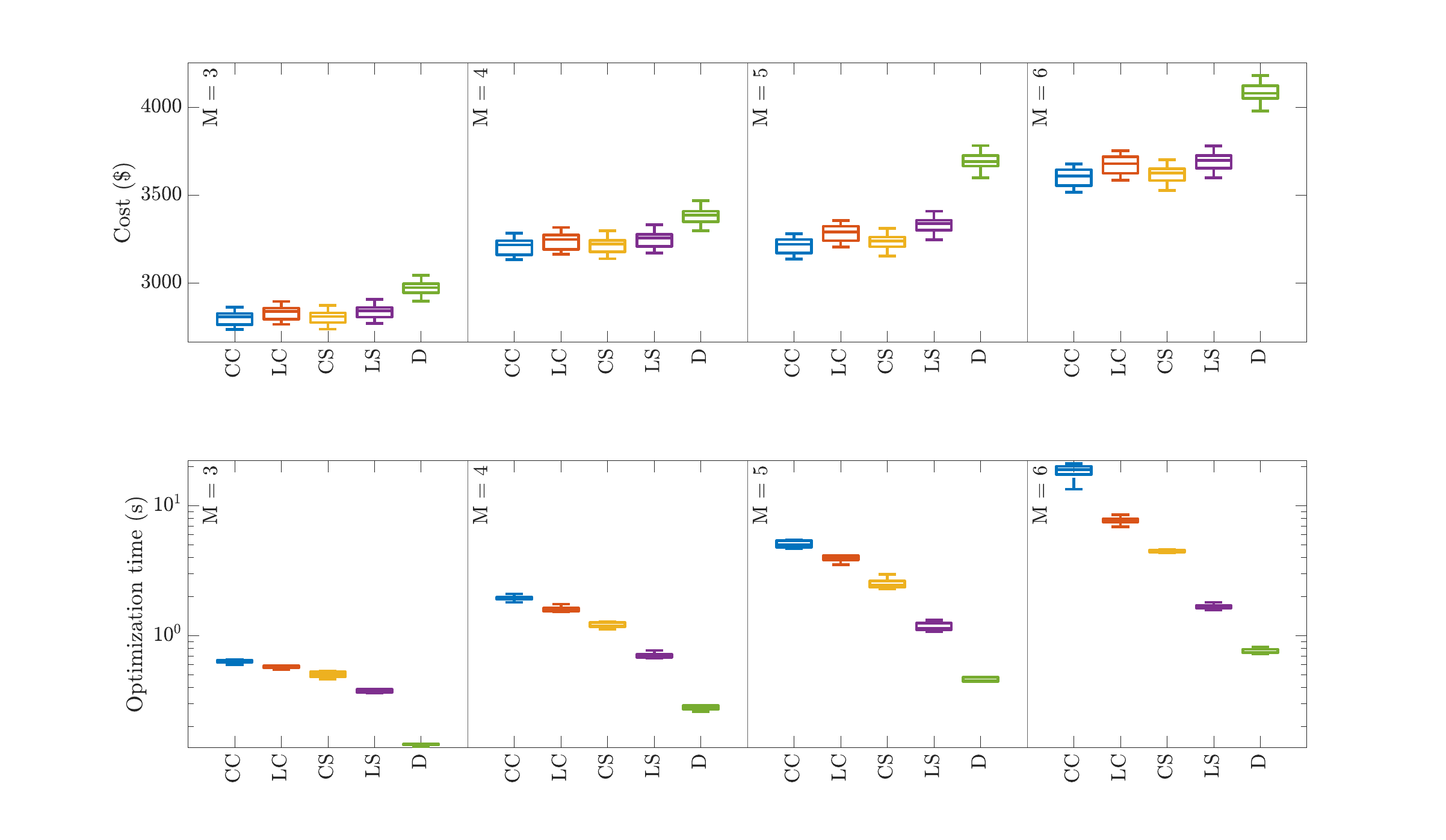}
	\caption{Comparison of the cost and optimization time needed to solve the centralized, local, and decoupled information problems for serial and complete network. The result is based on $10$ independent simulations. (CC) Centralized Complete, (CS) Centralized Serial, (LC) Local Complete, (LS) Local Serial, (D) Decoupled.}
	\label{fig::energy-execution}
\end{figure*}


The first experiment investigates the impact of the number of prosumer $M$ in the system on the worst-case cost for centralized and local information, and decoupled  problems. Figure~\ref{fig::energy-execution} presents the results for the serial and complete networks across $10$ independent simulation runs. Observe that the optimal worst-case cost for the decoupled problem is significantly higher than that of the centralized and local information exchange problems. This is expected because in the decoupled setting, agents do not benefit from sharing electricity among themselves. Furthermore, the optimal worst-case cost for the local information exchange is higher than that for the centralized approach, as indicated by Theorems~\ref{thm::1}, \ref{thm::3}, and Corollary~\ref{corr::1}. Additionally, the structure of a complete network positively influences the worst-case cost outcomes when compared to a serial network configuration.


Furthermore, our findings suggest that the restrictive nature of our proposed local information exchange has only a minimal impact on the performance of the control policy, as demonstrated in Figure~\ref{fig::energy-execution} (top). Specifically, the performance gap in terms of the worst-case cost between the local and central controllers is approximately 2\% on average. Additionally, when compared to the decoupled problems, the centralized controller reduces costs by 15\%, while the localized controller achieves a cost reduction of 13\% on average.



We also observe that the time required to solve the local information exchange problem is less than the time needed for the centralized optimization problem, and the computational gains are more pronounced as the number of residential prosumers increase, as illustrated in Figure~\ref{fig::energy-execution} (bottom). From a computational standpoint, the optimization times for the complete network are considerably longer than those for the serial network. This is expected due to the quadratic increase in the number of communication links with respect to the prosumers in the complete network compared to the linear increase in the serial network. Specifically, optimization problems for the local controller are solved approximately 2.5 times faster in the largest instances of the complete network structure, while in the serial network, this ratio is about 3 times faster. Moreover, the optimization time required for the local information problem is approximately 5 times faster for the largest instance in the serial network compared to the complete network. As anticipated, the computational effort increases with the increasing number of building units.

We conclude our numerical experiment by comparing the average performance of the centralized and local information problems using a rolling horizon scheme when $\varepsilon_t = 0$ for all $t$ and the network structure is serial. In this experiment, we fix all parameters as in the first experiment. We then generate a random realization of the uncertainty $\bm \xi\in\Xi$. Next, we solve the centralized, local, and decoupled information for the horizon length $T=12$ and update the battery levels  according to the realization $\bm \xi_1$ and the first stage decisions. We repeat the process until we reach the end of the horizon. Specifically, at any time $t =2,\ldots,T$, we resolve each  problem for the shorter horizon length $T-t$ and the initial battery level $I_{i,t-1}$ for every $i \in \mathcal M$. We then update the battery levels according to realization $\bm \xi_t$ and corresponding first stage decisions. Figure~\ref{fig::energy-rolling} summarizes our results for $10$ randomly generated realization of uncertainty for a serial network, where the worst-case cost of corresponding problem with horizon $T=12$ is denoted by the marker~$*$. 
We observe that the average performances of the local and centralized information problems are significantly improved compared to their worst-case performance. The improvement is approximately around half (2 times smaller). This is extremely valuable as in reality this problem will be repeatedly solved as time progresses. Moreover, the performance of both the local and centralized information controllers remains superior to that of the decoupled controller. We exclude the result for the complete network structure as it exhibits the same behavior.

\begin{figure*}[!tb]
	\center
	\includegraphics[width = \textwidth]{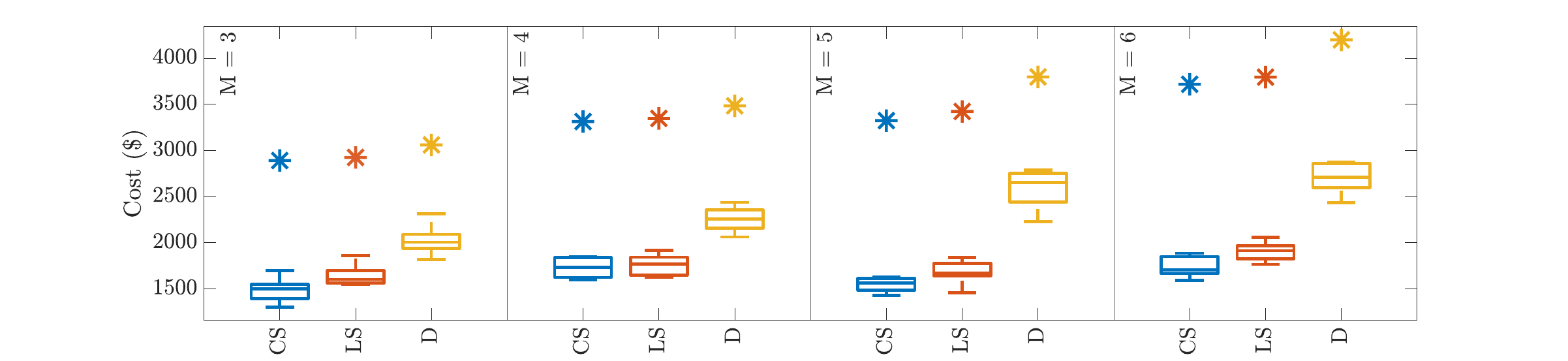}
	\caption{Comparison of the rolling horizon cost of the  centralized and local  information problems for a serial network. The result is based on $10$ independent simulations. (CS) Centralized Serial, (LS) Local Serial, (D) Decoupled.}
	\label{fig::energy-rolling}
\end{figure*} 

\paragraph{Acknowledgements.} This research was supported by the Swiss National Science Foundation under the NCCR Automation (grant agreement 51NF40\_180545) and under an Early Postdoc.Mobility Fellowship awarded to the third author (grant agreement P2ELP2\_195149).

\bibliographystyle{abbrvnat}
\bibliography{ref}

\newpage
\appendix

\section{Proofs}
\label{app::proof}
\setcounter{equation}{0}
\renewcommand{\theequation}{\Alph{section}.\arabic{equation}}

\begin{proof}[\large \bf Proof of Theorem \ref{thm::1}]
	
	We show that every feasible solution of Problem~\eqref{Semi-Centralizedb} is a feasible solution to Problem~\eqref{Centralizedb}. Let $ \bm \Phi_i $ for all $ i \in \mc M $ be any feasible policy in Problem~\eqref{Semi-Centralizedb}. Starting with $ \chi_{i,1} = x_{i,1}$, the state of agent $ i $ at time $ t $ are given as
	\begin{equation}\label{eq::thm1::eq1}
		\begin{array}{@{}r@{~}l}
			x_{i,t} &=\displaystyle A_{i,1}^t x_{i,1} + \sum_{\tau=1}^{t-1} \Big(A_{i,\tau+1}^t B_{i,\tau} [\chi_{j,\tau}(\bm \xi_{\oNj}^{\tau-1})]_{j\in \mc N_i} + A_{i,\tau+1}^t D_{i,\tau} \Phi_{i,\tau}(\bm \xi_{\oNi}^{\tau-1}) + A_{i,\tau+1}^t E_{i,\tau} \xi_{i,\tau} \Big) \\
			&=: \chi_{i,t}(\bm \xi_{\oNi}^{t-1})
		\end{array}
	\end{equation}
	where $ A_{i,\tau}^t = A_{i,\tau}A_{i,\tau+1}\dots A_{i,t-1} $ for $ \tau<t $ and $ A_{i,t}^t = I $. The last implication follows from the fact that $ \ol{\mc N}_i \supseteq \ol{\mc N}_j $ for all $ j \in \mc N_i $ since the network admits a partially nested structure. Given \eqref{eq::thm1::eq1}, it is easy to verify that for each agent $ i $ its dynamics and constraints in Problem~\eqref{Semi-Centralizedb} only depend on $ \bm \xi_{\oNi} $. Hence, any feasible solution to Problem~\eqref{Semi-Centralizedb} is also feasible to the following optimization problem:
	\begin{equation}\label{eq::thm1::eq2}
		\begin{array}{l}
			\text{minimize}  \;\; \displaystyle\sum\limits_{i = 1}^M \max\limits_{\bm \xi_\Mu \in \Xi} J_i(\bm x_i,\bm u_i) \\
			\!\!\!\left.\begin{array}{r@{~}l@{}}
				\text{subject to}
				&  \bm u_i =\bm \Phi_i(\bm \xi_{\oNi})  := [\Phi_{i,t}(\bm \xi_\oNi^{t-1})]_{t\in\mc T} \\
				& \bm x_{i} = f_{i}\big(\bm x_{\Ni}, \bm u_{i}, \bm \xi_{i}\big)\\
				& (\bm x_{i}, \bm u_{i}) \in \mc O_i
			\end{array}\right \rbrace \forall \bm \xi_\Mu \in \Xi_\Mu,\;
			\forall i \in \mc M
		\end{array}
	\end{equation}
	Additionally, they achieve the same objective value since they share the same objective function. This shows equivalence of Problem~\eqref{Semi-Centralizedb} and Problem \eqref{eq::thm1::eq2}. The relation between Problem~\eqref{Centralizedb} and Problem~\eqref{Semi-Centralizedb} stated in the theorem now follows immediately since the two problems share the same constraints and objective function, and the policies in Problem~\eqref{Semi-Centralizedb} are restricted compared to Problem~\eqref{Centralizedb} as they are functions of $\bm \xi_{\oNi}$ while the policies in Problem~\eqref{Centralizedb} are functions of~$\bm \xi_{\mc M}$.
\end{proof}

\begin{proof}[\large \bf Proof of Theorem \ref{thm::2}]
	
	The statement is proved by induction using similar theoretical tools to \cite[Prop. 2.1]{Hadjiyiannis2011}. The statement holds for $ t = 1 $ since the initial state, $ x_{i,1} $, is known for every $ i \in \mc M $; therefore, functions $ \psi_{i,1}(x_{i,1}, \bm \zeta_{\Ni,1}) $ and $ \Psi_{i,1}(\bm \zeta_{\Ni,1}) $ can always be constructed such
	\begin{equation}
		\psi_{i,1}(x_{i,1}, \bm \zeta_{\Ni,1}) = \Psi_{i,1}(\bm \zeta_{\Ni,1})
	\end{equation}
	
	Assume now that the statement holds for all $ 1 < \tau \leq t-1 $, i.e., there exist policies $ \psi_i(\cdot) $ and $ \Psi_i(\cdot) $ such that $ \psi_{i,\tau}(\bm x_i^{\tau}, \bm \zeta_\Ni^{\tau}) = \Psi_{i,\tau}(\bm \xi_i^{\tau-1}, \bm \zeta_\Ni^{\tau}) $. In the sequel, we show that the statement also holds for $ \tau = t $. From \eqref{eq::stateDynamics}, we have that
	\begin{equation}\label{eqqv1}
		\begin{array}{r@{}l}
			x_{i,t} &=\displaystyle A_{i,1}^t x_{i,1} + \sum_{\tau=1}^{t-1} \Big( A_{i,\tau+1}^t B_{i,\tau} \bm \zeta_{\Ni,\tau} + A_{i,\tau+1}^t D_{i,\tau} \Psi_{i,\tau}(\bm \xi_i^{\tau-1}, \bm \zeta_\Ni^\tau) +  A_{i,\tau+1}^t E_{i,\tau} \xi_{i,\tau} \Big) \\
			&=: \chi_{i,t}(\bm \xi_{i}^{t-1}, \bm \zeta_{\Ni}^{t-1}),
		\end{array}
	\end{equation}
	where $ A_{i,\tau}^t = A_{i,\tau}A_{i,\tau+1}\dots A_{i,t-1} $ for $ \tau<t $ and $ A_{i,t}^t = I $. Moreover, it holds that
	\begin{equation}\label{eqq1v1}
		\begin{array}{r@{}l}
			\xi_{i,t-1} &= E_{i,t-1}^{+}\big(x_{i,t} - A_{i,t-1} x_{i,t-1} + B_{j,t-1} \bm \zeta_{\Ni,t-1} + D_{i,t-1} \psi_{i,t-1}(\bm x_i^{t-1}, \bm \zeta_\Ni^{t-1})\big) \\
			&=: \rho_{i,t}(\bm x_{i}^{t}, \bm \zeta_\Ni^{t-1}),
		\end{array}
	\end{equation}
	where $ E^+_{i,t} := (E_{i,t}^\top E_{i,t})^{-1}E^\top_{i,t} $ is the left inverse of $E_{i,t}$ since it is full rank.
	
	The relation \eqref{eqqv1} implies that given a feasible policy $ \psi_{i,t}(\cdot) $ for Problem~\eqref{Decentralized}, we can construct a feasible policy for Problem~\eqref{Decentralizedb} as
	\begin{equation}\label{eqq2v1}
		\psi_{i,t}(\bm x_i^t, \bm \zeta_\Ni^t) = \psi_{i,t}(\chi_i^t(\bm \xi_{i}^{t-1},\bm \zeta_{\Ni}^t), \bm \zeta_\Ni^t):= \Psi_{i,t}(\bm \xi_i^{t-1}, \bm \zeta_\Ni^t).
	\end{equation}
	The claim follows from the fact that the composition of continuous differentiable functions is a continuous differentiable function. Hence, the policy $ \psi_{i,t}(\cdot) $ will also be feasible in Problem~\eqref{Decentralizedb} since the two problems have the same pointwise constraints. Additionally, they achieve the same objective value since they share the same objective function. 
	
	Similarly, the relation \eqref{eqq1v1} implies that given a feasible policy $ \Psi_{i}(\cdot) $ for Problem~\eqref{Decentralizedb}, we can construct a feasible policy for Problem~\eqref{Decentralized} as
	\begin{equation}\label{eqq3v1}
		\Psi_{i,t}(\bm \xi_i^{t-1}, \bm \zeta_\Ni^t) = \Psi_{i,t}(\rho_i^t(\bm x_{i}^{t},\bm \zeta_\Ni^{t-1}), \bm \zeta_\Ni^t) := \psi_{i,t}(\bm x_i^t, \bm \zeta_\Ni^t).
	\end{equation}
	The claim follows from the fact that the composition of continuous differentiable functions is a continuous differentiable function. Hence, the policy $ \Psi_{i,t}(\cdot) $ is also feasible in Problem~\eqref{Decentralized} since the two problems have the same pointwise constraints. Additionally, they achieve the same objective value since they share the same objective function. 
\end{proof}

\begin{proof}[\large \bf Proof of Theorem \ref{thm::3}]	
	We show that every feasible solution of Problem~\eqref{Decentralizedb} is feasible in Problem~\eqref{Semi-Centralizedb}. Let $ (\bm \Psi_i, \mc X_i) $ for all $ i \in \mc M $ be feasible in Problem~\eqref{Decentralizedb}. Since the state of agent $i$ evolve according to \eqref{eq::stateDynamics}, we can conclude that at time $t$ we have	 
	\begin{equation}
		\begin{array}{r@{}l}
			x_{i,t} &=\displaystyle A_{i,1}^t x_{i,1} + \sum_{\tau=1}^{t-1} \Big( A_{i,\tau+1}^t B_{i,\tau} \bm \zeta_{\Ni,\tau} + A_{i,\tau+1}^t D_{i,\tau} \Psi_{i,\tau}(\bm \xi_i^{\tau-1}, \bm \zeta_\Ni^\tau) + A_{i,\tau+1}^t E_{i,\tau} \xi_{i,\tau}  \Big)\\
			&=: \chi_{i,t}(\bm \xi_{i}^{t-1}, \bm \zeta_{\Ni}^{t-1}),
		\end{array}
	\end{equation}
	where where $ A_{i,\tau}^t = A_{i,\tau}A_{i,\tau+1}\dots A_{i,t-1} $ for $ \tau<t $ and $ A_{i,t}^t = I $.
	
	We note that $ \bs\chi_{i}(\bm \xi_{i},\bm \zeta_{\Ni}) = [\chi_{i,t}(\bm \xi_{i}^{t-1}, \bm \zeta_{\Ni}^{t-1})]_{t \in \mc T_+} \in \mc X_i $ for all $ \bm \xi_i \in \Xi_i $ and $ \bm \zeta_{\Ni} \in \mc X_\Ni $ due to the feasibility of Problem~\eqref{Decentralizedb}. To show that $ \bm \Psi_i $ is feasible in Problem~\eqref{Semi-Centralizedb}, we first construct the state of agent $i$ which evolves according to $\bm x_{i} = f_{i}\big(\bm  x_{\Ni}, \bm \Psi_{i}(\bs{\xi}_i, \bm \zeta_\Ni), \bm \xi_{i}\big)$. Starting with $\hat \chi_{i,1} = x_{i,1}$, we have that
	\begin{equation}\label{eq::c1}
		\begin{array}{r@{}l}
			&x_{i,t} =\displaystyle A_{i,1}^t x_{i,1} + \sum_{\tau=1}^{t-1} \Big(A_{i,\tau+1}^t B_{i,\tau} [\widehat \chi_{j,\tau}(\bm \xi_{\oNj}^{\tau-1}]_{j\in \mc N_i} + A_{i,\tau+1}^t D_{i,\tau} \Psi_{i,\tau}(\bm \xi_i^{\tau-1}, \bm \zeta_{\Ni}^{\tau}) + A_{i,\tau+1}^t E_{i,\tau} \xi_{i,\tau} \Big)\\
			&=\displaystyle A_{i,1}^t x_{i,1} + \sum_{\tau=1}^{t-1} \Big(A_{i,\tau+1}^t B_{i,\tau} [\widehat \chi_{j,\tau}(\bm \xi_{\oNj}^{\tau-1}]_{j\in \mc N_i} + A_{i,\tau+1}^t D_{i,\tau} \Psi_{i,\tau}(\bm \xi_i^{\tau-1}, [\widehat{\bm \chi}_{j}^\tau (\bm \xi_{\oNj}^{\tau-1}]_{j\in \mc N_i}) + A_{i,\tau+1}^t E_{i,\tau} \xi_{i,\tau} \Big)\\
			& = \chi_{i,t}(\bm \xi_i^{t-1},[\widehat \chi^{t-1}_{j}(\bm \xi_{\oNj}^{t-2})]_{j\in \mc N_i})\\
			& =: \widehat\chi_{i,t}(\bm \xi_{\oNi}^{t-1}).
		\end{array}
	\end{equation}
	where the implication follows because $\bs{\widehat\chi}_{i}([\bm \xi_{j}^{t-1}]_{j\in \oNi}) = [\widehat\chi_{i,t}(\bm \xi_{\oNi}^{t-1})]_{t \in \mc T_+}\in\mc X_i$ for all $ \bm \xi_\oNi \in \Xi_\oNi$ and $ i \in \mc M $.
	For each $i\in \mc M$, we consider the decision $\bs{\hat \Phi}_i(\bm \xi_\oNi)$ 
	defined through
	\begin{equation}\label{eq::c2}
		\widehat \Phi_{i,t}(\bm \xi_{\oNi}^{t-1}):= \Psi_{i,t}(\bm \xi_i^{t-1},[\widehat \chi^{t}_{j}(\bm \xi_{\oNj}^{t-1})]_{j\in \mc N_i})
	\end{equation}
	Notice that \eqref{eq::c2} defines a valid policy construction since $\bs{\widehat\chi}_{i}([\bm \xi_{j}^{t-1}]_{j\in \oNi})\in\mc X_i$ for all $ \bm \xi_\oNi \in \Xi_\oNi$. It remains to show that $\bm \Psi_{i}$ is feasible also for the constraints of Problem~\eqref{Semi-Centralizedb}. We do so using deduction, as follows:  
	\begin{equation}
		\begin{array}{rll}
			& \big(\bs\chi_{i}(\bm \xi_{i},\bm \zeta_{\Ni}), \bm \Psi_{i}(\bs{\xi}_i,\bm \zeta_{\Ni})\big) \in \mc O_i, & \forall \bm \xi_i \in {\Xi}_i, \forall \bm \zeta_\Ni \in \mc X_\Ni,\\
			\implies & \big(\bs\chi_{i}(\bm \xi_{i},[\bs{\widehat\chi}_{j}(\bm \xi_\oNj)]_{j \in \mc N_i}), \bm \Psi_{i}(\bs{\xi}_i,[\bs{\widehat\chi}_{j}(\bm \xi_\oNj)]_{j \in \mc N_i})\big)  \in \mc O_i, & \forall \bm \xi_\oNi \in {\Xi}_\oNi, \\
			\implies & \big(\bs{\widehat\chi}_{i}(\bm \xi_\oNi), \bs{\widehat\Phi}_{i}(\bm \xi_\oNi)\big)  \in \mc O_i, & \forall \bm \xi_\oNi \in {\Xi}_\oNi,
		\end{array}
	\end{equation}
	The first implication follows from \eqref{eq::c1} and the fact that $\bs{\widehat\chi}_{i}(\bm \xi_\oNi) \subseteq \mc X_i$ for all $\bs{\xi}_\oNi\in\Xi_\oNi$, while the second implication follows from \eqref{eq::c1} and \eqref{eq::c2}.
	Finally, this feasible solution attains a value for the objective function of Problem~\eqref{Decentralizedb} which is equal or larger than the value attained for the objective function of Problem~\eqref{Semi-Centralizedb}, that is
	\begin{equation*}
		\begin{array}{@{\,}r@{\,}l@{\,}}
			&\sum\limits_{i = 1}^M \max\limits_{\bm \xi_i \in \Xi_i, \bm \zeta_{\Ni} \in \mc X_{\Ni}} 
			J_i\big(\bs\chi_{i}(\bm \xi_{i},\bm \zeta_{\Ni}), \bm \Psi_{i}(\bs{\xi}_i,\bm \zeta_{\Ni})\big)\\
			\geq& \sum\limits_{i = 1}^M \max\limits_{\bm \xi_\oNi \in \Xi_\oNi}\; J_i\big(\bs\chi_{i}(\bm \xi_{i},[\bs{\widehat\chi}_{j}(\bm \xi_\oNj)]_{j \in \mc N_i}), \bm \Psi_{i}(\bs{\xi}_i,[\bs{\widehat\chi}_{j}(\bm \xi_\oNj)]_{j \in \mc N_i})\big) \\
			=& \sum\limits_{i = 1}^M \max\limits_{\bm \xi_\oNi \in \Xi_\oNi}\; J_i\big(\bs{\widehat\chi}_{i}(\bm \xi_\oNi), \bs{\widehat\Phi}_{i}(\bm \xi_\oNi)\big),
		\end{array}
	\end{equation*}
	where again the first implication follows from \eqref{eq::c1} and the fact that $\bs{\widehat\chi}_{i}(\bm \xi_\oNi) \subseteq \mc X_i$ for all $\bs{\xi}_\oNi\in\Xi_\oNi$, while the second implication follows from \eqref{eq::c1} and \eqref{eq::c2}.
\end{proof}

\begin{proof}[\large \bf Proof of Proposition~\ref{prop::1}]
	Theorem~\ref{thm::3} already established that a feasible solution of Problem~\eqref{Decentralizedb} is feasible in Problem~\eqref{Semi-Centralizedb}.  We will now show that a feasible solution in Problem~\eqref{Semi-Centralized} is feasible in  Problem~\eqref{Decentralized}. The result will then follow due to the relation between Problems~\eqref{Semi-Centralized} and \eqref{Semi-Centralizedb}, see \citep{Lin2016}, and  Problems~\eqref{Decentralized} to \eqref{Decentralizedb}, see Theorem~\ref{thm::2}.
	
	In a directed acyclic bipartite network, the agents can be split in two groups, the first layer agents where $\ol{\mc N}_i = \{i\}$ and the second layer agents where  $\ol{\mc N}_i = \{\mathcal{N}_i,i\}$. Sets $\mathcal{FL}$ and $\mathcal{SL}$ denote the first and second layer agents, respectively. Problem~\eqref{Decentralized} can then be explicitly written as  
	\begin{equation}\label{Decentralized_bipartite}
		\begin{array}{l}
			\text{~minimize~~} \displaystyle\sum\limits_{i \in \mc{FL}}  \max\limits_{\bm \xi_i \in \Xi_i} J_i(\bm x_i,\bm u_i)+\sum\limits_{i \in\mc{SL}} \max\limits_{\bm \xi_i \in \Xi_i, \bm \zeta_{\Ni} \in \mc X_{\Ni}} J_i(\bm x_i,\bm u_i) \\
			\left.\begin{array}{@{}r@{~}l@{}}
				\text{subject to }
				& \bm u_i = \bm {\psi}_{i}(\bm x_i) := [\psi_{i,t}(\bm x_i^t)]_{t \in \mc T}\\
				& \bm x_{i} = f_{i}\big(\bm u_{i}, \bm \xi_{i}\big)\\
				& (\bm x_{i}, \bm u_{i}) \in \mc O_i \\
				& \bm x_i \in \mc X_i, \; \mc X_i \in \mc{F}(\mc S_i)
			\end{array}\right \rbrace \forall \bm \xi_i \in \Xi_i,\;\forall i \in \mc{FL},\\
			\left.\begin{array}{@{}r@{~}l@{}}
				\phantom{\text{subject to }}
				& \bm u_i = \bm {\psi}_{i}(\bm x_i, \bm \zeta_{\Ni}) := [\psi_{i,t}(\bm x_i^t, \bm \zeta_\Ni^t)]_{t \in \mc T}\\
				& \bm x_{i} = f_{i}\big(\bm \zeta_{\Ni}, \bm u_{i}, \bm \xi_{i}\big)\\
				& (\bm x_{i}, \bm u_{i}) \in \mc O_i \\
			\end{array}\right \rbrace \begin{array}{@{}l}
				\forall \bm \zeta_{\Ni} \in \mc X_{\Ni}\\
				\forall \bm \xi_{\mc M} \in \Xi_{\mc M}
			\end{array}\;\forall i \in \mc{SL}, 
		\end{array}
	\end{equation}
	where the  constraints distinguish between  first and second layer agents. Notice that agents in the first layer are the ones communicating their state forecast sets $\mathcal{X}_i$ to their neighbors (constraint $\bm x_i\in\mathcal{X}_i$), while the agents on the second layer only receive these sets ($\bm x_{i} = f_{i}\big(\bm \zeta_{\Ni}, \bm u_{i}, \bm \xi_{i}\big)$, $\forall \bm \zeta_{\Ni} \in \mc X_{\Ni}$).
	
	Let $\phi_i$ for all $i\in\mathcal{M}$ be a feasible policy for Problem~\eqref{Semi-Centralized}. It is easy to see that policy is feasible in $(f_{i}\big(\phi_{i}(\bm x_i), \bm \xi_{i}\big), \bm u_{i}) \in \mc O_i\quad \forall \bm \xi_i\in\Xi_i,\,i\in\mc{FL}$ as Problems~\eqref{Semi-Centralized} and  \eqref{Decentralized} share the same constraints. We set
	\begin{equation}\label{def:information_set_proof}
		\mc X_i:=\{\bm x_i\in\mb R^{N_{x,i}}\;:\; \bm x_{i} = f_{i}\big(\phi_{i}(\bm x_i), \bm \xi_{i}\big)\;\;\forall \bm \xi_i \in \Xi_i\}\quad \forall i \in \mc{FL},
	\end{equation}
	which satisfies $\mc X_i\in\mc F(\mc S_i)$  where $\mc S_i=\mb R^{N_{x,i}}$ for all $i \in \mc{FL}$.
	By the feasibility of $\phi_i$ in Problem~\eqref{Semi-Centralized}, we have
	\begin{equation*}
		\begin{array}{lll}
			
			&\left.\begin{array}{l}
				\bm x_{j} = f_{j}\big(\phi_{j}(\bm x_j), \bm \xi_{j}\big),\;\forall \,j\in\mc{N}_i\\
				(f_i(\bm x_{\mathcal{N}_i},\phi_i(\bm x_i,\bm x_{\mathcal{N}_i}),\bm \xi_i),\phi_i(\bm x_i,\bm x_{\mathcal{N}_i}))\in\mc O_i
			\end{array}\right\}& \forall \bm \xi_\oNi \in \Xi_{\oNi},\;\forall i \in \mc{SL},\\

			\iff&\left.\begin{array}{l}
				(f_i(\bm x_{\mathcal{N}_i},\phi_i(\bm x_i,\bm x_{\mathcal{N}_i}),\bm \xi_i),\phi_i(\bm x_i,\bm x_{\mathcal{N}_i}))\in\mc O_i
			\end{array}\right\}& 
			\begin{array}{l}
				\forall  [\bm x_{j}]_{j\in\mc N_i} = [f_{j}\big(\phi_{j}(\bm x_j), \bm \xi_{j}\big)]_{j\in\mc N_i}\\\forall \bm \xi_\oNi \in \Xi_{\oNi} 
			\end{array},\,\forall i \in \mc{SL},\\
			
			\iff&\left.\begin{array}{l}
				(f_i(\bm \zeta_{\mathcal{N}_i},\phi_i(\bm x_i,\bm \zeta_{\mathcal{N}_i}),\bm \xi_i),\phi_i(\bm x_i,\bm \zeta_{\mathcal{N}_i}))\in\mc O_i
			\end{array}\right\}& 
			\begin{array}{l}
				\forall \bm \zeta_{\Ni} \in \mc X_{\Ni}\\\forall \bm \xi_{\mc M} \in \Xi_{\mc M} 
			\end{array},\,\forall i \in \mc{SL},\\
			
		\end{array}
	\end{equation*}
	where the first equivalence hold since for fixed policy $\phi_i$ the dynamics of all agents in the first layer are uniquely determined by the uncertain parameters $\bm \xi_{i}\in\Xi_i$, $i\in\mc{FL}$, while the second equivalence holds by the definition of $\mc X_i$ in \eqref{def:information_set_proof}. Thus, $\phi_i$ is feasible in  Problem~\eqref{Decentralized}. 
	
	We next show that $\phi_i$ achieves the same optimal value in both problems. 
	\begin{equation}\label{cor:obj}
		\begin{array}{ll}
			&\displaystyle\sum\limits_{i = 1}^M \max\limits_{\bm \xi_i \in \Xi_\oNi} J_i(\bm x_i,\bm \phi_{i}(\bm x_\oNi)) \\

			=&\displaystyle\sum\limits_{i\in\mc{FL}} \max\limits_{\bm \xi_i\in \Xi_i} J_i(\bm x_i,\bm \phi_{i}(\bm x_i)) 
			
			+ \sum\limits_{i\in\mc{SL}} \max\limits_{\bm \xi_\oNi \in \Xi_\oNi} J_i(\bm x_i,\bm \phi_{i}(\bm x_i,{\bm x}_{{\mc N}_i})) \\

			=&\displaystyle\sum\limits_{i\in\mc{FL}} \max\limits_{\bm \xi_i\in \Xi_i} J_i(\bm x_i,\bm \phi_{i}(\bm x_i)) 
			
			+ \sum\limits_{i\in\mc{SL}} \max\limits_{\bm \xi_\oNi \in \Xi_\oNi,\, [\bm x_{j}]_{j\in\mc N_i} = [f_{j}\big(\phi_{j}(\bm x_j), \bm \xi_{j}\big)]_{j\in\mc N_i}} J_i(\bm x_i,\bm \phi_{i}(\bm x_i,{\bm x}_{{\mc N}_i}))\\

			=&\displaystyle\sum\limits_{i\in\mc{FL}} \max\limits_{\bm \xi_i\in \Xi_i} J_i(\bm x_i,\bm \phi_{i}(\bm x_i)) 
			
			+ \sum\limits_{i\in\mc{SL}} \max\limits_{\bm \xi_i \in \Xi_i,\,\bm \zeta_{\Ni} \in \mc X_{\Ni}} J_i(\bm x_i,\bm \phi_{i}(\bm x_i,{\bm \zeta}_{{\mc N}_i}))\\
		\end{array}
	\end{equation}
	Starting from the objective of  Problem~\eqref{Semi-Centralized}, the first equivalence rewrites the objective function it in terms of the first and second layer agents. Since for fixed $\phi_i$ the dynamics of the first layer agents  are uniquely determined by the uncertain parameters $\bm \xi_{i}\in\Xi_i$, $i\in\mc{FL}$, the second equality re-expresses the dynamics of ${\bm x}_{{\mc N}_i}$ in terms of $\bm \xi_{i}\in\Xi_i$, $i\in\mc{FL}$. Finally, in the third equality we substitute the definition of $\mc X_i$ in \eqref{def:information_set_proof}. Notice that the uncertain parameters governing $\mc X_i$ are those affective first layer agents, thus making $\mc X_{\mc{FL}}$ and $\Xi_{\mc{SL}}$ rectangular. The last expression in \eqref{cor:obj} coincides with the objective of Problem~\eqref{Decentralized}.
\end{proof}

\begin{proof}[\large \bf Proof of Proposition \ref{prop::nCc}]
	The recession cone of the set $ \mc S_i $ is defined as $ \textrm{recc}(\mc S_i) = \{\bm \nu_i \in \mb R^{n_x^i}: \bm s_i + \lambda \bm \nu_i \in \mc S_i, \forall \bm s_i \in \mc S_i, \,\lambda \ge 0 \} $. The fact that $ \mc S_i $ is bounded implies that the recession cone of $ \mc S_i $ is empty, i.e., $ \textrm{recc}(\mc S_i) = \{0\} $. We now show that,
	\begin{equation*}
		\mc X_{\FS} = \Big\{(\bm x_i, y_i, z_i)\,:\, \exists \bm s_i \in\mathbb{R}^{N_{x,i}} \text{ s.t. }\bm x_{i} = \sum_{k=1}^{K_i} y_{i,k} P_{i,k} \bm s_{i} + z_{i},\; G_{i,k} P_{i,k} \bm s_{i} \preceq_{\mc K_{i,k}}  g_{i,k},\; k=1,\ldots,K_i \Big\} 
	\end{equation*}
	is equivalent to
	\begin{equation*}
		\wh{\mc X}_{\FS} =\Big\{(\bm x_i, y_i, z_i)\,:\, \exists \bm  \nu_{i,k}\in\mathbb{R}^{N_{x,i}} \text{ s.t. } \bm x_{i} = \sum_{k=1}^{K_i} P_{i,k} \bm \nu_{i,k} + z_{i},\; G_{i,k} P_{i,k} \bm \nu_{i,k}  \preceq_{\mc K_{i,k}}  y_{i,k}  g_{i,k},\;k=1,\ldots,K_i \Big\}
	\end{equation*}
	It is easy to verify that this in the case where $ y_{i,k} $ are positive scalar by using the substitution $\bs{\nu}_{i,k} = y_{i,k} \bs{s}_i$. In the case that any $ y_{i,k} = 0 $ then it remains to show that the only feasible solution is $\bs{\nu}_{i,k} = 0 $ so that the equality $\bs{\nu}_{i,k} = y_{i,k} \bs{s}_i$  holds. Assume that this is not the case, i.e., there exist $ \bs{\nu}_{i,k} \neq 0 $. Then, $ \bs{\nu}_{i,k} \in \textrm{recc}(S_i) $ which means that the $ \mc S_i $ recedes in the direction of  $ \bs{\nu}_{i,k} $. However, this is a contradicts the boundedness of $ \mc S_i $. 
	The substitution $\bs{\nu}_{i,k} = y_{i,k} \bs{s}_i$ was first proposed by George Dantzig in \citep{Dantzig2016}, and a similar proof also appear in  \citep{Gorissen2014}.
\end{proof}

\subsection*{Preliminaries for the proof of Theorem \ref{thm::5}}
The following lemma establishes constraint satisfaction between Problem~\eqref{DecentralizedXab_GeneralY} and Problem~\eqref{DecentralizedFinal_GeneralY}.
\begin{lemma}\label{lem::1}
	Given vectors $ y_i $ and $ z_i $ such that $ \wh{\mc X}_i(y_i, z_i) $, then for any two functions $ f_{i,t} $ and $ g_{i,t} $, it holds:
	\begin{equation}\label{eq::op}
		\begin{array}{r@{\,}ll}
			& f_{i,t}(\bm \zeta_\Ni^t, \bm \xi_i^t) \leq 0,& \forall \bm \zeta_\Ni \in \wh{\mc X}_\Ni(y_\Ni, z_\Ni), \,\forall \bm \xi_i \in \Xi_i, \\
			\Rightarrow & f_{i,t}\big([R_j^t(\bm s_j^t)]_{j\in \mc N_i}), \bm \xi_i^t\big) \leq 0, &\forall \bm s_\Ni \in \mc S_\Ni,\,\forall \bm \xi_i \in \Xi_i,
		\end{array}
	\end{equation}
	and
	\begin{equation}\label{eq::opInv}
		\begin{array}{r@{\,}ll}
			& g_{i,t}(\bm s_\Ni^t, \bm \xi^t)\leq 0,& \forall \bm s_\Ni \in \mc S_\Ni,\,\forall \bm \xi_i \in \Xi_i, \\
			\Rightarrow & g_{i,t}\big([L_j^t({\bm \zeta_j^t})]_{j\in \mc N_i}, \bm \xi_i^t\big) \leq 0, & \forall \bm \zeta_\Ni \in \wh{\mc X}_\Ni(y_\Ni, z_\Ni), \,\forall \bm \xi_i \in \Xi_i.
		\end{array}
	\end{equation}
\end{lemma}
\begin{proof}
	We prove \eqref{eq::op} by contradiction. Assume that $ f_{i,t}(\bm \zeta_\Ni^t, \bm \xi_i^t) \leq 0, \forall \bm \zeta_\Ni \in \wh{\mc X}_\Ni(y_\Ni, z_\Ni), \,\forall \bm \xi_i \in \Xi_i, $ and there exist $ \bm s_\Ni \in \mc S_\Ni $ such that $ f_{i,t}\big([R_j^t(\bm s_j^t)]_{j\in \mc N_i}), \bm \xi_i^t\big) > 0 $. Considering that $ [R_j^t(\bm s_j^t)]_{j \in \mc N_i} \in \wh{\mc X}(y_\Ni, z_\Ni) $ for all $ \bm s_\Ni \in \mc S_\Ni $ by construction, this leads to a contradiction. 
	The proof of \eqref{eq::opInv} follows similar arguments. 	
\end{proof}

\begin{proof}[\large \bf Proof of Theorem \ref{thm::5}]
	We show that every feasible solution of Problem~\eqref{DecentralizedFinal_GeneralY} is feasible in Problem~\eqref{DecentralizedXab_GeneralY}. Let $ (\wh{\bm \Gamma}_i, \wh{\mc X}_i) $ for all $ i \in \mc M $ be a feasible solution in Problem~\eqref{DecentralizedFinal_GeneralY}. Since the state of agent $i$ evolve according to \eqref{eq::stateDynamics}, we can conclude that at time $t$ we have	 
	\begin{equation*}
		\begin{array}{r@{}l}
			
			x_{i,t} &=\displaystyle A_{i,1}^t x_{i,1} + \sum_{\tau=1}^{t-1} \Big(A_{i,\tau+1}^t B_{i,\tau} \bm \zeta_{\mc N_i,\tau} + A_{i,\tau+1}^t D_{i,\tau} \wh{\Gamma}_{i,\tau}(\bm \xi_i^{\tau-1}, \bm s_\Ni^\tau) + A_{i,\tau+1}^t E_{i,\tau} \xi_{i,\tau} \Big) \\
			
			&=\displaystyle A_{i,1}^t x_{i,1} + \sum_{\tau=1}^{t-1} \Big(A_{i,\tau+1}^t B_{i,\tau} (Y_{\Ni,\tau} \bm s_{\Ni,\tau} + z_{\Ni,\tau}) + A_{i,\tau+1}^t D_{i,\tau} \wh{\Gamma}_{i,\tau}(\bm \xi_i^{\tau-1}, \bm s_\Ni^\tau) + A_{i,\tau+1}^t E_{i,\tau} \xi_{i,\tau} \Big) \\
			&=: \wh{\chi}_{i,t}(\bm \xi_{i}^{t-1}, \bm s_{\Ni}^{t-1}),
		\end{array}
	\end{equation*}
	where $ A_{i,\tau}^t = A_{i,\tau}A_{i,\tau+1}\dots A_{i,t-1} $ for $ \tau<t $ and $ A_{i,t}^t = I $. To show that $ \wh{\bm \Gamma}_i $ is feasible in Problem~\eqref{DecentralizedXab_GeneralY}, we first construct the state of agent $i$ which evolves according to $\bm x_{i} = f_{i}\big(\bm \zeta_{\Ni}, \wh{\bm \Gamma}_{i}(\bs{\xi}_i,\bm s_\Ni), \bm \xi_{i}\big)$. Starting with $\widetilde \chi_{i,1} = x_{i,1}$, we have that
	\begin{equation}\label{eq::c1_p1}
		\begin{array}{r@{}l}
			x_{i,t} &=\displaystyle A_{i,1}^t x_{i,1} + \sum_{\tau=1}^{t-1} \Big(A_{i,\tau+1}^t B_{i,\tau} \bm \zeta_{\Ni,\tau} + A_{i,\tau+1}^t D_{i,\tau} \wh{\Gamma}_{i,\tau}(\bm \xi_i^{\tau-1}, \bm s_\Ni^\tau) + A_{i,\tau+1}^t E_{i,\tau} \xi_{i,\tau} \Big) \\
			&=\displaystyle A_{i,t-1}^t x_{i,1} + \sum_{\tau=1}^{t-1} \Big(A_{i,\tau+1}^t B_{i,\tau} \bm \zeta_{\Ni} + A_{i,\tau+1}^t D_{i,\tau} \wh{\Gamma}_{i,\tau}(\bm \xi_i^{\tau-1}, [L^{\tau}_j(\bm \zeta_j^{\tau})]_{j\in \mc N_i}) + A_{i,\tau+1}^t E_{i,\tau} \xi_{i,\tau} \Big)\\
			& = \wh{\chi}_{i,t}\left(\bm \xi_i^{t-1},[L^{t-1}_j(\bm \zeta_j^{t-1})]_{j\in \mc N_i}\right)\\
			& =: \wt{\chi}_{i,t}\left(\bm \xi_i^{t-1},\bm \zeta_\Ni^{t-1}\right),
		\end{array}
	\end{equation}
	where the implications follow due to the mapping \eqref{app::map2}.
	For each $i\in \mc M$, we consider the decision $ \wt{\bm \Psi}_{i}(\bm \xi_i, \bm \zeta_{\Ni}) $ defined through
	\begin{equation}\label{eq::c2_p1}
		\wt{\Psi}_{i,t}\left(\bm \xi_i^{t-1},\bm \zeta_\Ni^{t}\right)= \wh{\Gamma}_{i,t}\left(\bm \xi_i^{t-1},[L^{t}_j(\bm \zeta_j^{t})]_{j\in \mc N_i}\right).
	\end{equation}
	Notice that \eqref{eq::c2_p1} defines a valid policy construction due to the mapping \eqref{app::map2}. It remains to show that $\wh{\bm \Gamma}_{i}$ is feasible also for the constraints of Problem~\eqref{DecentralizedXab_GeneralY}. We do so using deduction, as follows:  
	\begin{equation}
		\begin{array}{rll}
			& \big(\wh{\bs\chi}_{i}(\bm \xi_{i},\bm s_{\Ni}), \wh{\bm \Gamma}_{i}(\bs{\xi}_i,\bm s_{\Ni})\big) \in \mc O_i, & \forall \bm \xi_i \in {\Xi}_i, \forall \bm s_\Ni \in \mc S_\Ni,\\
			\implies & \big(\wh{\bs\chi}_{i}(\bm \xi_{i},[L_j(\bm \zeta_j)]_{j \in \mc N_i}),\wh{\bm \Gamma}_{i}(\bs{\xi}_i,[L_j(\bm \zeta_j)]_{j \in \mc N_i})\big) \in \mc O_i,& \forall \bm \xi_i \in {\Xi}_i, \forall \bm \zeta_\Ni \in \wh{\mc X}_\Ni,\\
			\implies &\big(\wt{\bm \chi}_{i}(\bm \xi_i,\bm \zeta_\Ni),\wt{\bm \Psi}_{i}(\bm \xi_i, \bm \zeta_{\Ni})\big) \in \mc O_i,& \forall \bm \xi_i \in {\Xi}_i, \forall \bm \zeta_\Ni \in \wh{\mc X}_\Ni,
		\end{array}
	\end{equation}
	where the implications directly follow from \eqref{eq::c1_p1} and \eqref{eq::c2_p1}, and Lemma \ref{lem::1}. Same reasoning applies to all constraints in the problem formulation. This feasible solution attains the same value, $ \ell $, for the objective functions of Problem~\eqref{DecentralizedFinal_GeneralY} and Problem~\eqref{DecentralizedXab_GeneralY}, that is:
	\begin{equation}
		\begin{array}{rl}
			\ell & = \sum\limits_{i = 1}^M \max\limits_{\bm \xi_i \in \Xi, \bm s_\Ni \in \mc S_\Ni} J_i\big(\wh{\bs\chi}_{i}(\bm \xi_{i},\bm s_{\Ni}), \wh{\bm \Gamma}_{i}(\bs{\xi}_i,\bm s_{\Ni})\big) \\& = \left \lbrace \begin{array}{l}
				J_i\big(\wh{\bs\chi}_{i}(\bm \xi_{i},\bm s_{\Ni}), \wh{\bm \Gamma}_{i}(\bs{\xi}_i,\bm s_{\Ni})\big) \leq\ell_i,~~ \forall \bm \xi_i \in {\Xi}_i, \forall \bm s_\Ni \in \mc S_\Ni\\
				\sum_{i = 1}^{M} \ell_i = \ell,
			\end{array} \right \rbrace \\
			& = \left \lbrace \begin{array}{l}
				J_i\big(\wh{\bs\chi}_{i}(\bm \xi_{i},[L_j(\bm \zeta_j)]_{j \in \mc N_i}),\wh{\bm \Gamma}_{i}(\bs{\xi}_i,[L_j(\bm \zeta_j)]_{j \in \mc N_i})\big) \leq\ell_i,~~ \forall \bm \xi_i \in {\Xi}_i, \forall \bm \zeta_\Ni \in \wh{\mc X}_\Ni\\
				\sum_{i = 1}^{M} \ell_i = \ell,
			\end{array} \right \rbrace \\
			& = \sum\limits_{i = 1}^M \max\limits_{\bm \xi_i \in {\Xi}_i, \bm \zeta_\Ni \in \wh{\mc X}_\Ni} J_i\big(\wt{\bm \chi}_{i}(\bm \xi_i,\bm \zeta_\Ni),\wt{\bm \Psi}_{i}(\bm \xi_i, \bm \zeta_{\Ni})\big) = \ell.
		\end{array}
	\end{equation}
	The implications directly follow from \eqref{eq::c1_p1} and \eqref{eq::c2_p1}, and Lemma \ref{lem::1}.
	
	Similarly, we now show that every feasible solution of Problem~\eqref{DecentralizedXab_GeneralY} is feasible in Problem~\eqref{DecentralizedFinal_GeneralY}. Let $ (\wt{\bm \Psi}_i, \wh{\mc X}_i) $ for all $ i \in \mc M $ be feasible in Problem~\eqref{DecentralizedXab_GeneralY}. Since the state of agent $i$ evolve according to \eqref{eq::stateDynamics}, we can conclude that at time $t$ we have	 
	\begin{equation}
		\begin{array}{r@{}l}
			x_{i,t} &=\displaystyle A_{i,1}^t x_{i,1} + \sum_{\tau=1}^{t-1} \Big(A_{i,\tau+1}^t B_{i,\tau} \bm \zeta_{\Ni,\tau} + A_{i,\tau+1}^t D_{i,\tau} \wt{\Psi}_{i,\tau}(\bm \xi_i^{\tau-1}, \bm \zeta_\Ni^\tau) + A_{i,\tau+1}^t E_{i,\tau} \xi_{i,\tau} \Big)\\
			&=: \wt{\chi}_{i,t}(\bm \xi_{i}^{t-1}, \bm \zeta_{\Ni}^{t-1})
		\end{array}
	\end{equation}
	To show that $ \wt{\bm \Psi}_i $ is feasible in Problem~\eqref{DecentralizedFinal_GeneralY}, we first construct the state of agent $i$ which evolves according to $\bm x_{i} = f_{i}\big(Y_{\Ni} \bm s_{\Ni} + z_{\Ni},\wt{\bm \Psi}_{i}(\bs{\xi}_i, \bm \zeta_{\Ni}), \bm \xi_{i}\big)$. Starting with $\widehat \chi_{i,1} = x_{i,1}$, we have that
	\begin{equation}\label{eq::c1_p2}
		\begin{array}{r@{}l}
			x_{i,t}&=\displaystyle A_{i,1}^t x_{i,1} + \sum_{\tau=1}^{t-1} \Big(A_{i,\tau+1}^t B_{i,\tau} (Y_{\Ni,\tau} \bm s_{\Ni,\tau} + z_{\Ni,\tau}) + A_{i,\tau+1}^t D_{i,\tau} \wt{\Psi}_{i,t}(\bm \xi_i^{\tau-1},\bm \zeta_{\Ni}^{\tau}) + A_{i,\tau+1}^t E_{i,\tau} \xi_{i,\tau} \Big) \\ 
			&=\displaystyle A_{i,1}^t x_{i,1} + \sum_{\tau=1}^{t-1} \Big(A_{i,\tau+1}^t B_{i,\tau} [R_{j,\tau}(\bm s_{j,\tau})]_{j\in \mc N_i} + A_{i,\tau+1}^t D_{i,\tau} \wt{\Psi}_{i,t}(\bm \xi_i^{\tau-1},[R^{\tau}_j(\bm s_j^{\tau})]_{j\in \mc N_i}) + A_{i,\tau+1}^t E_{i,\tau} \xi_{i,\tau} \Big) \\
			& = \wt{\chi}_{i,t}\left(\bm \xi_i^{t-1},[R^{t-1}_j(\bm s_j^{t-1})]_{j\in \mc N_i}\right)\\
			& =: \wh{\chi}_{i,t}\left(\bm \xi_i^{t-1},\bm s_\Ni^{t-1}\right).
		\end{array}
	\end{equation}
	where the implications follow due to the mapping \eqref{app::map1}.
	For each $i\in \mc M$, we consider the decision $ \wh{\bm \Gamma}_{i}(\bm \xi_i, \bm s_{\Ni}) $ defined through
	\begin{equation}\label{eq::c2_p2}
		\wh{\Gamma}_{i,t}\left(\bm \xi_i^{t-1},\bm s_\Ni^{t}\right)= \wt{\Psi}_{i,t}\left(\bm \xi_i^{t-1},[R^{t}_j(\bm s_j^{t})]_{j\in \mc N_i}\right).
	\end{equation}
	Notice that \eqref{eq::c2_p2} defines a valid policy construction due to the mapping \eqref{app::map1}. It remains to show that $\wh{\bm \Gamma}_{i}$ is feasible also for the constraints of Problem~\eqref{DecentralizedFinal_GeneralY}. We do so using deduction, as follows:  
	\begin{equation}
		\begin{array}{rll}
			& \big(\wt{\bs\chi}_{i}(\bm \xi_{i},\bm \zeta_{\Ni}) ,\wt{\bm \Psi}_{i}(\bs{\xi}_i,\bm \zeta_{\Ni})\big) \in \mc O_i,& \forall \bm \xi_i \in {\Xi}_i, \forall \bm \zeta_\Ni \in \wh{\mc X}_\Ni,\\
			\implies & \big(\wt{\bs\chi}_{i}(\bm \xi_{i},[R_j(\bm s_j)]_{j \in \mc N_i}), \wt{\bm \Psi}_{i}(\bs{\xi}_i,[R_j(\bm s_j)]_{j \in \mc N_i})\big) \in \mc O_i,& \forall \bm \xi_i \in {\Xi}_i, \forall \bm s_\Ni \in {\mc S}_\Ni,\\
			\implies &\big(\wh{\bm \chi}_{i}(\bm \xi_i,\bm s_\Ni), \wh{\bm \Gamma}_{i}(\bm \xi_i, \bm s_{\Ni})\big) \in \mc O_i, & \forall \bm \xi_i \in {\Xi}_i, \forall \bm s_\Ni \in {\mc S}_\Ni,
		\end{array}
	\end{equation}
	where the implications directly follow from \eqref{eq::c1_p2} and \eqref{eq::c2_p2}, and Lemma \ref{lem::1}. Same reasoning applies to all constraints in the problem formulation. This feasible solution attains the same value, $ \ell $, for the objective functions of Problem~\eqref{DecentralizedXab_GeneralY} and Problem~\eqref{DecentralizedFinal_GeneralY}, that is:
	\begin{equation}
		\begin{array}{rl}
			\ell & = \sum\limits_{i = 1}^M \max\limits_{\bm \xi_i \in {\Xi}_i, \bm \zeta_\Ni \in \wh{\mc X}_\Ni} J_i\big(\wt{\bs\chi}_{i}(\bm \xi_{i},\bm \zeta_{\Ni}) ,\wt{\bm \Psi}_{i}(\bs{\xi}_i,\bm \zeta_{\Ni})\big) \\& = \left \lbrace \begin{array}{l}
				J_i\big(\wt{\bs\chi}_{i}(\bm \xi_{i},\bm \zeta_{\Ni}) ,\wt{\bm \Psi}_{i}(\bs{\xi}_i,\bm \zeta_{\Ni})\big) \leq\ell_i,~~ \forall \bm \xi_i \in {\Xi}_i, \forall \bm \zeta_\Ni \in \wh{\mc X}_\Ni,\\
				\sum_{i = 1}^{M} \ell_i = \ell,
			\end{array} \right \rbrace \\
			& = \left \lbrace \begin{array}{l}
				J_i\big(\wt{\bs\chi}_{i}(\bm \xi_{i},[R_j(\bm s_j)]_{j \in \mc N_i}), \wt{\bm \Psi}_{i}(\bs{\xi}_i,[R_j(\bm s_j)]_{j \in \mc N_i})\big) \leq\ell_i,~~ \forall \bm \xi_i \in {\Xi}_i, \forall \bm s_\Ni \in {\mc S}_\Ni,\\
				\sum_{i = 1}^{M} \ell_i = \ell,
			\end{array} \right \rbrace \\
			& = \sum\limits_{i = 1}^M \max\limits_{\bm \xi_i \in {\Xi}_i, \bm s_\Ni \in {\mc S}_\Ni} J_i\big(\wh{\bm \chi}_{i}(\bm \xi_i,\bm s_\Ni), \wh{\bm \Gamma}_{i}(\bm \xi_i, \bm s_{\Ni})\big) = \ell.
		\end{array}
	\end{equation}
	The implications directly follow from \eqref{eq::c1_p2} and \eqref{eq::c2_p2}, and Lemma \ref{lem::1}.
\end{proof}

\begin{proof}[\large \bf Proof of Corollary~\ref{cor::2}]
	To demonstrate the result, it is sufficient to show that approximation \eqref{SetApproximation}   has sufficient degrees of freedom to represent the optimal state forecast set $\mc X_i$. 
	
	We first need to determine the complexity of $\mc X_i$ in Problem~\eqref{Decentralizedb}. For fixed $\mc X_i$ and an appropriate linearization of the piecewise objective function $J_i(\bm x_i,\bm u_i)$ using epigraph variables, Problem~\eqref{Decentralizedb} falls into the class of linear multistage robust optimization problem with right-hand-side uncertainty. This implies that we can replace the for all $\bm\xi_i\in\Xi_i$ and $\bm \zeta_{\Ni} \in \mc X_{\Ni}$, with for all \emph{extreme points} $\bm\xi_i\in\text{ext}(\Xi_i)$ and $\bm \zeta_{\Ni} \in\text{ext}(\mc X_{\mc N_i})$, without affecting the optimal value of the problem, see \cite{Georghiou2019b}. Since the choice of $\mc X_i$ at the beginning of the argument was arbitrary, we can conclude that we can always make this replacement  without loss of generality. 
	Moreover, due to the arborescence network structure and the linearity of the dynamics, the state $\bm x_{\text{root}}$ of the root node can take at most $|\text{ext}(\Xi_{\text{root}})|$ unique values. Hence, due to the convexity of the objective function, the smallest state forecast set $\mc X_\text{root}$  can be described with a convex combination of at most $|\text{ext}(\Xi_{\text{root}})|$ points. Using a recursive construction, the states $\bm x_i$ all agents in the tree can take at most   $\prod_{j\in \oNi} |\text{ext}(\Xi_j)|$  unique values, hence  the state forecast set $\mc X_i$  can be described with a convex combination of at most $\prod_{j\in \oNi} |\text{ext}(\Xi_j)|$ points.

	
	The above arguments shows that the maximum degrees of freedom needed to describe $\mc X_i$ is $\prod_{j\in \oNi} |\text{ext}(\Xi_j)|$. Hence, by setting $\mathcal{S}_i$ to be a simplex of dimension $\prod_{j\in \oNi} |\text{ext}(\Xi_i)|$, then for each $i\in\mc M$ the affine mapping in \eqref{SetApproximation} can project $\mathcal{S}_i$ to a set of at most $\prod_{j\in \oNi} |\text{ext}(\Xi_i)|$. Hence the result follows.
\end{proof}

\section{Supply chain with quantity flexibility contracts}\label{supplychain}

In this section, we evaluate the performance of the proposed method in a contract design mechanism for supply chains with decentralized operations. The proposed contract design is based on the structure of quantity flexibility (QF) contracts described in \citep{Tsay1999}. Decentralized supply chains are the norm in modern businesses since agents around the world cooperate to deliver multiple products. Due to this decentralized structure, each manufacturer (supplier) knows only what its immediate retailer (manufacturer) has requested, and is only concerned with its own performance cost. This, however, leads to ``mutual deception'' situations in which, for instance, some buyers inflate demand only to later disavow any undesired product \citep{Lee1997}, which increases uncertainty and operational costs in decentralized supply change networks \citep{Magee1967,Lovejoy1998}. 


To address this problem, QF contracts are used in the industry to coordinate the flow of materials and information in distributed supply chains over a fixed period of time.
In this setting, for a given product $p \in \mathcal P$ the QF contract between the pair manufacturer-retailer is parametrized by lower and upper bounds $\bm{\underline b}^p = [\underline b^p_1,\ldots,\underline b^p_T]$ and $\bm{\overline b}^p = [\overline b^p_1,\ldots,\overline b^p_T]$, respectively. Every period $t \in \mathcal T$, the retailer has the right to request delivery of any quantity of product $p \in \mc P$ within the agreed bounds $[\underline b_t^p,\overline b_t^p]$, and the manufacturer has the obligation to deliver it. 
QF contracts exist between suppliers and manufacturers as well.
In this way, the contract provides some flexibility for the retailer, helping to mitigate the uncertainties of future demand, as well as provide strong indications to the manufacturer about how to schedule the production line.
In practice, as time passes and the actual demand faced by the retailer is revealed, the two parties are allowed to revise their contracts within pre-agreed percentage changes of the lower and upper bounds. Figure~\ref{fig::QFC_v2} shows graphically how a serial supply chain with a supplier, a manufacturer and a retailer operates using QF contracts to move a single product.

\begin{figure*}[!t]
	\centering	
	\includegraphics[width = 0.8\textwidth]{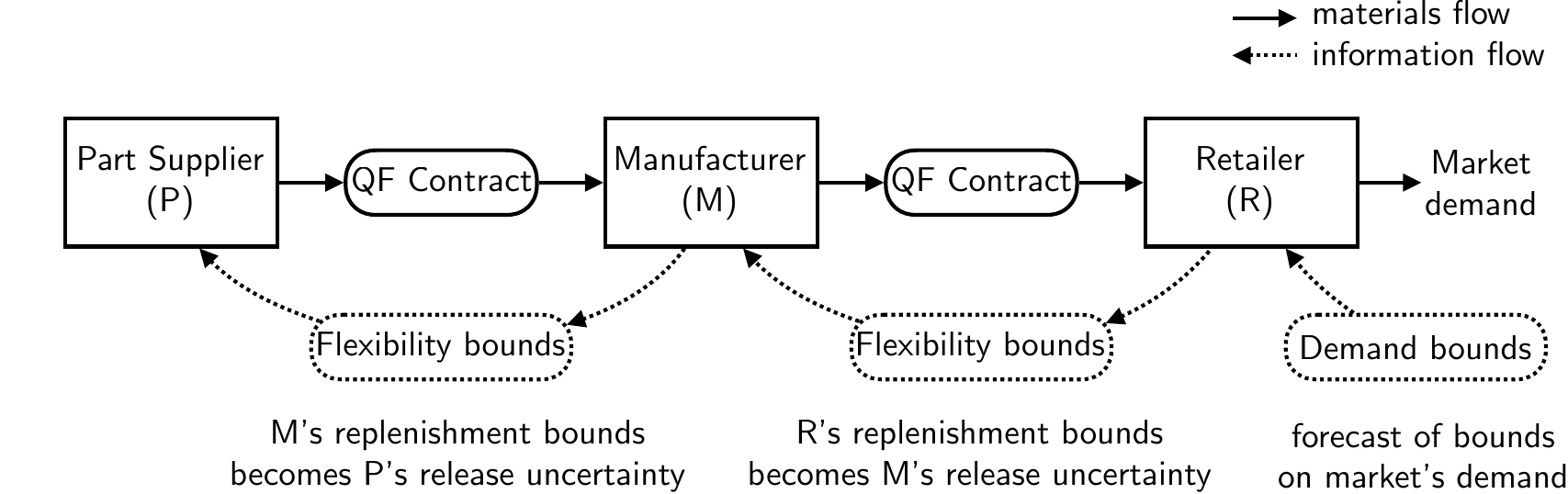}
	\caption{Supply chain design with quantitative flexibility contracts \citep{Tsay1999}}
	\label{fig::QFC_v2}
\end{figure*}

The design of QF contracts fits perfectly within the proposed framework where the forecast sets are in fact the upper and lower bounds that define the QF contracts. 
The seminal work of \citep{Lee1997} first proposes the use of optimization techniques for designing QF contracts in a so-called open-loop feedback control system, in which uncertain exogenous values are assumed to be known. Our proposed framework, however, truthfully models the uncertainty and incorporates it within the optimization framework. In the following, we discuss the design of such QF contracts. 

\subsection{Problem Formulation}
First consider a  supply chain with $M = 3$ agents as depicted in Figure~\ref{fig::QFC_v2}. In this simple example, the supplier is $i = 1$, the manufacturer is $i = 2$ and the retailer is $i = 3$, with $\mc N_1 = \{ 2 \}$,  $\mc N_2 = \{ 3 \}$ and $\mc N_3 = \emptyset$. We next partition the index set $M$ into three disjoint sets $\mc M_s = \{ 1 \}$, $\mc M_m = \{ 2 \}$, and $M_r = \{ 3 \}$ to represent the indices associated with the supplier, the manufacturer, and the supplier, respectively. 
The inventory dynamics for product $p\in\mathcal{P}=\{1,\ldots,P\}$ and for every $t \in \mc T$ and $i \in \mc M$ are expressed through
\begin{equation}
	\label{eq:dynamic}
	\begin{cases}
		I^p_{1,t+1} = I^p_{1,t} + R^p_{1,t} - U_{2,t}^p & \qquad \text{(supplier)}, \\
		I^p_{2,t+1} = I^p_{2,t} + R^p_{2,t} - U_{3,t}^p & \qquad \text{(manufacturer)}, \\
		I^p_{3,t+1} = I^p_{3,t} + R^p_{3,t} - D_t^p & \qquad \text{(retailer)}, \\
	\end{cases}
\end{equation}
where $ I^p_{i,t}$ denotes the inventory stock of product $p \in \mc P$ held by agent $i \in \mc M$ at time $t$. Furthermore, $R^p_{i,t}$ is the replenishment decision defined as
\begin{equation*}
	R^p_{i,t} =  \sum \limits_{p=1}^{P} B^p_i U^p_{i,t} + \xi^p_{i,\text{production},t},
\end{equation*}
where $ B_i^p $ is the blending coefficients, $U^p_{i,t} \in [\underline b^p_{i,t}, \overline b^p_{i,t}]$ denotes the quantity of product $p$ that agent $i$ will request from its neighbor at time~$t$ with $\underline b^p_{i,t}$ and $\overline b^p_{i,t}$ being the lower bound and upper bound of QF contracts, and
$\xi_{i,\text{production},t}^p$ is an uncertain vector capturing materials loss.
Finally, $D_t^p$ denotes the product demand for item $p$ at time $t$ originating from the market, and it is assumed to be periodic and governed by a factor model with $K$ factors of the form
\begin{equation} \label{eq:demand}
	D^p_t = \left \lbrace\begin{array}{ll}
		2 + \sin\left(2\pi \dfrac{t}{T-1}\right) + \dfrac{1}{K} \sum \limits_{k=1}^{K} F^p_k \xi_{k,t,\text{demand}} & \text{for } p \text{ even}\\[2ex]
		2 + \cos\left(2\pi \dfrac{t}{T-1}\right) + \dfrac{1}{K} \sum \limits_{k=1}^{K} F^p_k \xi_{k,t,\text{demand}} & \text{for } p \text{ odd},
	\end{array}\right.
\end{equation}
where $ F_{p}^k $ captures correlations amongst the products and
$ \xi_{k,\text{demand},t}$ captures uncertainty in the factor $k$, similar to \citep{Georghiou2019b}. 

Since the market demand and material loss are uncertain, we seek to design QF contracts that minimize the worst-case sum of backlog and inventory holding costs across all agents. Moreover, we assume that agents do not wish to disclose their actual demands from their neighbors due to privacy concerns. However, they are willing to share lower and upper limits of what they need from each other. In this case, the inventory dynamic~\eqref{eq:dynamic} translates to the following compact representation
\begin{align*}
	I^p_{i,t+1} = I^p_{i,t} + R^p_{i,t} - D_{i,t}^p,
\end{align*}
where $D_{i,t}^p = \zeta_{\mc N_i, t}^p$ for every $i \in \mc M_s \bigcup \mc M_m$ with  $\zeta_{i,t}^p \in \mathcal{U}_{i,t}^p = [\underline b^p_{i,t}, \overline b^p_{i,t}]$ being the introduced auxiliary uncertainty to represent the unknown demand from the neighbor and
$D_{i,t}^p = D_t^p$ for $i \in M_r$.
The objective of agent $i \in \mc M$ is to determine an ordering policy $\bm U_{i,t}$ for every $t \in \mc T$ based on his inventory level up to time $t$, namely $[I_{i,1}, \dots, I_{i,t}]$, and the uncertain demand from his neighbor up to time $t$, namely $[\zeta_{i+1,1}, \dots, \zeta_{i+1,t}]$, if additionally $i \in \mc M_s \bigcup \mc M_m$.
The overall optimization problem for designing the QF contracts is  an instance of Problem~\eqref{Decentralized} and is formulated as

\begin{equation}\label{QFSupplyChain}
	\begin{array}{cll}
		\text{minimize} & \displaystyle \sum_{i = 1}^M \max\limits_{\bm \xi_i \in \Xi_i, \bm \zeta_{\Ni} \in \mc U_{\Ni}} \sum_{t =1}^{T+1} \sum_{p=1}^P c_H \big[ I^p_{i,t}\big]_+ + c_B \big[-I^p_{i,t}\big]_+ \\
		\text{subject to} & 
		\bm U_i = \bm {\psi}_{i}(\bm I_i, \bm \zeta_{\Ni}) & \hspace{-5em} \forall \bm \xi_i \in \Xi_i, \forall \bm \zeta_{\mc N_i} \in \mc U_{\mc N_i}, \forall i \in \mc M \\
		& \underline{\bm{b}}_i = [\underline{b}_i^1, \dots, \underline{b}_i^P] \in \mathbb{R}^{T \times P} & \hspace{-5em} \forall i \in \mc M_m \bigcup \mc M_r \\
		& \overline{\bm{b}}_i = [\overline{b}_i^1, \dots, \overline{b}_i^P] \in \mathbb{R}^{T \times P} & \hspace{-5em} \forall i \in \mc M_m \bigcup \mc M_r \\
		& \bm U_{i} \in \mc U_{i} = [\underline{b}_i^1,\,\overline{b}_i^1] \times \cdots \times [\underline{b}_i^P,\,\overline{b}_i^P] & \hspace{-5em} \forall \bm \xi_i \in \Xi_i, \forall \bm \zeta_{\mc N_i} \in \mc U_{\mc N_i}, \forall i \in \mc M_m \bigcup \mc M_r \\
		& \displaystyle R^p_{i,t} = \sum \limits_{p=1}^{P} B^p_i U^p_{i,t} + \xi^p_{i,\text{production},t} & \hspace{-5em} \forall t \in \mc T, \forall p \in \mc P, \forall \bm \xi_i \in \Xi_i, \forall \bm \zeta_{\mc N_i} \in \mc U_{\mc N_i}, \forall i \in \mc M \\
		& \displaystyle I^p_{i,t+1} = I^p_{i,t} + R^p_{i,t} - D^p_{i,t} & \hspace{-5em} \forall t \in \mc T, \forall p \in \mc P, \forall \bm \xi_i \in \Xi_i, \forall \bm \zeta_{\mc N_i} \in \mc U_{\mc N_i}, \forall i \in \mc M \\
	\end{array}
\end{equation}
where $[\cdot]_+ = \max \{0, \cdot \}$ and the coefficients $ c_B, c_H $ denote the backlogging and inventory holding costs, respectively.
By construction of Problem~\eqref{QFSupplyChain}, for each agent $i$ and product $p$ we have  $\mathcal{U}_{i}^p = \mathcal{U}_{i,1}^p\times \cdots \times \mathcal{U}_{i,T}^p$ and  $\mathcal{U}_{i} = \mathcal{U}_{i}^1 \times \cdots \times \mathcal{U}_{i}^P$. Thus  $\mathcal{U}_{i}$  is a  hyper-rectangles which is controlled coordinate wise. Hence, it can be exactly represented by the primitive sets $\mathcal S_{i,t}^p = [-1,\,1] $ and 
\begin{equation*}
	\mc U_{i,t}^p(y_{i,t}^p,z_{i,t}^p) = \displaystyle\left\{U_{i,t}^p \in\mathbb{R} \,:\, \exists s_{i,t}^p  \in \mc S_{i,t}^p \text{ s.t. } U_{i,t}^p =    y_{i,t}^p s_{i,t}^p + z_{i,t}^p \right\},
\end{equation*}	
as discussed in Remark~\ref{remark1}. Applying Theorem~\ref{thm::5}, Problem~\eqref{QFSupplyChain} can thus be reformulated as an instance of Problem~\eqref{DecentralizedFinal_GeneralY} where sets $\mc S_i = \prod_{p=1}^P [-1, \, 1]^T$.
\begin{align} \label{eq:QF:final}
	\begin{array}{cll}
		\text{minimize} & \displaystyle \sum_{i = 1}^M \max\limits_{\bm \xi_i \in \Xi_i, \bm s_{\Ni} \in \mc S_{\Ni}} \sum_{t =1}^{T+1} \sum_{p=1}^P c_H \big[ I^p_{i,t} \big]_+ + c_B \big[ -I^p_{i,t} \big]_+ \\
		\text{subject to} &
		\bm U_i = \bm {\Gamma}_{i}(\bm \xi_i, \bm s_{\Ni}) & \hspace{-5em} \forall \bm \xi_i \in \Xi_i, \forall \bm s_{\mc N_i} \in \mc S_{\mc N_i}, \forall i \in \mc M \\
		& \underline{\bm{b}}_i,\,\overline{\bm{b}}_i,\, \bm z_i \in \mathbb{R}^{T \times P},\,\bm y_i \in \mathbb{R}_+^{T \times P} & \hspace{-5em} \forall i \in \mc M_m \bigcup \mc M_r \\
		& \underline{\bm{b}}_i = \bm z_i - \bm y_i, \, \overline{\bm{b}}_i = \bm z_i + \bm y_i & \hspace{-5em} \forall i \in \mc M_m \bigcup \mc M_r \\
		& \bm U_{i} \in \mc U_{i} = [\underline b_i^1,\,\overline b_i^1] \times \cdots \times [\underline b_i^P,\,\overline b_i^P] & \hspace{-5em} \forall \bm \xi_i \in \Xi_i, \forall \bm s_{\mc N_i} \in \mc S_{\mc N_i}, \forall i \in \mc M_m \bigcup \mc M_r \\
		& \zeta_{i,t}^p = y_{i,t}^p s_{i,t}^p + z_{i,t}^p & \hspace{-5em} \forall t \in \mc T, \forall p \in \mc P, \forall \bm s_{i} \in \mc S_{i}, \forall i \in \mc M_m \bigcup \mc M_r \hspace{-8em} \\
		& \displaystyle R^p_{i,t} = \sum \limits_{p=1}^{P} B^p_i U^p_{i,t} + \xi^p_{i,\text{production},t} & \hspace{-5em} \forall t \in \mc T, \forall p \in \mc P, \forall \bm \xi_i \in \Xi_i, \forall \bm s_{\mc N_i} \in \mc S_{\mc N_i}, \forall i \in \mc M \hspace{-8em} \\
		& \displaystyle I^p_{i,t+1} = I^p_{i,t} + R^p_{i,t} - D^p_{i,t} & \hspace{-5em} \forall t \in \mc T, \forall p \in \mc P, \forall \bm \xi_i \in \Xi_i, \forall \bm s_{\mc N_i} \in \mc S_{\mc N_i}, \forall i \in \mc M \hspace{-8em}
	\end{array}
\end{align}

The corresponding centralized supply chain problem that \emph{does not involve the quantity flexibility contracts} can be written as an instance of Problem~\eqref{Centralized} in which all agents have access to the inventory levels of all other agents. This results in the following optimization problem
\begin{align} \label{eq:centralized:supply:chain}
	\begin{array}{cll}
		\text{minimize} & \displaystyle \sum_{i = 1}^M \max\limits_{\bm \xi_{\mc M} \in \Xi_{\mc M}} \sum_{t =1}^{T+1} \sum_{p=1}^P c_H \big[ I^p_{i,t} \big]_+ + c_B \big[ -I^p_{i,t} \big]_+ \\
		\text{subject to} & 
		\bm U_i = \bm {\pi}_{i}(\bm I_{\mc M}) & \hspace{-5em} \forall \bm \xi_{\mc M} \in \Xi_{\mc M}, \forall i \in \mc M \\
		& \displaystyle R^p_{i,t} = \sum \limits_{p=1}^{P} B^p_i U^p_{i,t} + \xi^p_{i,\text{production},t} & \hspace{-5em} \forall t \in \mc T, \forall p \in \mc P, \forall \bm \xi_{\mc M} \in \Xi_{\mc M}, \forall i \in \mc M \\
		& \displaystyle I^p_{i,t+1} = I^p_{i,t} + R^p_{i,t} - D^p_{i,t} & \hspace{-5em} \forall t \in \mc T, \forall p \in \mc P, \forall \bm \xi_{\mc M} \in \Xi_{\mc M}, \forall i \in \mc M, \hspace{-5em}
	\end{array}
\end{align}
where, with slight abuse of nation, $D^p_{i,t} = U^p_{\Ni, t}$ for every $i \in \mc M_s \bigcup \mc M_m$ and $D^p_{i,t} = D^p_{t}$ for $i \in \mc M_r$. Problem~\eqref{eq:centralized:supply:chain} is in turn can be reformulated as an instance of Problem~\eqref{Centralizedb}. 

In the following numerical experiments, we assume that the initial inventory levels are zero, the uncertain demand~\eqref{eq:demand} are produced by $K = 4$ factors. We generate random instance of Problem~\eqref{eq:QF:final} by uniformly generating the coefficients $ F_{p}^k$ from the interval $[-1, 1]$, the backlogging and holding coefficients, $c_B$ and $c_H$ respectively, are randomly generated from $[0, 1]$, and the blending coefficients $B_i^p$ are uniformly generated from $[0.5, 1]^P$ for all $i \in \mc M$ and $p \in \mc P$. We also assume that the production and demand uncertainties 
and the optimization belong to $\xi_{i,\text{production},t}^p \in [-0.1,0]$ and $ \xi_{k,\text{demand},t} \in [-\theta,\,\theta]$ with $\theta$ denoting the level of uncertainty, respectively. In this way, the uncertainty is characterized by the set $\Xi_i = \prod_{t=1}^T  [-0.1,0]^P $ for agents $i \in \mc M_s \bigcup \mc M_m$ and by the set $\Xi_i = \prod_{t=1}^T  [-0.1,\, 0]^P \times [-\theta, \,\theta]^K$ for agents $i \in \mc M_r$.
Finally, we approximate Problems~\eqref{eq:QF:final} and~\eqref{eq:centralized:supply:chain} using affine policies, while the maximum operator $ [\cdot]_+ $ is linearized using epigraphical variables similar to \cite[Section 5.1]{Bertsimas2015}.

\subsection{Decentralized Optimization via ADMM}
We demonstrate how an ADMM algorithm can be applied to solve the  Problem~\eqref{eq:QF:final}, which promotes decentralized computation and provides significant privacy to all agents. For illustration, we consider the system depicted in Figure~\ref{fig::QFC_v2} with $M=3$ agents, $T=20$ horizon length, $P=1$ product, and the market demand parameter $\theta = 1$. Notice that the decision variables $\underline{\bm b}_i, \overline{\bm b}_i$ in the optimization problem \eqref{eq:QF:final} are auxiliary. Removing these auxiliary variables yields the following reformulation
\begin{align*}
	\min_{\bm y_2, \bm y_3\in \mathbb R_+^T\,, \bm z_2, \bm z_3 \in \mathbb R^T} ~ J_1(\bm y_2, \bm z_2) + J_2(\bm y_2, \bm z_2, \bm y_3, \bm z_3) + J_3(\bm y_3, \bm z_3),
\end{align*}
where the implicit functions $J_i$ are defined as 
\begin{align*}
	J_1(\bm y_2, \bm z_2) &= 
	\left\{
	\begin{array}{cll}
		\text{minimize} & \displaystyle \max\limits_{\bm \xi_1 \in \Xi_1, \bm s_{2} \in \mc S_{2}} \sum_{t =1}^{T+1} \sum_{p=1}^P c_H \big[ I^p_{1,t} \big]_+ + c_B \big[ -I^p_{1,t} \big]_+ \\
		\text{subject to} & 
		\bm U_1 = \bm {\Gamma}_{1}(\bm \xi_1, \bm s_{2})
		& \hspace{-3.5em} \forall \bm \xi_1 \in \Xi_1, \forall \bm s_{2} \in \mc S_{2} \\
		& \zeta_{2,t}^p = y_{2,t}^p\, s_{2,t}^p + z_{2,t}^p & \hspace{-3.5em} \forall t \in \mc T, \forall p \in \mc P, \forall \bm s_{2} \in \mc S_{2} \\
		& \displaystyle R^p_{1,t} = \sum \limits_{p=1}^{P} B^p_1 U^p_{1,t} + \xi^p_{1,\text{production},t} & \hspace{-3.5em} \forall t \in \mc T, \forall p \in \mc P, \forall \bm \xi_1 \in \Xi_1, \forall \bm s_{2} \in \mc S_{2} \\
		& \displaystyle I^p_{1,t+1} = I^p_{1,t} + R^p_{1,t} - \zeta_{2,t}^p & \hspace{-3.5em} \forall t \in \mc T, \forall p \in \mc P, \forall \bm \xi_1 \in \Xi_1, \forall \bm s_{2} \in \mc S_{2},
	\end{array}
	\right. \\
	J_2(\bm y_2, \bm z_2, \bm y_3, \bm z_3) &= 
	\left\{
	\begin{array}{cll}
		\text{minimize} & \displaystyle \max\limits_{\bm \xi_2 \in \Xi_2, \bm s_{3} \in \mc S_{3}} \sum_{t =1}^{T+1} \sum_{p=1}^P c_H \big[ I^p_{2,t} \big]_+ + c_B \big[ -I^p_{2,t} \big]_+ \\
		\text{subject to} & 
		\bm U_2 = \bm {\Gamma}_{2}(\bm \xi_2, \bm s_{3})
		& \hspace{-3.5em} \forall \bm \xi_2 \in \Xi_2, \forall \bm s_{3} \in \mc S_{3} \\
		& \zeta_{3,t}^p = y_{3,t}^p \, s_{3,t}^p + z_{3,t}^p & \hspace{-3.5em} \forall t \in \mc T, \forall p \in \mc P, \forall \bm s_{3} \in \mc S_{3} \\
		& U_{2,t}^p \in \left[ z_{2,t}^p - y_{2,t}^p,\, z_{2,t}^p + y_{2,t}^p \right] & \hspace{-3.5em} \forall t \in \mc T, \forall p \in \mc P, \forall \bm \xi_2 \in \Xi_2, \forall \bm s_{3} \in \mc S_{3} \\
		& \displaystyle R^p_{2,t} = \sum \limits_{p=1}^{P} B^p_2 U^p_{2,t} + \xi^p_{2,\text{production},t} & \hspace{-3.5em} \forall t \in \mc T, \forall p \in \mc P, \forall \bm \xi_2 \in \Xi_2, \forall \bm s_{3} \in \mc S_{3} \\
		& \displaystyle I^p_{2,t+1} = I^p_{2,t} + R^p_{2,t} - \zeta^p_{3,t} & \hspace{-3.5em} \forall t \in \mc T, \forall p \in \mc P, \forall \bm \xi_2 \in \Xi_2, \forall \bm s_{3} \in \mc S_{3},
	\end{array}
	\right. \\
	J_3(\bm y_2, \bm z_2) &= 
	\left\{
	\begin{array}{cll}
		\text{minimize} & \displaystyle \max\limits_{\bm \xi_3 \in \Xi_i} \sum_{t =1}^{T+1} \sum_{p=1}^P c_H \big[ I^p_{3,t} \big]_+ + c_B \big[ -I^p_{3,t} \big]_+ \\
		\text{subject to} &
		\bm U_3 = \bm {\Gamma}_{3}(\bm \xi_3)
		& \hspace{-1em} \forall \bm \xi_3 \in \Xi_3 \\
		& U_{3,t}^p \in \left[ z_{3,t}^p - y_{3,t}^p,\, z_{3,t}^p + y_{3,t}^p \right] & \hspace{-1em} \forall t \in \mc T, \forall p \in \mc P, \forall \bm \xi_3 \in \Xi_3 \\
		& \displaystyle R^p_{3,t} = \sum \limits_{p=1}^{P} B^p_3 U^p_{3,t} + \xi^p_{3,\text{production},t} & \hspace{-1em} \forall t \in \mc T, \forall p \in \mc P, \forall \bm \xi_3 \in \Xi_3 \\
		& \displaystyle I^p_{3,t+1} = I^p_{3,t} + R^p_{3,t} - D^p_{3,t} & \hspace{-1em} \forall t \in \mc T, \forall p \in \mc P, \forall \bm \xi_3 \in \Xi_3.
	\end{array}
	\right. 
\end{align*}
Define next the global decision variable $\bm \alpha = [\bm y_2^\top, \bm z_2^\top, \bm y_3^\top, \bm y_4^\top]^\top$. With slight abuse of notation, we can now reformulate the above optimization problem as the following optimization problem
\begin{align}
	\label{eq:ADMM:problem}
	\min\limits_{\bm \beta_1, \bm \beta_2, \bm \beta_3, \bm \alpha} \left\{ \sum_{i \in \mc M} J_i(\bm \beta_i): \bm \beta_i = \widetilde{\bm \alpha}_i ~ \forall i \in \mc M \right\},
\end{align}
where given a global decision variable $\bm \alpha = [\bm y_2^\top, \bm z_2^\top, \bm y_3^\top, \bm y_4^\top]^\top$, we split $\bm \alpha$ to the (overlapping) chunks $\widetilde{\bm \alpha}_1 = [\bm y_2^\top, \bm z_2^\top]^\top$, $\widetilde{\bm \alpha}_2 = [\bm y_2^\top, \bm z_2^\top, \bm y_3^\top, \bm y_4^\top]^\top$, and $\widetilde{\bm \alpha}_3 = [\bm y_3^\top, \bm z_3^\top]^\top$.
Figure~\ref{fig::ADMM}~(a) depicts the underlying graph structure of Problem~\eqref{eq:ADMM:problem}, whose augmented Lagrangian is of the form
\begin{align*}
	\mc L(\bm \beta_1, \bm \beta_2, \bm \beta_3, \bm \alpha, \bm \gamma_1, \bm \gamma_2, \bm \gamma_3) = \sum_{i \in \mc M} J_i(\bm \beta_i) + \bm \gamma_i^\top (\bm \beta_i - \widetilde{\bm \alpha}_i) + \frac{\rho}{2} \| \bm \beta_i - \widetilde{\bm \alpha}_i \|^2,
\end{align*}
with $\bm \gamma_i$ being the dual variable associated with the constraint $\bm \beta_i = \bm {\widetilde \alpha}_i$, and $\rho$ being a positive constant. Then, the ADMM update at iteration $k$ follows the form
\begin{align*}
	\bm \beta_i^{(k+1)} &\gets \argmin_{\bm \beta_i} \left\{ J_i(\bm \beta_i) + \bm \left( \gamma_i^{(k)} \right)^\top \bm \beta_i + \frac{\rho}{2} \left\| \bm \beta_i - \widetilde{\bm \alpha}_i^{(k)} \right\|^2 \right\} \\
	\bm \alpha^{(k+1)} &\gets \argmin_{\bm \alpha} \left\{ \sum_{i \in \mc M} - \left(\bm \gamma_i^{(k)} \right)^\top \widetilde{\bm \alpha}_i + \frac{\rho}{2} \left\| \bm \beta_i^{(k+1)} - \widetilde{\bm \alpha}_i \right\|^2 \right\} \\
	\bm \gamma_i^{(k+1)} &\gets \bm \gamma_i^{(k)} + \rho \left( \bm \beta_i^{(k+1)} - \widetilde{\bm \alpha}_i^{(k+1)} \right).
\end{align*}
Notice that the updates of decision variables $\bm \beta_i$ and $\bm \gamma_i$ can be carried out locally for every agent $i \in \mc M$, which implies that the structure of $J_i$ is only known to agent $i$. This salient feature promotes privacy amongst the agents. In addition, the update of the decision variable $\bm \alpha$ involves solving an unconstrained quadratic minimization problem that can be solved analytically; see \citep[\S~7.2]{Boyd2011}.
The analytic expression constitutes local averaging rather than global averaging, and therefore, it can be accomplished by local information exchange. This indicates that all ADMM updates can be performed locally up to the information exchange $\widetilde{\bm \alpha}_i$ between neighbors. Figure~\ref{fig::ADMM}~(b) reports the convergence behavior of the average performance of the ADMM algorithm for solving $10$ random instances of the Problem~\eqref{eq:QF:final} where the functions $J_i$ are approximated via affine decision rules. In the experiments we set set $\rho = 0.1$ and  decision variables $\widetilde{\bm \alpha}^{(0)}_i,\, \bm \gamma_i^{(0)},\bm \beta_i^{(0)}$ are initialized at zero for all $i\in\mathcal{M}$. We observe that the algorithm converges to an optimal solution and achieves the machine precision in $10$ iterations. 

\begin{figure*}
	\center
	\begin{minipage}{0.37\textwidth}
		\subfigure[]{\includegraphics[width = \textwidth]{ADMM.pdf}}
	\end{minipage}\hfil
	\begin{minipage}{0.4\textwidth}
		\subfigure[]{\includegraphics[width = \textwidth]{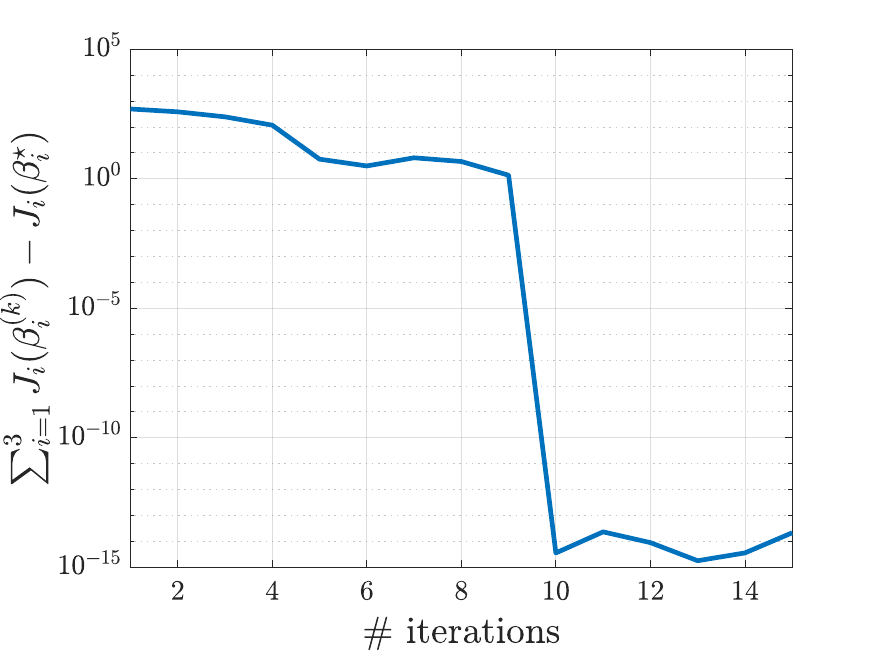}}
	\end{minipage}
	\caption{(a) Illustration of the graph structure with $3$ agents. Local objective functions are on the left; global variable components are on the right. The bipartite graph can be viewed as a consistency constraint that links local variables and global variables. (b)  Convergence behavior of the ADMM algorithm.}
	\label{fig::ADMM}
\end{figure*}

\begin{figure*}[t]
	\centering
	\subfigure[]{\includegraphics[width=0.425\textwidth]{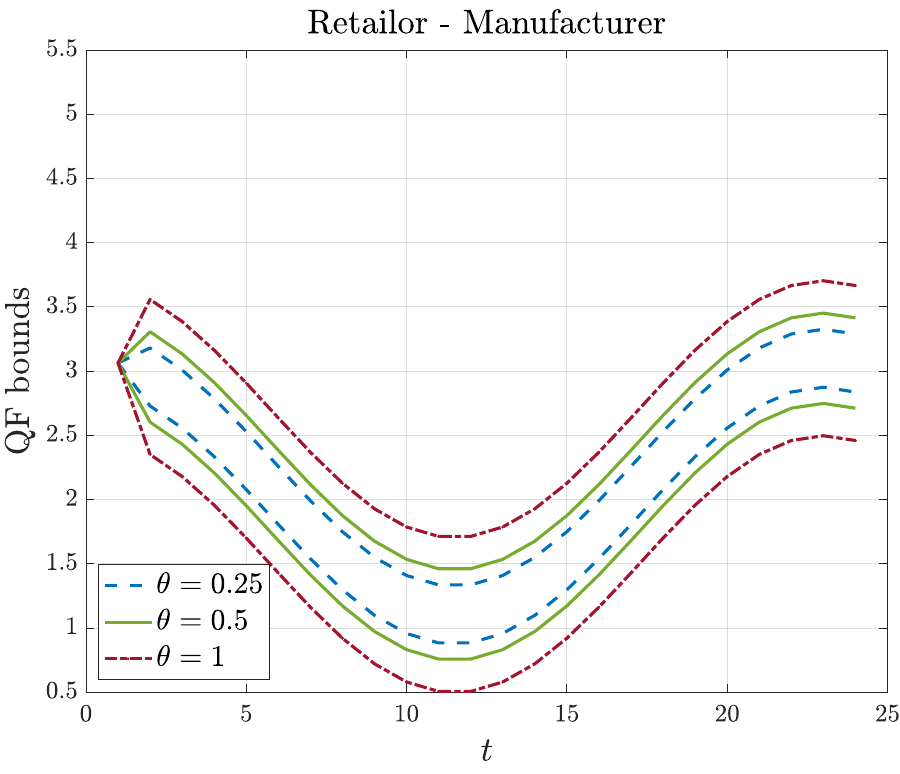}}\hfil
	\subfigure[]{\includegraphics[width=0.425\textwidth]{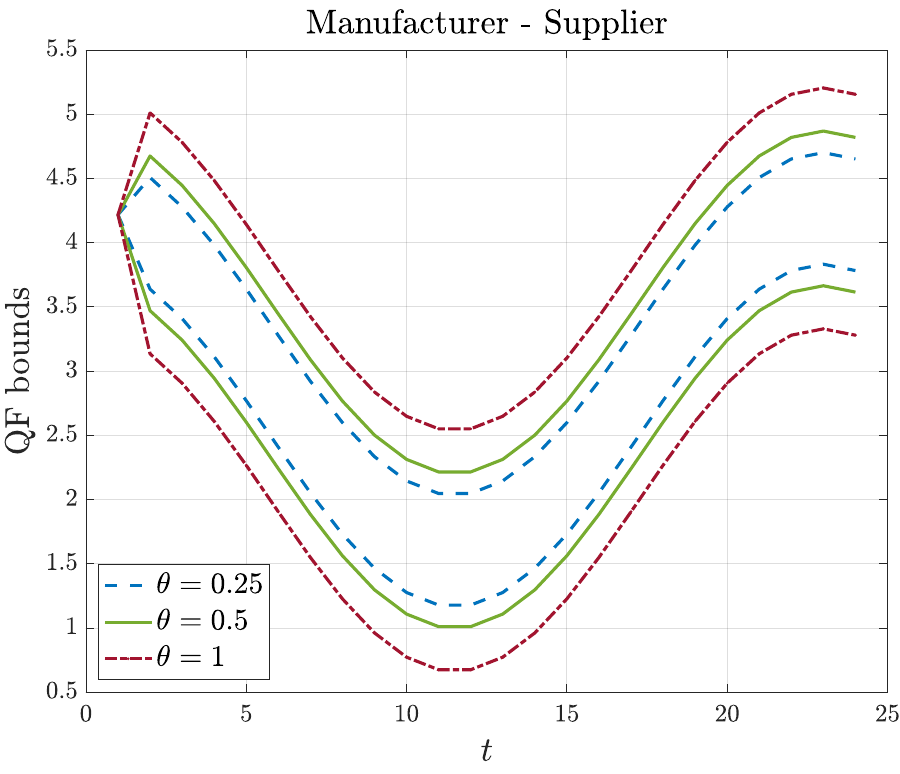}}
	\caption{Effect of uncertainty on (a) retailer-manufacturer and (b) manufacturer-supplier QF contracts over a $T = 24$ horizon length.}
	\label{fig::SC_UndEff}
\end{figure*}

\subsection{Numerical Results}
\label{sec::numerics:supply}
In the first experiment, we investigate how the degree of uncertainty in the market demand affects the QF bounds. We consider the system depicted in Figure~\ref{fig::QFC_v2} with $M=3$ agents and $P=1$ product. We solve the optimization problem~\eqref{eq:QF:final} with $\theta = \{ 0.25,0.5,1 \}$ for a horizon length of $T=24$. The QF contracts between retailer/manufacturer and manufacturer/supplier pairs are depicted in Figure~\ref{fig::SC_UndEff} when we set $F_p^k = (-1)^k / 2, c_B=c_H = 1$ and $B_1^p = B_2^p = B_3^p = 1$. We observe that the size of the QF contrasts increases as the uncertainty in the market demand increases. However, due to the adaptive nature of the recourse decisions, the size of the QF bounds does not substantially increase over time. 
Moreover, we observe that the QF bounds between manufacturers and suppliers are wider than those between manufacturers and retailers. This observation is consistent with the bullwhip effect in \citep{Lee1997}, a theory that describes how small fluctuations in demand at the retail level can cause progressively larger fluctuations at the manufacturer and supplier levels.

In the second experiment, we investigate the effects of horizon length and number of agents in the network with a single supplier, $N$ intermediate manufacturers, and a single retailer.
We compare our proposed local information exchange policy design to the centralized one over $10$ randomly generated instances. Throughout the experiment, we fix the number of products at $P=2$ and the market demand parameter constant at $\theta = 1$.
First, we compare the optimal values of the local problem~\eqref{eq:QF:final} and the centralized problem~\eqref{eq:centralized:supply:chain}. Denoting by $\text{obj}_L$ the objective value of~\eqref{eq:QF:final} and by $\text{obj}_C$ the objective of the centralized~\eqref{eq:centralized:supply:chain}, we define the (percentage) suboptimality as $100 \times ({\text{obj}_L - \text{obj}_C}) / {\text{obj}_C}$.
We evaluate the effect of the horizon length and the number of manufacturers on the quality of the solution in terms of the suboptimality metric. In addition, we examine the impact of demand-side delays on the solution of the local and centralized policy designs, which occur when agents report their demands to their preceding agents at the end of the time interval rather than submitting their requests at the beginning.
Specifically, we assume that the inventory dynamic is of the form
\begin{align*}
	I_{i,t+1}^p = I_{i,t+1}^p + R_{i,t}^p - D_{i,t-1}^p
\end{align*}
for every $i \in \mc M$, where the demand term $D_{i,t-1}^p$ is modified to incorporate the demand-side delay. The introduction of the delay yields policies of the form $U_{i,t} = \bm \Psi(\bm \xi_i^{t-1}, \bm s_{\mc N_i}^{t-1})$ and $U_{i,t} = \bm \Pi(\bm \xi_i^{t-1}, \bm \xi_{\mc M \backslash \{i\}}^{t-1})$ for every $t \in \mc T$ and $i \in \mc M$ in the decentralized and centralized settings, respectively, and in turn approximated by  affine decision rules.
We consider two different cases. In the first case, we consider $N=1$ intermediate manufacturer and increase the time horizon $T$ up to the horizon length $10$. In the second case, we fix the time horizon to $T=5$, while changing the number of manufacturers $N$ from 1 to 10. 
Figure~\ref{fig::SC_Comparison} summarizes the results. In the absence of a time delay in the system, we observe that the suboptimality is zero. Interestingly, the introduction of QF contracts between agents not only preserves their privacy but also does not impact  performance. 
In contrast, time delay has a significant impact on the quality of the solutions.
Specifically, we observe that an increase in horizon length can have an adverse effect on the quality of the local information problem. This is to be expected, as the uncertainty faced by each agent increases with increasing the horizon length in the local information problem. However, as the horizon length increases, the suboptimality becomes saturated, and the decentralized policy absorbs the effect of delay, as agents begin mitigating uncertainty through the use of local information exchanges. 
In addition, as the number of agents increases beyond $4$, there is essentially no difference in terms of suboptimality as the function of the number of manufacturers.

\begin{figure*}[t]
	\center
	\includegraphics[width = 0.45\textwidth]{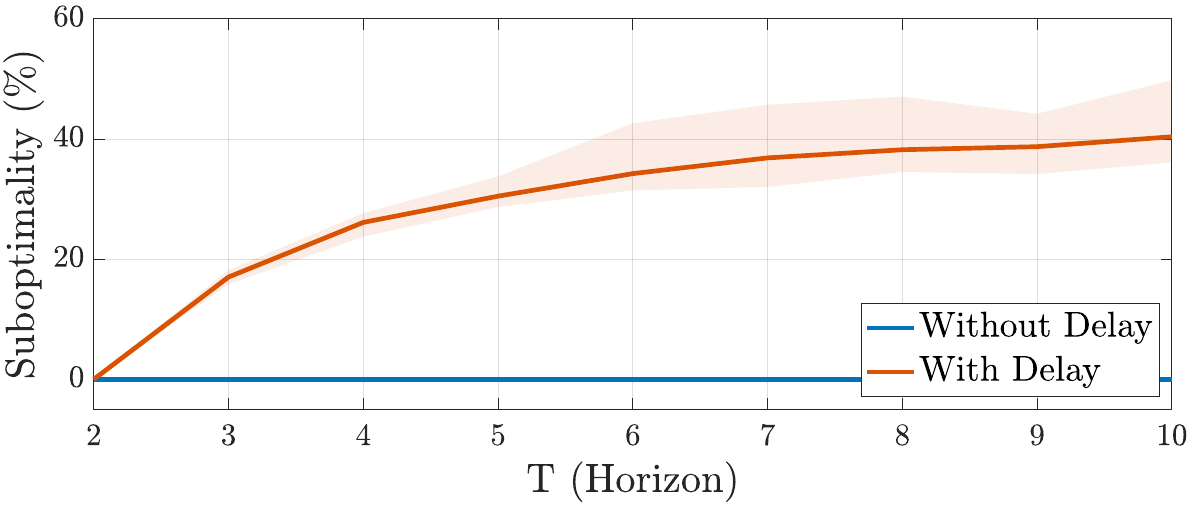}
	\includegraphics[width = 0.45\textwidth]{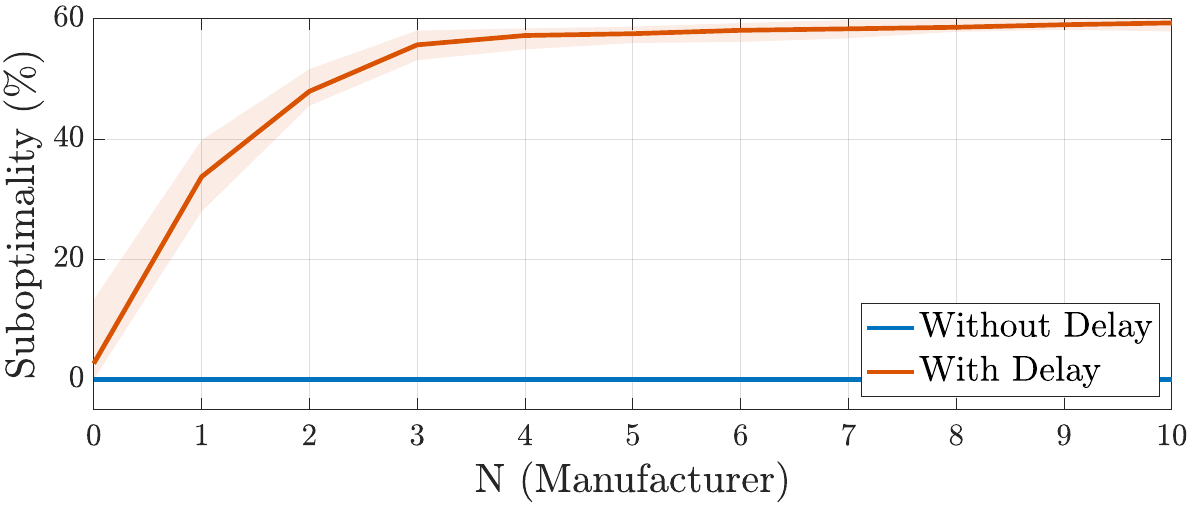} \\
	\caption{Suboptimality of local vs centralized problem as a function of the horizon length (left) and number of agents (right). Solid lines (shaded regions) represent averages (ranges) across $10$ independent simulations.}
	\label{fig::SC_Comparison}
\end{figure*}

Next, we use a rolling horizon scheme to compare the average performance of the centralized and local information problems. In this experiment, we fix all parameters as in the first experiment. We then generate a random realization of the uncertainty $\bm \xi\in\Xi$. Next, we solve the centralized and local information for the horizon length $T=10$ and update the inventory levels  according to the realization $\bm \xi_1$ and the first stage decisions $U_{i,1}^p$. We repeat the process until we reach the end of the horizon. Specifically, at any time $t =2,\ldots,T$, we resolve each  problem for the shorter horizon length $T-t$ and the initial inventory stocks $I^p_{i,t-1}$ for every $i \in \mathcal M$. We then update the inventory levels according to realization $\bm \xi_t$ and corresponding first stage decisions $U_{i,t}^p$. Figure~\ref{fig::SC_Comparison::2} summarizes our results for $10$ randomly generated realization of uncertainty. We observe that the average performances of the local and centralized information problems are significantly improved compared to their worst-case performance obtained by solving \eqref{eq:QF:final} and \eqref{eq:centralized:supply:chain}, respectively. The improvement is more significant when there is a delay in the system.

\begin{figure*}[t]
	\center
	\includegraphics[width = 0.45\textwidth]{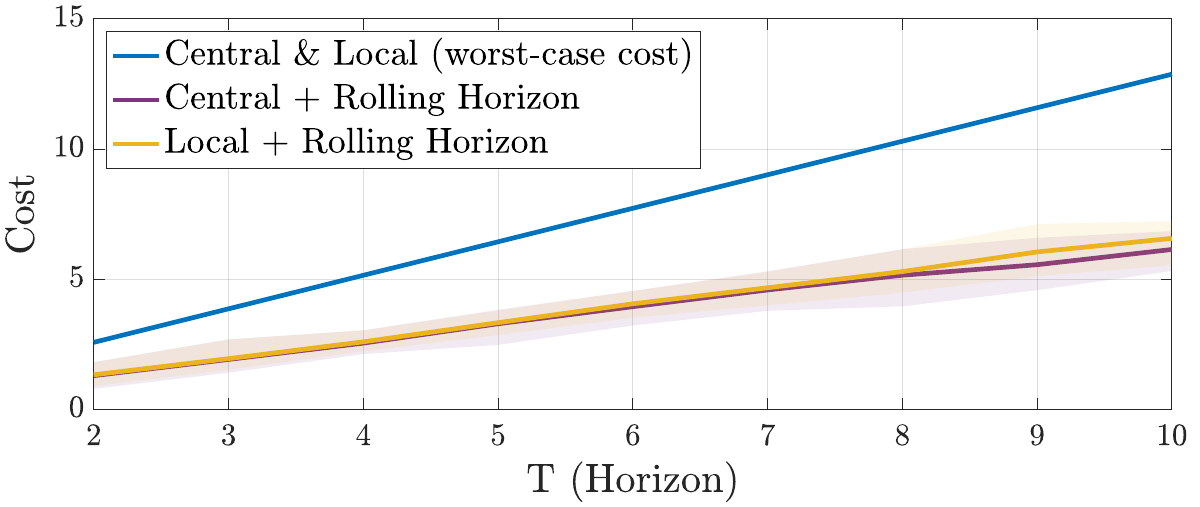}
	\includegraphics[width = 0.45\textwidth]{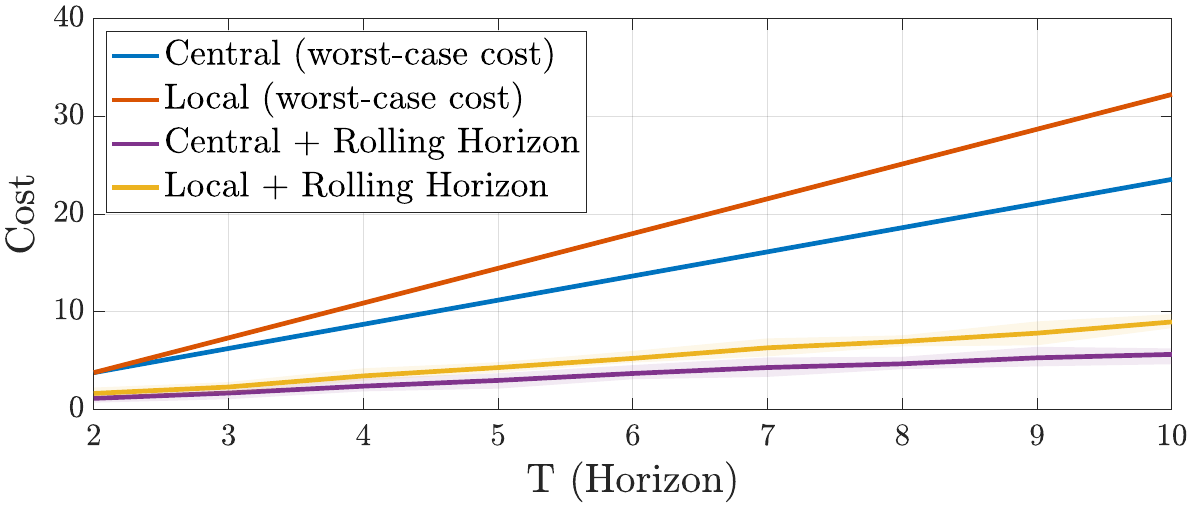} 
	\caption{Effect of the rolling horizon in a system with no delay (left) and with delay (right). The graphs report the worst-case cost computed by solving Problems \eqref{eq:QF:final} and \eqref{eq:centralized:supply:chain}, and the cost from the rolling horizon scheme. Solid lines (shaded regions) represent averages (ranges) across $10$ independent simulations.}
	\label{fig::SC_Comparison::2}
\end{figure*}

Finally, we compare the optimization time required to solve the local and centralized problems.
Figure~\ref{fig::runtime} reports the execution times required by Gurobi to solve 10 randomly generated instances of the centralized and local information exchange designs. 
Figure~\ref{fig::runtime} (left) the number of manufacturers is fixed to $N=5$, whereas in Figure~\ref{fig::runtime} (right) the horizon length is fixed to $T=5$.
In both cases, the time required to solve the local information problem~\eqref{eq:QF:final} is nearly half the time required to solve the centralized optimization problem~\eqref{eq:centralized:supply:chain}. 
This can be attributed to the nearly decoupled structure of the problem, in which only the QF bounds link dynamics and constraints. 

\begin{figure*}[t]
	\centering
	\includegraphics[width=0.45\textwidth]{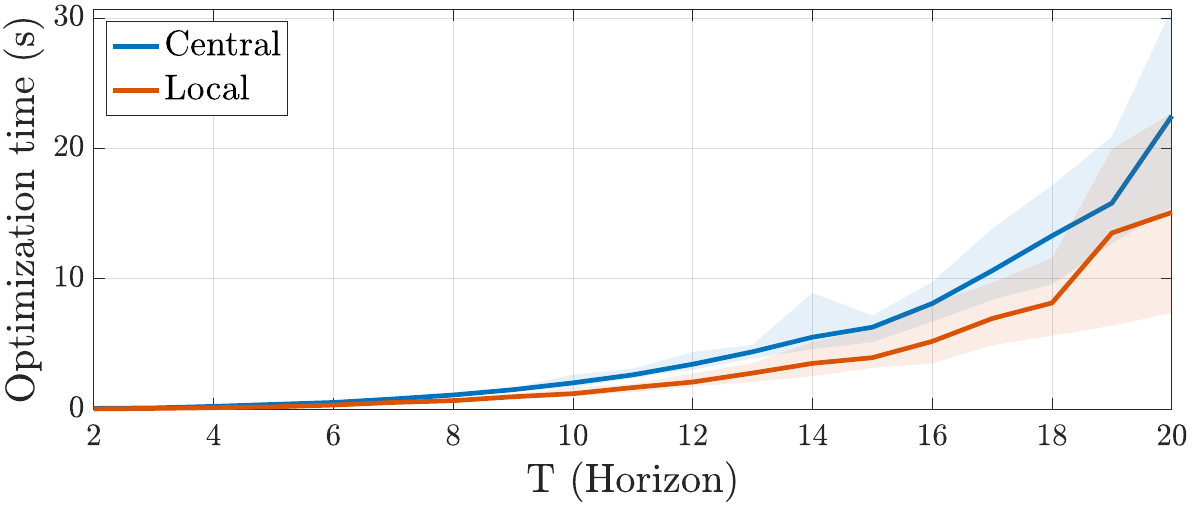}
	\includegraphics[width=0.45\textwidth]{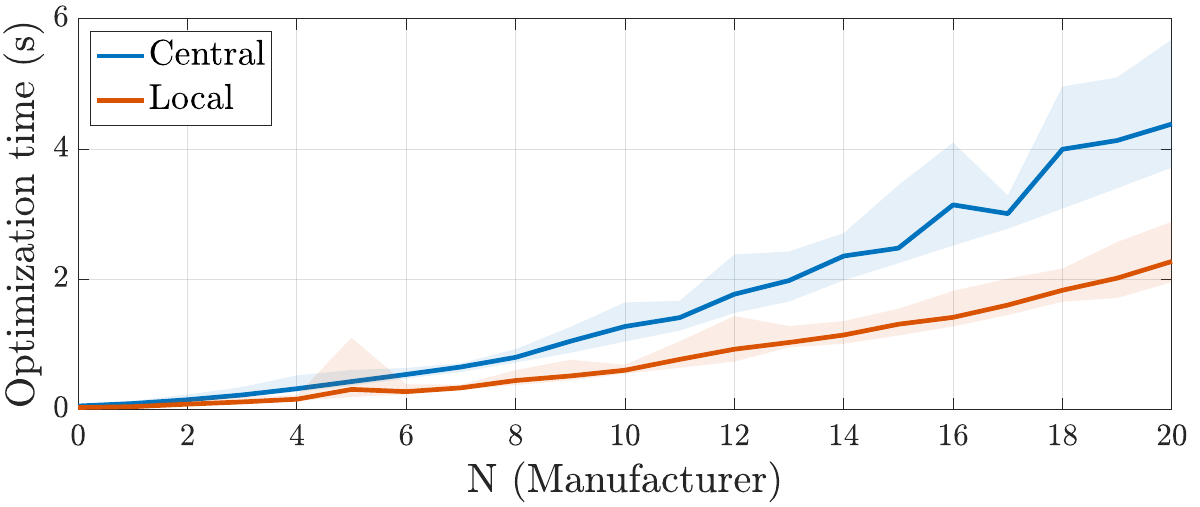}
	\caption{Comparison of the runtime of the centralized and local information problems. Solid lines (shaded regions) represent averages (ranges) across $10$ randomly generate instances.}	
	\label{fig::runtime}
\end{figure*}

\newpage
\section{Summary of major notation}\label{appedix::notation}

\noindent\textbf{Index sets:} We use $\mc M$ to denote the set of all agents (Section~\ref{sec::ProbForm}), $\mc N_i$ to denote the set of neighbors of agent $i$ (Section~\ref{sec::ProbForm}), and  $ \ol{\mc N}_i$ the set that includes agent $ i $ and all its precedent agents (Section~\ref{partiallynested}).

\noindent \textbf{Vectors concatenation:} For given vectors $ v_{i} \in \mb R^{k_i} $ with $ k_i \in \mb N $, $ i \in \mc M$, we define $ \bm v_\Mu = [v_{i}]_{i\in \mc M} = [v_{1}^\top \ldots v_{M}^\top]^\top \in \mb R^{k} $ with $ k = \sum_{i=1}^{M}k_i $ as their vector concatenation. Given time dependent vectors $ \nu_{i,t} \in \mb R^{\ell_i} $ with $ i \in \mc M $, $ t \in \mc T $ and $ \ell_i \in \mb N $, we define $ \bm \nu_{\Mu,t} = [\nu_{i,t}]_{i \in \mc M} $ as the concatenated vector at time $t$, $ \bm \nu_i^t = [\nu_{i,1}^\top \ldots \nu_{i,t}^\top]^\top $ as the history of the $ i $-th vector up to time $ t $, and $ \bm \nu_\Mu^t = [\bm \nu_i^t]_{i \in \mc M} $ as the history of the concatenated vector up to time $ t $.

\noindent\textbf{Concatenated Vectors:} The linear dynamics of the agent $i$ is written in the compact form $\bm x_i = f_i(\bm x_{\Ni}, \bm u_i, \bm \xi_i)$, where $ \bm x_i := [x_{i,t}]_{t\in \{0\} \cup \mc T} $, $ \bm u_i := [u_{i,t}]_{t \in \mc T} $, $ \bm \xi_i := [\xi_{i,t}]_{t\in \mc T} $ and $ \bm x_{\Ni} := [\bm x_{\Ni,t}]_{t \in \mc T} $.
Here, $\bm x_i$, $\bm u_i$, $\bm \xi_i$ and $\bm x_{\mathcal N_i} $ denote the state, input, exogenous uncertainty, and the state of neighbors affecting agents $i$, respectively (Section~\ref{Section::BasicInfo}). Vector $\bm \zeta_j$ denote belief states of the neighbor agent $j \in \mathcal N_i$, i.e., the dynamic of agent $i$ are affected by its belief $\bm \zeta_j\in \mc X_j$ of what values the states of agent $j$ will take  (Section~\ref{sec::DecCont}). Vector $\bm s_j\in\mc S_j$ is used in the construction of the approximation \eqref{SetApproximation} (Section~\ref{sec::SolMethod}). \\
\noindent\textbf{Optimization variables:} 
The paper analyzes four pairs of problems. A major distinction between problems is the information available to the policies and the present of sets as decision variables. Below we summarize this information. 
\begin{itemize}
	\item \textbf{Centralized information exchange (Section~\ref{centralized}):} Problem~\eqref{Centralized} has optimization variables the state feedback policies  denoted with lower case $\bm \pi_{i}(\bm x_\Mu) :=  [\pi_{i,t}(\bm x_\Mu^t)]_{t\in\mc T}$, and   Problem~\eqref{Centralizedb} has the uncertainty feedback policies   denoted by upper case $\bm \Pi_{i}(\bm \xi_\Mu) := [\Pi_{i,t}(\bm \xi_\Mu^{t-1})]_{t\in\mc T}$.
	
	\item \textbf{Partially nested information exchange (Section~\ref{partiallynested}):} Problem~\eqref{Semi-Centralized} has optimization variables the state feedback policies  denoted with lower case $\bm \phi_{i}(\bm x_\oNi) := [\phi_{i,t}(\bm x_\oNi^t)]_{t\in\mc T}$, and  Problem~\eqref{Semi-Centralizedb} has the uncertainty feedback policies denoted with upper case $\bm \Phi_{i}(\bm \xi_\oNi) := [\Phi_{i,t}(\bm \xi_\oNi^{t-1})]_{t\in\mc T}$.
	
	\item \textbf{Local information exchange (Section~\ref{sec::DecCont}):} Problem~\eqref{Decentralized} has optimization variables the state feedback policies denoted with lower case $\bm \psi_{i}(\bm x_i, \bm \zeta_\Ni) := [\psi_{i,t}(\bm x_i^t, \bm \zeta_\Ni^t)]_{t \in \mc T}$ and the state forecast sets $\mc X_i$. Besides, Problem~\eqref{Decentralizedb} has optimization variables the uncertainty feedback policies denoted with upper case $\bm \Psi_{i}(\bm \xi_i,\bm \zeta_\Ni) := [\Psi_{i,t}(\bm \xi_i^{t-1},\bm \zeta_\Ni^{t})]_{t\in\mc T}$ and the state forecast sets~$\mc X_i$.

	\item \textbf{Approximation of Problem~\eqref{Decentralizedb} (Section~\ref{sec::SolMethod}):} Problem~\eqref{DecentralizedXab_GeneralY} has optimization variables  the uncertainty feedback policies $\bm \Psi_{i}(\bm \xi_i,\bm \zeta_\Ni)$ and the state forecast sets $\mc X_i$ which are parameterized by approximation \eqref{SetApproximation} through matrices $Y_i$ and vector $z_i$ and a given set $\mc S_i$. Problem~\eqref{DecentralizedFinal_GeneralY} has optimization variables  the policy $\bm \Gamma_i(\bm \xi_i, \bm s_{\Ni}) := [\Gamma_{i,t}(\bm \xi_i^{t-1}, \bm s^t_\Ni)]_{t\in\mc T}$ and the state forecast sets $\mc X_i$ which are also  parameterized by approximation~\eqref{SetApproximation} through matrices $Y_i$ and vector $z_i$ and a given set $\mc S_i$.   
\end{itemize}

\end{document}